\documentclass{amsart} 

\usepackage[paperheight=11in, 
    paperwidth=8.5in,
    outer=1.35in,
    inner=1.35in,
    bottom=1.45in,
    top=1.45in]{geometry} 
\usepackage{amsmath, amsxtra, amsthm, amsfonts, amssymb}
\usepackage{mathrsfs}
\usepackage{graphicx}
\usepackage{url}
\usepackage[normalem]{ulem}
\usepackage{color}
\usepackage{tikz-cd}
\usepackage{bbm}

\usepackage[T1]{fontenc}
\usepackage{enumitem}
\usepackage{array,multirow}
\usepackage{arydshln}

 \usepackage[hidelinks,backref=page]{hyperref} 
 \hypersetup{
    colorlinks,
    citecolor=magenta,
    filecolor=magenta,
    linkcolor=blue,
    urlcolor=black
}
\usepackage{cleveref}

\usepackage{tikz}
\usetikzlibrary{decorations.pathmorphing}

\tikzset{snake it/.style={decorate, decoration=snake}}

\tikzset{
  dep u/.style={insert path={-- ++(0,15) node{}}},
  dep r/.style={insert path={-- ++(15,0) node{}}},
  dep d/.style={insert path={-- ++(0,-15) node{}}},
  dep l/.style={insert path={-- ++(-15,0) node{}}},
  recurse lattice path/.code args={#1#2}{
    \ifx#1.\else\tikzset{dep #1,recurse lattice path=#2}\fi
  },
  lattice path/.style={recurse lattice path=#1.}
}
\usetikzlibrary{arrows,calc,intersections,matrix,positioning,through}

\usepackage{tkz-euclide}

\newcommand{\QQ}{\mathbb Q}
\newcommand{\RR}{\mathbb R}
\newcommand{\CC}{\mathbb C}
\newcommand{\PP}{\mathbb P}
\newcommand{\ZZ}{\mathbb Z}
\newcommand{\TT}{\mathbb T}
\renewcommand{\SS}{\mathbb S}

\newcommand{\kk}{\mathbbm k}
\newcommand{\C}{\mathcal C}
\newcommand{\I}{\mathcal I}
\newcommand{\be}{\mathbf e}
\newcommand{\M}{{\boldsymbol M}}
\newcommand{\br}{\mathbf r}

\newcommand{\bmu}{{\boldsymbol \mu}}

\newcommand{\newterm}{\emph}

\newcommand{\centered}[1]{\begin{tabular}{l} #1 \end{tabular}}

\theoremstyle{definition}
\newtheorem{thm}{Theorem}[section]
\newtheorem{cor}[thm]{Corollary}
\newtheorem{lem}[thm]{Lemma}
\newtheorem{prop}[thm]{Proposition}
\newtheorem{defn}[thm]{Definition}

\newtheorem{eg}[thm]{Example}
\newtheorem{rem}[thm]{Remark}
\newtheorem{ques}[thm]{Question}

\newtheorem{mainthm}{Theorem}

\DeclareMathOperator{\distance}{dist}
\DeclareMathOperator{\val}{val}

\DeclareMathOperator{\Dr}{Dr}
\DeclareMathOperator{\Conv}{Conv}
\DeclareMathOperator{\Trop}{Trop}
\DeclareMathOperator{\trop}{trop}
\DeclareMathOperator{\TrFl}{TrFl}
\DeclareMathOperator{\FlDr}{FlDr}
\DeclareMathOperator{\TrGr}{TrGr}
\DeclareMathOperator{\Gr}{Gr}
\DeclareMathOperator{\Gl}{GL}
\DeclareMathOperator{\PPP}{P}
\DeclareMathOperator{\Fl}{Fl}

\DeclareMathOperator{\Perm}{Perm}

\title[Polyhedral and tropical geometry of flag positroids]
{Polyhedral and tropical geometry of flag positroids}
\author{Jonathan Boretsky, Christopher Eur, Lauren Williams}
\date{}

\address{Centre de Recherches Math{\'e}matiques in Montreal}
\email{jonathan.boretsky@mail.mcgill.ca}

\address{Carnegie Mellon University}
\email{ceur@cmu.edu}

\address{Harvard University}
\email{williams@math.harvard.edu}

\begin{document}

\maketitle

\begin{abstract}
	A \emph{flag positroid}  of ranks $\br:=(r_1<\dots <r_k)$ on $[n]$
is a flag matroid that can be realized by a real $r_k \times n$ matrix  $A$
such that the $r_i \times r_i$ minors of $A$
involving rows $1,2,\dots,r_i$ are nonnegative for all $1\leq i \leq k$. In this paper we explore
the polyhedral and tropical geometry of flag positroids, particularly when 
$\br:=(a, a+1,\dots,b)$ is a sequence of consecutive numbers.  In this case we show that 
the nonnegative tropical flag variety 
	$\TrFl_{\br,n}^{\geq 0}$ equals the 
nonnegative flag Dressian $\FlDr_{\br,n}^{\geq 0}$,  and that the points 
 $\bmu  = (\mu_a,\ldots, \mu_b)$ of 
$\TrFl_{\br,n}^{\geq 0} = 
 \FlDr_{\br,n}^{\geq 0}$ give rise to coherent subdivisions of 
the flag positroid polytope $P(\underline{\bmu})$ into flag positroid
polytopes.   Our results have applications to Bruhat interval polytopes:
for example, we show that a complete flag matroid polytope is a Bruhat interval polytope 
if and only if its $(\leq 2)$-dimensional faces are Bruhat interval polytopes.
Our results also have applications to realizability questions.  
	We define
	a \emph{positively oriented flag matroid} to be a sequence
	of positively oriented matroids $(\chi_1,\dots,\chi_k)$ which is also 
	an oriented flag matroid.  We then prove
	that every positively oriented flag matroid of ranks 
	$\br=(a,a+1,\dots,b)$ is realizable.
\end{abstract}

\setcounter{tocdepth}{1}
\tableofcontents

\section{Introduction}

In recent years there has been a great deal of interest in the 
tropical Grassmannian \cite{SS04, HJJS08, HJS14, Cachazo, Bosstrop}, 
and matroid polytopes and their subdivisions \cite{Spe08, AFF, Early}, 
as well as ``positive'' 
\cite{Pos, SW05,  Oh, ARW, LeFraser, LPW, SW21, AHLS}
and ``flag'' 
\cite{TW15, BEZ21,  Bossinger, JL, JLLO, Bor}
versions of the above objects.
The aim of this paper is to illustrate the beautiful relationships between the
nonnegative tropical flag variety, the nonnegative flag Dressian, 
and flag positroid polytopes and their subdivisions, unifying and 
generalizing some of the existing results.  We will particularly 
focus on the case of flag varieties (respectively, flag positroids)
consisting of subspaces (respectively, matroids)
of \emph{consecutive} ranks.  This case 
includes both Grassmannians
and complete flag varieties.

\medskip
For positive integers $n$ and $d$ with $d<n$, we 
let $[n]$ denote the set $\{1, \ldots, n\}$ and we let 
${[n] \choose d}$ denote the collection of all $d$-element subsets of $[n]$.
Given a subset $S\subseteq [n]$ we 
let $\be_S$ denote the sum of standard basis vectors $\sum_{i\in S} \be_i$.  
For a collection $\mathcal B \subset {[n]\choose d}$, we let
\[
P(\mathcal B) = \text{the convex hull of $\{\be_B : B\in \mathcal B\}$ in $\RR^n$}.
\]
The collection $\mathcal B$ is said to define a \newterm{matroid $M$ 
of rank $d$ on $[n]$} if every edge of the polytope
$P(\mathcal B)$
is parallel to $\be_i - \be_j$ for some $i\neq j \in [n]$.
In this case, we call $\mathcal B$ the set of 
\newterm{bases} of $M$, and define the \newterm{matroid polytope} $P(M)$ of $M$ to be the polytope $P(\mathcal B)$.  
When $\mathcal B$ indexes the nonvanishing Pl\"ucker coordinates of 
an element $A$ of the Grassmannian $\Gr_{d,n}(\CC)$, 
we say that $A$ \emph{realizes} $M$, and it is well-known that 
$P(\mathcal B)$ is the moment map image
of the closure of the torus orbit of $A$ in the Grassmannian \cite{GGMS87}.
We assume familiarity with the fundamentals of matroid theory as in  \cite{Oxl11} and
\cite{BGW03}.

\medskip
The above definition of matroid in terms of its 
polytope is due to \cite{GGMS87}.
Flag matroids are natural generalizations of matroids that admit the following polytopal definition.

\begin{defn}\cite[Corollary 1.13.5 and Theorem 1.13.6]{BGW03}
Let $ \br = (r_1, \ldots, r_k)$ be a sequence of increasing integers in $[n]$.
A \newterm{flag matroid} of ranks $\br$ on $[n]$ is a sequence 
$\M = (M_1, \ldots, M_k)$ of matroids of ranks $(r_1, \ldots, r_k)$ on $[n]$ such that
all vertices of the polytope
\[
P(\M) = P(M_1)+ \cdots + P(M_k), \text{ the Minkowski sum of  matroid polytopes},
\]
are equidistant from the origin.
	The polytope $P(\M)$ is called the \newterm{flag matroid polytope} of $\M$; we sometimes
	say it is a flag matroid polytope of \emph{rank $\br$}.
\end{defn}

Flag matroids are exactly the type $A$ objects in the theory of Coxeter matroids \cite{GS87, BGW03}.
Just as a realization of a matroid is a point in a Grassmannian, a realization of a flag matroid 
is a point in a flag variety.  More concretely, a \newterm{realization} of a flag matroid 
of ranks $(r_1,\dots,r_k)$ 
is an $r_k \times n$ matrix $A$ 
over a field such that for each $1 \leq i \leq k$, the $r_i \times n$ submatrix of $A$ formed by the first $r_i$ rows of $A$ is a realization of $M_i$.
For an equivalent definition of flag matroids in terms of Pl\"ucker relations on partial flag varieties, see \cite[Proposition A]{JL}.

There is a notion of moment map for any flag variety (indeed for any generalized partial
flag variety $G/P$) \cite{GS87, BGW03}.  When a flag matroid $\M$ can be realized by 
a point $A$ in the flag variety, then its matroid polytope $P(\M)$
is the moment map image of the closure of the torus orbit of $A$ in the flag variety
\cite{GS87}, \cite[Corollary 1.13.5]{BGW03}.

There are natural ``positive'' analogues of matroids, flag matroids, and their polytopes.
\begin{defn}\label{def:flagpositroid}
Let $\br = (r_1, \cdots ,r_k)$ be a sequence of increasing integers in $[n]$.
	We say that a flag matroid $(M_1,\dots,M_k)$  of ranks $\br$ on $[n]$ is  a \newterm{flag positroid} if it has a realization by a real matrix $A$ such that 
	the $r_i \times n$ submatrix of $A$ formed by the first $r_i$ rows of $A$ has all nonnegative minors
	for each $1 \leq i\leq k$.
\end{defn}

We refer to the flag matroid polytope of a flag positroid as a \emph{flag positroid polytope}.
It follows from our definition above that flag positroids are realizable.

\medskip
Setting $k = 1$ in \Cref{def:flagpositroid} gives the well-studied notion of \newterm{positroids} 
and \emph{positroid polytopes} \cite{Pos, Oh, ARW}.  
Therefore each flag positroid is a sequence of positroids.

In recent years it has been gradually understood that the tropical 
geometry of the Grassmannian and flag variety,
and in particular, the \emph{Dressian} and \emph{flag Dressian}, are
  intimately connected to (flag) matroid 
polytopes and their subdivisions \cite{Spe08, HJJS08, BEZ21}
(see also \cite[\S4]{MS15}).  
A particularly attractive point of view, which sheds light on 
the above connections, 
is the theory of \emph{(flag) matroids over hyperfields}
\cite{BB19, JL}. In this framework, the Dressian 
and flag Dressian are the Grassmannian and flag variety 
over the \emph{tropical hyperfield}, while  matroids
and flag matroids are the points of the Grassmannian and flag 
variety over the \emph{Krasner hyperfield}.

The tropical geometry of the \emph{positive} Grassmannian and flag variety
are particularly nice:
the positive tropical Grassmannian equals the positive Dressian, 
whose cones in turn parameterize 
subdivisions of the hypersimplex into positroid polytopes \cite{SW05, SW21, LPW, AHLS}.  And the positive tropical complete flag variety equals
the positive complete flag Dressian, whose cones
parameterize subdivisions of the permutohedron into 
\emph{Bruhat interval polytopes} \cite{Bor, JLLO}.
\Cref{thm:eqvs} below unifies and generalizes the above results.

\begin{defn}\label{def:tropicalsemifield}
Let $\TT = \RR \cup \{\infty\}$ be the set underlying the tropical hyperfield, endowed with
	the topology such that $-\log: \RR_{\geq 0} \to \TT$ is a homeomorphism.  
Given a point $w \in \TT^{\binom{[n]}{r}}$, 
we  define the \newterm{support} of $w$ 
	to be $\underline w = \{S\in \binom{[n]}{r} : w_S\neq \infty\}$.
	When $\underline w$ is the set of bases of a matroid,
	we identify $\underline w$ with that matroid.
	Let $\PP\left(\TT^{\binom{[n]}{r}}\right)$ be the tropical projective space of $\TT^{\binom{[n]}{r}}$, which is defined as $\big(\TT^{\binom{[n]}{r}}\setminus \{(\infty, \ldots, \infty)\}\big) / \sim$, where $w \sim w'$ if $w = w' + (c, \ldots, c)$ for some $c\in \RR$.
\end{defn}
Our main result is the following.

\begin{mainthm}\label{thm:eqvs}
Suppose $\br$ is a sequence of consecutive integers $(a, \ldots, b)$ for some $1\leq a \leq b \leq n$.
Then, for $\bmu  = (\mu_a,\ldots, \mu_b) \in \prod_{i = a}^{b} \PP\left( \TT^{\binom{[n]}{i}}\right)$, 
	the following statements are equivalent:
\begin{enumerate}[label = (\alph*)]
\item\label{eqvs:TrFl} $\bmu\in \TrFl_{\br,n}^{\geq 0}$, the nonnegative tropicalization of the flag variety, i.e.\ the closure of the coordinate-wise valuation of points in $\Fl_{\br,n}(\C_{\geq 0})$.
\item\label{eqvs:FlDr} $\bmu \in \FlDr_{\br,n}^{\geq 0}$, the nonnegative flag Dressian, i.e. the ``solutions" to the positive-tropical Grassmann-Pl\"ucker and incidence-Pl\"ucker relations.
\item\label{eqvs:subdiv} 
		Every face in the coherent subdivision $\mathcal 
		D_{\boldsymbol\mu}$ of the polytope $P(\underline{\bmu}) = P(\underline{\mu_1}) + \cdots + P(\underline{\mu_k})$ induced by $\bmu$ is a flag positroid polytope (of rank $\br$).

\item\label{eqvs:2faces} 
	Every face of dimension at most 2 in the subdivision $\mathcal D_{\boldsymbol\mu}$ of $P(\underline{\boldsymbol \mu})$ is a flag positroid polytope (of rank $\br$).
\item\label{eqvs:3terms} 
The support $\underline \bmu$ of $\bmu$
	is a flag matroid, 
	$\bmu$ satisfies every three-term positive-tropical incidence relation when $a< b$ (respectively, every three-term positive-tropical Grassmann-Pl\"ucker relation when $a = b$), and either $\underline\bmu$ consists of uniform matroids or $\mu_i\in \Dr_{i;n}^{\geq 0}$ for at least one $a\leq i \leq b$.		
\end{enumerate}
\end{mainthm}

For the definitions of the objects in \Cref{thm:eqvs}, see \Cref{prop:closure} for \ref{eqvs:TrFl}, \Cref{defn:tropFl} for \ref{eqvs:FlDr}, \Cref{defn:subdiv} for \ref{eqvs:subdiv}, and \Cref{defn:3terms} for \ref{eqvs:3terms}.  

\medskip
We note that if $\br=(d)$ is a single integer, 
\Cref{thm:eqvs} describes the
relationship between the nonnegative tropical Grassmannian,
the nonnegative Dressian, and subdivisions of positroid polytopes (e.g. the hypersimplex,
if $\bmu$ has no coordinates equal to $\infty$)
into positroid polytopes.  And when $\br=(1,2,\dots,n)$, \Cref{thm:eqvs} 
describes the 
relationship between the nonnegative tropical complete flag variety, 
the nonnegative complete flag Dressian, and subdivisions of Bruhat interval polytopes (e.g. the permutohedron, if $\bmu$ has no coordinates equal to $\infty$) into 
Bruhat interval polytopes. We illustrate this relationship in the case where $\bmu$ has no coordinates equal to $\infty$ in \Cref{sample}.

\begin{figure}[h]\centering
 \begin{tikzpicture}[scale=1.5, every node/.style={scale=1.5}]
     \tkzDefPoint(-.2,0){e14}
     \tkzDefPoint(1,0){e13}
     \tkzDefPoint(.5,.43){e24}
     \tkzDefPoint(1.65,.43){e23}
     \tkzDefPoint(.7,1.3){e12}
     \tkzDefPoint(.7,-.95){e34}
         \filldraw[fill=gray!30!white](e14)--(e24)--(e23)--(e13)--(e14);
                \filldraw[fill=blue!60!white](e14)--(e13)--(e12);
                \filldraw[fill=blue!60!white](e13)--(e23)--(e12);
                \filldraw[fill=magenta](e13)--(e14)--(e34);
                \filldraw[fill=magenta](e13)--(e23)--(e34);
\tkzDrawPolygon(e24, e14,e13,e23,e24,e12,e14,e13,e12,e23,e34,e14,e13,e34,e23)
                \tkzDrawPolygon(e24,e14,e12,e24,e23,e12,e13,e14,e13,e23);
                \node[shift=(e24), anchor=north] {\tiny $e_{24}$};
                \node[shift=(e14), anchor= east] {\tiny $e_{14}$};
                \node[shift=(e23), anchor=west] {\tiny $e_{23}$};
                \node[shift=(e13), anchor=north east] {\tiny $e_{13}$};
                \node[shift=(e12), anchor=east] {\tiny $e_{12}$};
                \node[shift=(e34), anchor=east] {\tiny $e_{34}$};
        \filldraw[color=black,fill=black] (e24) circle (1pt);
        \filldraw[color=black,fill=black] (e14) circle (1pt);
        \filldraw[color=black,fill=black] (e23) circle (1pt);
        \filldraw[color=black,fill=black] (e13) circle (1pt);
        \filldraw[color=black,fill=black] (e12) circle (1pt);
        \filldraw[color=black,fill=black] (e34) circle (1pt);
 \end{tikzpicture}\hspace{.8cm}
     \begin{tikzpicture}[scale=1.4, every node/.style={scale=1.5}]
     \tkzDefPoint(0,0){Z1}
     \tkzDefPoint(0,-1){newZ1}
     \tkzDefPoint(1.732,0){Z3}
     \tkzDefPoint(1.732,-1){newZ3}
                \tkzDefPoint(.866,.5){Z2}
     \tkzDefPoint(.866,-1.5){Z4}
                \filldraw[fill=blue!60!white](newZ1)--(Z1)--(Z2)--(Z3);
                \filldraw[fill=magenta](newZ1)--(Z3)--(newZ3)--(Z4);
                \tkzDrawPolygon(newZ1, Z1, Z2, Z3, newZ3, Z4, newZ1)
                \node[shift=(Z1), anchor=east] {\tiny $(2,3,1)$};
                \node[shift=(newZ1), anchor=east] {\tiny $(1,3,2)$};
                \node[shift=(Z2), anchor=south] {\tiny $(3,2,1)$};
                \node[shift=(Z3), anchor=west] {\tiny $(3,1,2)$};
                \node[shift=(newZ3), anchor=west] {\tiny $(2,1,3)$};
                \node[shift=(Z4), anchor=north] {\tiny $(1,2,3)$};
        \filldraw[color=black,fill=black] (Z1) circle (1pt);
        \filldraw[color=black,fill=black] (Z2) circle (1pt);
        \filldraw[color=black,fill=black] (Z3) circle (1pt);
        \filldraw[color=black,fill=black] (Z4) circle (1pt);
        \filldraw[color=black,fill=black] (newZ1) circle (1pt);
        \filldraw[color=black,fill=black] (newZ3) circle (1pt);
\end{tikzpicture} 

\caption{At left: the coherent subdivision of the hypersimplex into positroid polytopes 
induced by a point $\bmu \in \Dr^{> 0}_{2,4}$  such that 
   $\mu_{13}+\mu_{24} = \mu_{23}+\mu_{14} < \mu_{12}+\mu_{34}$.
At right: the coherent subdivision of the
permutohedron into Bruhat interval polytopes
induced by a point $\bmu\in \FlDr^{>0}_{(1,2,3),3}$ 
such that $\mu_{2}+\mu_{13} = \mu_{1}+\mu_{23} < \mu_{3}+\mu_{12}$.}
\label{sample}
\end{figure}
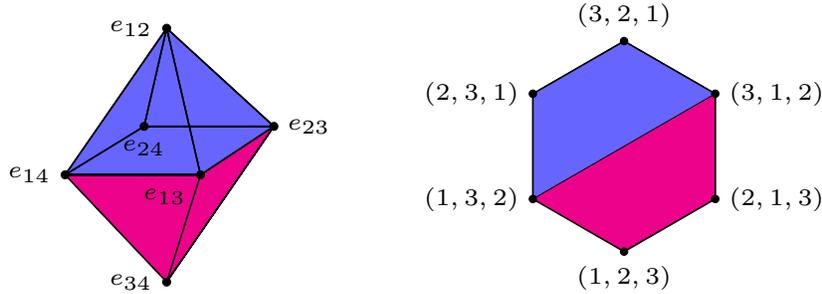

We prove the equivalence \ref{eqvs:TrFl}$\iff$\ref{eqvs:FlDr} in \Cref{pf:1}, the implications \ref{eqvs:FlDr}$\implies$\ref{eqvs:subdiv}$\implies$\ref{eqvs:2faces}$\implies$\ref{eqvs:FlDr} in \Cref{pf:2}, and the equivalence \ref{eqvs:FlDr}$\iff$\ref{eqvs:3terms} in \Cref{pf:3}.

\medskip

\Cref{thm:eqvs} has applications to flag positroid polytopes.

\begin{cor}\label{cor:2faces}
	For a flag matroid $\M=(M_a,M_{a+1},\dots,M_b)$ 
	 of consecutive ranks $\br=(a,a+1,\dots,b)$, 
	its flag matroid polytope $P(\M)$ is a 
 flag positroid polytope
if and only if its $(\leq 2)$-dimensional faces are 
	flag positroid polytopes (of rank $\br$).
\end{cor}

\begin{proof}
	Let $\bmu = (\mu_a,\dots,\mu_b)$, with $\mu_i\in \{0,\infty\}^{{[n]\choose i}}$, where the coordinates
of each $\mu_i$ are either $0$ or $\infty$ based on whether 
we have a basis or nonbasis of $M_i$.
This gives rise to the trivial subdivision of the corresponding flag matroid 
	polytope $P(\underline{\bmu})=P(\M)$.  
The result now follows from the equivalence of \ref{eqvs:subdiv} and \ref{eqvs:2faces} in \Cref{thm:eqvs}.
\end{proof}

In the Grassmannian case, that is, the case that $\br=(d)$ is a single integer, the flag positroid  polytopes  of rank $\br$
are precisely the positroid polytopes, and 
in that case the above corollary appeared as 
\cite[Theorem 3.9]{LPW}.

Also in the Grassmannian case, the objects discussed in \Cref{thm:eqvs} are
closely related to questions of realizability.  
Note that by definition, every positroid has a realization by a 
matrix whose Pl\"ucker coordinates are nonnegative, so 
it naturally defines a \emph{positively oriented matroid}, that is,
an oriented matroid defined by a chirotope whose values are all 
$0$ and $1$.  Conversely, 
every positively oriented matroid  can be realized by a positroid:
this was first proved in \cite{ARW17} using positroid polytopes, and 
subsequently in 
\cite{SW21}, using the positive tropical Grassmannian.
It is natural then to ask if there is an analogous realizability statement in the setting of flag matroids,
and if one can characterize when a sequence of positroids forms a flag positroid;
indeed, this was 
part of the motivation for \cite{BCT}, which studied
quotients of uniform positroids.
Note however that questions of realizability for flag matroids are rather subtle:
for example, a sequence of positroids that form a realizable flag matroid can still fail to be a flag positroid
(see \Cref{eg:notreal}).
By working with \emph{oriented} flag matroids, we 
give an 
answer to this realizability question 
in \Cref{cor:real}, in the case of consecutive ranks.

\begin{cor}\label{cor:real}
Suppose $(M_1, \ldots, M_k)$ is a sequence of positroids on $[n]$ of consecutive ranks
$\br = (r_1,\dots,r_k)$.
Then, when considered as a sequence of positively oriented matroids, $(M_1, \ldots, M_k)$ is a flag positroid if and only if it is an oriented flag matroid.
\end{cor}
	
We define
a \emph{positively oriented flag matroid} to be a sequence
of positively oriented matroids $(\chi_1,\dots,\chi_k)$ which is also 
an oriented flag matroid.  \Cref{cor:real} then says that every positively oriented flag matroid 
of consecutive ranks $(r_1,\dots,r_k)$ is realizable.

See \Cref{sec:posorient} for a review of oriented matroids and oriented flag matroids.
Note that because a positroid by definition has a realization over $\RR$ with all nonnegative minors, it defines a positively oriented matroid.
In \Cref{pf:cor}, we deduce \Cref{cor:real} from the equivalence of \ref{eqvs:TrFl} and \ref{eqvs:FlDr} in \Cref{thm:eqvs}.
Another proof using ideas from discrete convex analysis is sketched in \Cref{rem:altproof}.
In both proofs, the consecutive ranks condition is indispensable.  We do not know whether the corollary holds 
if $\br=(r_1,\dots,r_k)$ fails to satisfy the consecutive rank condition.

\begin{ques}\label{ques}
Suppose $M$ and $M'$ are positroids on $[n]$ such that, when considered as positively oriented matroids, they form an oriented flag matroid $(M,M')$.  Is $(M,M')$ then a flag positroid?
\end{ques}

One may attempt to answer the question by appealing to the fact \cite[Exercise 8.14]{Kun86} that for a flag matroid $(M, M')$, one can always find a flag matroid $(M_1,\ldots, M_k)$ of consecutive ranks such that $M_1 = M$ and $M_k = M'$.
However, the analogous statement fails for flag positroids:
See \Cref{eg:cantcomplete} for an example of a flag positroid $(M,M')$ on $[4]$ of ranks $(1,3)$ such that there is no flag positroid $(M, M_2, M')$ with rank of $M_2$ equal to 2.

The consecutive rank condition has recently shown up in 
 \cite{BK}, which studied the relation between two notions of total 
positivity for partial flag varieties,
``Lusztig positivity'' and 
``Pl\"ucker positivity'' (see \Cref{def:backgroundflag}).  In particular, 
the Pl\"ucker positive subset of a partial flag variety agrees with 
 the Lusztig positive subset of the partial flag variety precisely when the 
 flag variety consists  of linear subspaces of consecutive ranks \cite[Theorem 1.1]{BK}.

\medskip
A \emph{generalized Bruhat interval polytope} \cite[Definition 7.8 and Lemma 7.9]{TW15}
can be defined as the moment map image of the closure of the torus orbit of a point
$A$ in the nonnegative part $(G/P)^{\geq 0}$ (in the sense of Lusztig) of a flag variety $G/P$.
When $\br$ is a sequence of consecutive integers, it then follows from 
\cite{BK} that generalized Bruhat interval polytopes for $\Fl_{\br;n}^{\geq 0}$
are precisely the flag positroid polytopes of ranks $\br$.
In the complete flag case, a generalized Bruhat interval polytope is 
just a \emph{Bruhat interval polytope} \cite{KW15}, that is, the 
convex hull of the permutation vectors $(z(1),\dots,z(n))$ for all permutations $z$ lying
in some Bruhat interval $[u,v]$.

We can now restate \Cref{cor:2faces} as follows.

\begin{cor}
For a flag matroid on $[n]$ of consecutive ranks $\br$, 
its flag matroid polytope is a 
generalized Bruhat interval polytope
if and only if its $(\leq 2)$-dimensional faces are generalized Bruhat interval polytopes.
	In particular, for a complete flag matroid on $[n]$,
	its flag matroid polytope is a Bruhat interval polytope if and 
 only if its $(\leq 2)$-dimensional faces are Bruhat interval polytopes.
\end{cor}

The structure of this paper is as follows.
In \Cref{sec:background}, we give background on total positivity
and Bruhat interval polytopes.  In \Cref{sec:trop},
we introduce the tropical flag variety, the flag Dressian, and 
nonnegative analogues of these objects; we also prove
the equivalence of \ref{eqvs:TrFl} and \ref{eqvs:FlDr} in 
\Cref{thm:eqvs}. 
In \Cref{sec:POM} we discuss positively oriented flag matroids
and prove 
\Cref{cor:real}.  In \Cref{sec:subdivision} we explain the relation
between the flag Dressian and subdivisions of flag matroid polytopes,
then prove that 
\ref{eqvs:FlDr}$\implies$\ref{eqvs:subdiv}$\implies$\ref{eqvs:2faces}$\implies$\ref{eqvs:FlDr} in \Cref{thm:eqvs}.
We prove some key results about three-term incidence and 
Grassmann-Pl\"ucker relations in \Cref{sec:3}, which allow us to 
prove \ref{eqvs:FlDr}$\iff$\ref{eqvs:3terms} in \Cref{thm:eqvs}.
\Cref{sec:projection} concerns projections of positive
Richardsons to positroids: we characterize
 the positroid constituents of complete flag positroids,
and we characterize when two adjacent-rank positroids form
an oriented matroid quotient, or equivalently, can appear
as constituents of a complete flag positroid.
In \Cref{sec:BIP}, we make some remarks about the various
fan structures for $\TrFl_{\br;n}^{>0}$; we then
 discuss 
fan structures and coherent subdivisions in the case of the Grassmannian
and complete flag variety, 
including a detailed look at the case of 
$\TrFl_4^{>0}$.

\subsection*{Acknowledgements} The first author is supported by the Natural Sciences and Engineering Research Council of Canada (NSERC). Le premier auteur a été financé par le Conseil de recherches en sciences naturelles et en génie du Canada (CRSNG), [Ref. no. 557353-2021]. The second author is partially supported by the US National Science Foundation (DMS-2001854).  The third author is partially supported by 
the National Science Foundation (DMS-1854512 and DMS-2152991).
We  are grateful to Tony Bloch, Michael Joswig, Steven Karp, Georg Loho, Dante Luber, and 
Jorge Alberto Olarte for sharing their work with us, which partially inspired this project. We also thank Yue Ren, Vasu Tewari, and Jorge Olarte for useful comments. 
We are grateful to Lara Bossinger for several invaluable
discussions about fan structures.
Lastly, we thank Jidong Wang for pointing out an error in a previous version of this paper.

\section{Background on total positivity and Bruhat interval polytopes}\label{sec:background}

\subsection{Background on total positivity}\label{def:backgroundflag}
Let $n\in \ZZ_+$ and let $\br = \{r_1 < \dots < r_k\} \subseteq [n].$ 
For a field $\kk$, let  $G=\Gl_n(\kk)$, and let 
$\PPP_{\br; n}(\kk)$ denote the parabolic subgroup of 
$G$ of block upper-triangular
matrices with diagonal blocks of sizes $r_1, r_2-r_1,\dots, r_k - r_{k-1}, n-r_k$.
We define the \emph{partial flag variety} 
$$\Fl_{\br; n}(\kk):=\Gl_n(\kk)/\PPP_{\br;n}(\kk).$$  
As usual, we identify $\Fl_{\br; n}(\kk)$ with the variety of partial flags of subspaces in $\kk^n$:
\[
\operatorname{Fl}_{\br;n}(\kk) = \{(V_1 \subset \cdots \subset V_k) : \text{$V_i$ a linear subspace of $\kk^n$ of dimension $r_i$ for $i = 1, \ldots, k$}\}.
\]
We write $\operatorname{Fl}_n(\kk)$ for the complete flag variety $\operatorname{Fl}_{1,2,\ldots, n;n}(\kk)$. Note that 
$\Fl_n(\kk)$ can be identified with $\Gl_n(\kk)/B(\kk)$, where
$B(\kk)$ is the subgroup of upper-triangular matrices.  
There is a natural projection $\pi$ from $\Fl_n(\kk)$ to 
any partial flag variety by simply forgetting some of the 
subspaces.

\medskip
If $A$ is an $r_k \times n$ matrix such that $V_{r_i}$ is the span of the first $r_i$ rows,
we say that $A$ is a \emph{realization} of $V:=(V_1 \subset \dots \subset V_k) \in \Fl_{\br;n}$.
Given any realization $A$ of $V$ and any $1\leq i \leq k$, we have the 
\emph{Pl\"ucker coordinates} or \emph{flag minors}
$p_I(A)$ where $I\in {[n] \choose r_i}$; concretely, $p_I(A)$ is the determinant of the 
submatrix of $A$ occupying the first $r_i$ rows and columns $I$.
This gives the Pl\"ucker embedding of $\Fl_{\br;n}(\kk)$ into $\PP^{{[n]\choose r_1}-1} 
\times \dots \times \PP^{{[n]\choose r_k}-1}$ taking 
$V$ to $\left( \big(p_I(A)\big)_{I\in {[n]\choose r_1}}, \dots, \big(p_I(A)\big)_{I\in {[n]\choose r_k}} \right)$.

We now let $\kk$ be the field $\RR$ of real numbers.  With this understanding,
we will often drop the $\RR$ from our notation.
\begin{defn}
	We say that a real matrix is \emph{totally positive} if all of its minors are positive.
	We let $\Gl_n^{>0}$ denote the subset of $\Gl_n$ of totally positive matrices.
\end{defn}

There are two natural ways to define positivity for partial flag varieties.
The first notion comes from work of Lusztig \cite{lusztig}.  The second notion
uses Pl\"ucker coordinates, and was initiated in work of Postnikov \cite{Pos}.

\begin{defn}\label{def:2positive}
	We define the \emph{(Lusztig)  positive part} of $\Fl_{\br; n}$,
	denoted by $\Fl_{\br;n}^{>0}$, as the image of $\Gl_n^{>0}$ inside 
	$\Fl_{\br;n} = \Gl_n/ \PPP_{\br; n}$. We define the 
	\emph{(Lusztig)  nonnegative part} of $\Fl_{\br; n}$, denoted 
	 by $\Fl_{\br;n}^{\geq 0}$, as the closure of 
	  $\Fl_{\br;n}^{>0}$ in the Euclidean topology.  

	We define the \emph{Pl\"ucker positive part} (respectively,
	\emph{Pl\"ucker nonnegative part}) of 
	$\Fl_{\br; n}$ to be the subset of 
	$\Fl_{\br; n}$ where all Pl\"ucker coordinates are 
	 positive (respectively, nonnegative).\footnote{The reader who is concerned
	 about the fact that we are working with projective
	 coordinates can replace  ``all Pl\"ucker coordinates are positive'' by 
	  ``all Pl\"ucker coordinates are nonzero and have the same sign''.}
\end{defn}

It is well-known that 
the Lusztig positive part of 
$\Fl_{\br;n}$ is a subset of the Pl\"ucker positive part of 
$\Fl_{\br;n}$, and that the two notions agree in the case of the Grassmannian
\cite[Corollary 1.2]{TW13}.  
The two notions also agree in the case of the complete flag variety \cite[Theorem 5.21]{Bor}. 
More generally, we have the following.

\begin{thm}\cite[Theorem 1.1]{BK}\label{thm:BK}
	The Lusztig positive (respectively, Lusztig nonnegative) part of 
	$\Fl_{\br;n}$ equals the Pl\"ucker positive (respectively, Pl\"ucker
	nonnegative) part of 
$\Fl_{\br;n}$  if and only if 
the set $\br$ consists of consecutive integers.
\end{thm}
See \cite[Section 1.4]{BK} for more references and 
a nice discussion of the history. Since
in this paper we will be mainly studying the case where $\br$ consists of consecutive integers,
we will use the two notions interchangeably when there is no ambiguity.

Let $B$ and $B^-$ be the opposite Borel subgroups consisting of upper-triangular
and lower-triangular matrices.  Let $W=S_n$ be the Weyl group of 
$\Gl_n$.  Given $u, v\in W$, the \emph{Richardson variety}
is the intersection of opposite Bruhat cells
$$\mathcal{R}_{u,v}:=(B\dot v B/B) \cap (B^- \dot u B/B),$$
where $\dot v$ and $\dot u$ denote permutation
matrices in $\Gl_n$ representing  $v$ and $u$.
It is well-known that $\mathcal{R}_{u,v}$ is nonempty precisely when 
$u \leq v$ in Bruhat order, and in that case is irreducible of 
dimension $\ell(v)-\ell(u)$.

For $u,v\in W$ with $u \leq v$, let 
$\mathcal{R}_{u,v}^{>0}:=\mathcal{R}_{u,v} \cap \Fl_n^{\geq 0}$ be the 
positive part of the Richardson variety.
Lusztig conjectured and Rietsch proved \cite{rietsch} that 
\begin{equation}\label{eq:Rietsch}
\Fl_n^{\geq 0}  = \bigsqcup_{u \leq v} \mathcal{R}_{u,v}^{>0}
\end{equation}
is a cell decomposition of $\Fl_n^{\geq 0}$.
Moreover, Rietsch showed that one obtains a cell decomposition of 
the nonnegative partial flag variety $\Fl_{\br; n}^{\geq 0}$ by projecting
the cell decomposition of $\Fl_n^{\geq 0}$
\cite{rietsch}, \cite[Section 6]{Rie06}.
  Specifically,
if we let $W_{\br}$ be the parabolic subgroup of $W$ generated by 
the simple reflections
$\{s_i \ \vert \ 1 \leq i \leq n-1 \text{ and } i\notin \{r_1,\dots,r_k\}\}$,
then one obtains a cell decomposition by using the projections 
$\pi(\mathcal{R}_{u,v}^{> 0})$  of 
the cells $\mathcal{R}_{u,v}^{> 0}$ 
where $u\leq v$ and $v$ is a minimal-length coset representative
of $W/W_{\br}$.
(We note moreover that Rietsch's results hold for $G$ a semisimple,
simply connected linear algebraic group over $\CC$ split over $\RR$).

In the case of the Grassmannian,  Postnikov 
studied the Pl\"ucker nonnegative part $\Gr_{d,n}^{\geq 0}$ of the Grassmannian,
and gave a decomposition of it into \emph{positroid cells}
$S_{\mathcal{B}}^{>0}$ by intersecting $\Gr_{d,n}^{\geq 0}$ with the 
matroid strata \cite{Pos}.  Concretely, if $\mathcal{B}$ is the collection of bases
of an element of $\Gr_{d,n}^{\geq 0}$, then 
$S_{\mathcal{B}}^{>0} = \{A \in \Gr_{d,n}^{\geq 0} \ \vert \ p_I(A) \neq 0
\text{ if and only if }I\in \mathcal{B}\}$.
 This cell decomposition of 
 $\Gr_{d,n}^{\geq 0}$ agrees with Rietsh's cell decomposition
\cite[Corollary 1.2]{TW13}.

\subsection{Background on (generalized) Bruhat interval polytopes}

Bruhat interval polytopes were defined in \cite{KW15}, motivated by 
the connections to the full Kostant-Toda hierarchy.
\begin{defn}[\cite{KW15}] Given two permutations
$u$ and $v$ in $S_n$ with $u\leq v$ in Bruhat order, the
\emph{Bruhat interval polytope}
$P_{u,v}$ is defined as
	\begin{equation}\label{eq:original}
 P_{u,v} = \Conv\{(x(1),x(2),\dots,x(n)) \ \vert \ u \leq x \leq v\} \subset \RR^n.
\end{equation}
We also define the \emph{(twisted) Bruhat interval polytope}
$\tilde{P}_{u,v}$ by 
\begin{equation}\label{eq:twisted}
	\tilde{P}_{u,v} = \Conv\{(n+1-x^{-1}(1),n+1-x^{-1}(2),\dots,n+1-x^{-1}(n)) \ \vert \ u \leq x \leq v\} \subset \RR^n.
\end{equation}
\end{defn}
While the definition of Bruhat interval polytope in \eqref{eq:original} is more natural from a combinatorial
point of view, as we'll see shortly, the definition  in \eqref{eq:twisted} is more natural from the 
point of view of the moment map.  Note that the set of Bruhat interval polytopes is the 
same as the set of twisted Bruhat interval polytopes; it is just a difference in labeling.

\begin{rem} 
	If we choose any point $A$ in the cell $\mathcal{R}_{u,v}^{>0} \subset 
\Fl_n^{\geq 0}$ (thought of as an $n\times n$ matrix),
and let $M_i$ be the matroid represented by the first  $i$ rows of $A$, 
then $\tilde{P}_{u,v}$ is the Minkowski sum of the matroid polytopes
$P(M_1),\dots, P(M_{n})$ \cite[Corollary 6.11]{KW15}. In particular, $\tilde{P}_{u,v}$ is the matroid polytope of the flag matroid $M_1,\ldots, M_{n}$.
 \end{rem}

Following \cite{TW15}, we can generalize the notion of Bruhat interval polytope as follows
(see \cite[Section 7.2]{TW15} for notation). 
\begin{defn}\label{def:genBIP} 
Choose a generalized partial flag variety $G/P=G/P_J$, let $W_J$ be the 
associated parabolic subgroup of the Weyl group $W$, and let 
$u,v\in W$ with 
 $u\leq v$ in Bruhat order and 
	$v$ a minimal-length coset representative of $W/W_J$.
	Let $\pi$ denote the projection from $G/B$ to $G/P$, and 
let $A$ be an element of the cell 
$\pi(\mathcal{R}_{u,v}^{>0})$ 
	of (Lusztig's definition of) $(G/P)^{\geq 0}$.

	A \emph{generalized Bruhat interval polytope} 
	$\tilde{P}_{u,v}^J$ 
	can be defined
in any of the following equivalent ways
\cite[Definition 7.8, Lemma 7.9, Proposition 7.10, Remark 7.11]{TW15}
	and \cite[Preface]{BGW03}:
	\begin{itemize}
          \item the moment map image of the closure of the torus orbit of 
		  $A$ in $G/P$ (which is a Coxeter matroid polytope)
		\item the moment map image of the closure 
		of the cell 
	$\pi(\mathcal{R}_{u,v}^{>0})$ 
\item the moment map image of the closure of the 
	projected Richardson variety 
	$\pi(\mathcal{R}_{u,v})$ 
\item the convex hull 
	$\Conv\{z \cdot \rho_J \ \vert \ u \leq z \leq v\} \subset
			\mathfrak{t}_{\RR}^*,$\\
			where $\rho_J$ is the sum of 
			fundamental weights $\sum_{j\in J} \omega_j$,
			and $\mathfrak{t}_{\RR}^*$ is the dual
			of the real part of the Lie algebra
			$\mathfrak{t}$ of the torus $T \subset G$.
	\end{itemize}
\end{defn}

\begin{rem}\label{rem:BIPsum}
When $G=\Gl_n$ with fundamental weights $\be_1, \be_1+\be_2, \ldots, \be_1+\cdots +\be_{n-1}$, each generalized Bruhat interval polytope 
	$\tilde{P}_{u,v}^J$ 
	is the flag positroid polytope
associated to a matrix $A$ representing a point of $ \Fl_{\br;n}^{\geq 0}$, with 
$\br=(r_1,\dots,r_k)$.  In this case 
the generalized Bruhat interval polytope  is precisely the Minkowski sum $P(M_1)+\dots +P(M_k)$ of the 
		matroid polytopes $P(M_i)$,
		where $M_i$ is the matroid realized by the first $r_i$ rows of $A$.
	In particular,  
	the generalized Bruhat interval 
	polytope $\tilde{P}_{u,v}^{\emptyset}=
	\tilde{P}_{u,v}$
	is the Minkowski sum
	 $P(M_1)+\dots+P(M_{n})$, where 
	$M_i$ is the positroid realized by the first $i$ rows 
	of any matrix representing a point of 
	 $A\in \mathcal{R}_{v,w}^{>0}$.  We will discuss 
	 how to read off the matroids $M_i$ from $(u,v)$ in 
	 \Cref{sec:projecting}.
\end{rem}

As mentioned in the introduction, 
when $\br$ is a sequence of consecutive ranks, the generalized Bruhat interval
polytopes for $\Fl_{\br;n}^{\geq 0}$ are precisely the flag positroid 
	polytopes
 of ranks $\br$.
	  When $\br=(1,2,\dots,n)$, we recover the notion
	of Bruhat interval polytope, and when $\br$ is a single integer,
	we recover the notion of positroid polytope.

\section{The nonnegative tropicalization}\label{sec:trop}

\subsection{Background on tropical geometry}
We define the main objects in \ref{eqvs:TrFl} and \ref{eqvs:FlDr} of \Cref{thm:eqvs}, and record some basic properties.
For a more comprehensive treatment of tropicalizations and positive-tropicalizations, we refer to \cite[Ch.~6]{MS15} and \cite{SW05}, respectively.

\medskip
For a point $w = (w_1, \ldots, w_m) \in \TT^m$, we write $\overline w$ for its image in the tropical projective space $\PP(\TT^m)$.  For $a = (a_1, \ldots, a_n) \in \ZZ^m$, write $a \bullet w = a_1 w_1 + \cdots + a_m w_m$.

\begin{defn}\label{defn:trophyper}
For a real homogeneous polynomial
\[
f = \sum_{a \in \mathcal A} c_{a} x^{a} \in \RR[x_1, \ldots, x_m], \quad\text{where $\mathcal A$ is a finite subset of $\ZZ_{\geq 0}^m$ and $0\neq c_{a}\in \RR$,}
\]
the \newterm{extended tropical hypersurface} $V_{\trop}(f)$ and the \newterm{nonnegative tropical hypersurface} $V_{\trop}^{\geq 0}(f)$ are subsets of the tropical projective space $\PP(\TT^m)$ defined by
\begin{align*}
V_{\trop}(f)  &= \left\{ \overline w \in \PP(\TT^m) \  \middle| \ \text{the minimum in 
	$\min_{a\in \mathcal A}(a\bullet w)$, if finite, is achieved at least twice}\right\},\\
\text{and}\\
V_{\trop}^{\geq 0}(f)  &= \left\{ \overline w \in \PP(\TT^m) \  \middle| \ \begin{matrix} \text{the minimum in $\displaystyle\min_{a\in \mathcal A}(a\bullet w)$, if finite, is achieved at least twice,}\\ \text{ 
including at some $a, a'\in \mathcal A$ such that $c_a$ and $c_{a'}$ have opposite signs} \end{matrix}\right\}.
\end{align*}
We say that a point \newterm{satisfies the tropical relation} of $f$ if it is in $V_{\trop}(f)$, and that it \newterm{satisfies the positive-tropical relation} of $f$ if it is in $V_{\trop}^{\geq 0}(f)$.
\end{defn}

When $f$ is a multihomogeneous real polynomial, we define $V_{\trop}(f)$ and $V_{\trop}^{\geq 0}(f)$ similarly as subsets of a product of tropical projective spaces.
We will consider tropical hypersurfaces of polynomials that define the Pl\"ucker embedding of a partial flag variety.

\begin{defn}
For integers $0 < r\leq s< n$, the \newterm{(single-exchange) Pl\"ucker relations of type $(r,s;n)$} are polynomials in variables $\{x_I : I \in \binom{[n]}{r} \cup \binom{[n]}{s}\}$ defined as
\[
\mathscr {P}_{r,s;n}=\left\{\sum_{j\in J\setminus I} \operatorname{sign}(j,I,J) x_{I\cup j}x_{J\setminus j }\left| I \in {\binom{[n]}{r-1}},\: J\in {\binom{[n]}{s+1}} \right.\right\},
\]
where $\operatorname{sign}(j,I,J)=(-1)^{|\{k\in J|k<j\}|+|\{i\in I|j<i\}|}$.
When $r = s$, the elements of $\mathscr P_{r,r;n}$ are called the \newterm{Grassmann-Pl\"ucker relations} (of type $(r;n)$), and when $r< s$, the elements of $\mathscr P_{r,s;n}$ are called the \newterm{incidence-Pl\"ucker relations} (of type $(r,s;n)$).
\end{defn}

As in the introduction, let 
$\br = (r_1 < \cdots < r_k)$ be a sequence of increasing integers in $[n]$.
We let
$
\mathscr P_{\br;n} = \bigcup_{\substack{r\leq s \\ r,s\in \br}} \mathscr P_{r,s;n},
$
and let $\langle \mathscr P_{\br;n}\rangle$ be the ideal generated by the elements of $\mathscr P_{\br;n}$.
It is well-known that for any field $\kk$ the ideal 
$\langle \mathscr P_{\br;n}\rangle$ 
set-theoretically carves out the partial flag variety $\operatorname{Fl}_{\br;n}(\kk)$ embedded in $\prod_{i = 1}^k \PP\left(\kk^{\binom{[n]}{r_i}}\right)$ via the standard Pl\"ucker embedding \cite[\S9]{Ful97}.  Similarly, the Pl\"ucker relations define the tropical analogues of partial flag varieties as follows.

\begin{defn}\label{defn:tropFl}
The \newterm{tropicalization $\TrFl_{\br;n}$} of $\operatorname{Fl}_{\br;n}$, the \newterm{nonnegative tropicalization $\TrFl_{\br;n}^{\geq 0}$} of $\operatorname{Fl}_{\br;n}$, the \newterm{flag Dressian $\FlDr_{\br;n}$}, and the \newterm{nonnegative flag Dressian} $\FlDr_{\br;n}^{\geq 0}$ are subsets of $\prod_{i = 1}^k \PP\left(\TT^{\binom{[n]}{r_i}}\right)$ defined as
\begin{align*}
\TrFl_{\br;n} & = \bigcap_{f\in \langle \mathscr P_{\br;n}\rangle} V_{\trop}(f)
\quad	\text{ and } \quad
	\TrFl_{\br;n}^{\geq 0}  = \bigcap_{f\in \langle \mathscr P_{\br;n} \rangle} V_{\trop}^{\geq 0}(f), \\
\FlDr_{\br;n} &=\bigcap_{f\in \mathscr P_{\br;n}} V_{\trop}(f) \quad \text{ and }\quad
\FlDr_{\br;n}^{\geq 0} = \bigcap_{f\in \mathscr P_{\br;n}} V_{\trop}^{\geq 0}(f).
\end{align*}
\end{defn}

When $k=1$, i.e.\ when $\br$ consists of one integer $d$, one obtains the (nonnegative) tropicalization of the Grassmannian $\operatorname{TrGr}_{d;n}^{(\geq 0)}$ and the (nonnegative) Dressian $\operatorname{Dr}_{d;n}^{(\geq 0)}$ studied in \cite{SS04, SW05, SW21, AHLS}.  Like $\operatorname{Fl}_n$, we  write only $n$ in the subscript when $\br = (1, 2,\ldots, n)$.

\begin{rem}
In \cite[\S6]{JLLO}, the authors define the ``positive flag Dressian'' to consist of the elements $\bmu = (\mu_1, \ldots, \mu_k) \in \FlDr_{\br;n}$ whose constituents $\mu_i$ are each in the strictly positive Dressian.  
	In our language, this is equal to considering the points of
\[
\bigcap_{f\in \bigcup_{i = 1}^k \mathscr P_{r_i,r_i;n}} V_{\trop}^{\geq 0}(f) \cap \bigcap_{f\in \bigcup_{r_i < r_j} \mathscr P_{r_i,r_j;n}} V_{\trop}(f)
\]
that have no $\infty$ coordinates.
In a similar vein, we could consider defining the ``nonnegative flag Dressian'' 
to be the elements of the flag Dressian whose constituents are in the 
nonnegative Dressian.  This gives a strictly 
larger set than our definition of the nonnegative flag Dressian,
and has the shortcoming that 
the equivalence of \ref{eqvs:TrFl} and \ref{eqvs:FlDr} in \Cref{thm:eqvs} would no longer hold.  See \Cref{eg:notreal}.
\end{rem}

We record a useful equivalent description of the (nonnegative) tropicalization of a partial flag variety 
using Puiseux series. Recall the notion of the tropical semifield from \Cref{def:tropicalsemifield}.
\begin{defn}\label{def:Puiseux}
Let $\C = \CC\{\{t\}\}$ be the field of Puiseux series with coefficients in $\CC$, with the usual valuation map $\operatorname{val}: \C \to \TT$.
	Concretely, for $f\neq 0$, $\val(f)$ is the exponent of the initial term of $f$,
	and $\val(0) = \infty$.
Let 
\[
	\C_{> 0} = \{f \in \C\setminus \{0\} : \text{the initial coefficient of $f$ is real and positive}\} \ 
	\text{ and } \ \C_{\geq 0} = \C_{>0} \cup \{0\}.
\]
\end{defn}

For a point $p \in \operatorname{Fl}_{\br;n}(\C) \subseteq \prod_{i=1}^k \PP\left(\C^{\binom{[n]}{r_i}}\right)$, applying the valuation $\operatorname{val}: \C \to \TT$ coordinate-wise to the Pl\"ucker coordinates gives a point $\operatorname{val}(p) \in \prod_{i=1}^k \PP\left(\TT^{\binom{[n]}{r_i}}\right)$.
Noting that $\operatorname{val}(\C) = \QQ\cup \{\infty\} \subset \TT$,
we say that a point in $\prod_{i=1}^k \PP\left(\TT^{\binom{[n]}{r_i}}\right)$ has \newterm{rational coordinates} if it is a point in  $\prod_{i=1}^k \PP\left(\left(\QQ\cup \{\infty\}\right)^{\binom{[n]}{r_i}}\right)$.
Let $\operatorname{Fl}_{\br;n}(\C_{\geq 0})$ be the subset of $\operatorname{Fl}_{\br;n}(\C)$ consisting of points with all coordinates in $\C_{\geq 0}$, i.e.\ the points $p \in \operatorname{Fl}_{\br;n}(\C) \subseteq \prod_{i=1}^k \PP\left(\C^{\binom{[n]}{r_i}}\right)$ that have a representative in $\prod_{i=1}^k \C_{\geq 0}^{\binom{[n]}{r_i}}$.

\begin{prop}\label{prop:closure}
The set $\{\operatorname{val}(p) : p \in \operatorname{Fl}_{\br;n}(\C)\}$ equals the set of points in $\TrFl_{\br;n}$ with rational coordinates.  Likewise, the set $\{\operatorname{val}(p) : p \in \operatorname{Fl}_{\br;n}(\C_{\geq 0})\}$ equals  the set of points in $\TrFl_{\br;n}^{\geq 0}$ with rational coordinates.  Moreover, we have
\begin{align*}
\TrFl_{\br;n} &= \text{the closure of } \{\operatorname{val}(p) : p \in \operatorname{Fl}_{\br;n}(\C)\} \text{ in }   \textstyle\prod_{i=1}^k \PP\left(\TT^{\binom{[n]}{r_i}}\right)\quad\text{and}\\
\TrFl_{\br;n}^{\geq 0} &= \text{the closure of } \{\operatorname{val}(p) : p \in \operatorname{Fl}_{\br;n}(\C_{\geq 0})\} \text{ in }   \textstyle\prod_{i=1}^k \PP\left(\TT^{\binom{[n]}{r_i}}\right).
\end{align*}
\end{prop}

\begin{proof}
	The first equality is known as the (extended) \emph{Fundamental Theorem of tropical geometry} \cite[Theorem 3.2.3 \& Theorem 6.2.15]{MS15}.  The second equality is the analogue for nonnegative tropicalizations, established in \cite[Proposition 2.2]{SW05}.
\end{proof}

\begin{rem}
	The need to restrict to rational coordinates and the need to take the closure in \Cref{prop:closure} can be removed if we let $\C$ be the Mal'cev-Neumann ring $\CC((\RR))$ (see \cite[\S3]{Poo93}) which satisfies $\operatorname{val}(\C) = \TT$. See also \cite{Markwig:Puiseux}.
\end{rem}

Let us also record an equivalent description of the (nonnegative) flag Dressian when $\br$ is a sequence of consecutive integers.  We need the following definition.  As is customary in matroid theory, we write $Sij$ for the union $S\cup \{i,j\}$ of subsets $S$ and $\{i,j\}$ of $[n]$.

\begin{defn}\label{defn:3terms}
The set $\mathscr P_{r,r;n}^{(3)}$ of \emph{three-term Grassmann-Pl\"ucker relations} (of type $(r;n)$) is the subset of $\mathscr P_{r,r;n}$ consisting of polynomials of the form
\[
x_{Sij}x_{Sk\ell} - x_{Sik}x_{Sj\ell} + x_{Si\ell}x_{Sjk}
\]
for a subset $S\subseteq [n]$ of cardinality $r-2$ and a subset $\{i < j < k < \ell\} \subseteq [n]$ disjoint from $S$.
Similarly, the set $\mathscr P_{r,r+1;n}^{(3)}$ of \emph{three-term incidence-Pl\"ucker relations} (of type $(r,r+1)$) is the subset of $\mathscr P_{r,r;n}$ consisting of polynomials of the form
\[
x_{Si}x_{Sjk} -  x_{Sj} x_{Sik}+ x_{Sk} x_{Sij}
\]
for a subset $S\subseteq [n]$ of cardinality $r-1$ and a subset $\{i < j < k\} \subseteq [n]$ disjoint from $S$.
\end{defn}

Let $\mathscr P_{\br;n}^{(3)}$ be the union of the three-term
Grassmann-Pl\"ucker and three-term incidence-Pl\"ucker relations, which we refer to as the \emph{three-term 
Pl\"ucker relations}.

\begin{prop}\label{prop:3terms}
	Suppose $\br=(r_1<\dots<r_k)$ consists of consecutive integers.  Then a point $\bmu = (\mu_1, \ldots, \mu_k) \in \prod_{i = 1}^k \PP\left(\TT^{\binom{[n]}{r_i}}\right)$ is in the (nonnegative) flag Dressian if and only if its support $\underline\bmu = (\underline \mu_1, \ldots, \underline \mu_k)$ is a flag matroid and $\bmu$ satisfies the (nonnegative-)tropical three-term Pl\"ucker relations.  More explicitly, we have
\begin{align*}
\FlDr_{\br;n} &= \left\{\bmu \in  \textstyle\prod_{i = 1}^k \PP\left(\TT^{\binom{[n]}{r_i}}\right) \  \ \middle| \ \ \underline\bmu\text{ is a flag matroid and } \bmu \in \bigcap_{f\in \mathscr P_{\br;n}^{(3)}} V_{\trop}(f) \right\}, \text{ and}\\
\FlDr_{\br;n}^{\geq 0} &= \left\{\bmu \in  \textstyle\prod_{i = 1}^k \PP\left(\TT^{\binom{[n]}{r_i}}\right) \  \ \middle| \ \ \underline\bmu\text{ is a flag matroid and } \bmu \in \bigcap_{f\in \mathscr P_{\br;n}^{(3)}} V_{\trop}^{\geq 0}(f) \right\}.
\end{align*}
\end{prop}

\begin{proof}
We will use the language and results from the study of matroids over hyperfields.  See \cite{BB19} for hyperfields and relation to matroid theory, and see \cite[\S2.3]{Gun} for a description of the signed tropical hyperfield $\TT\RR$, for which we note the following fact:  
The underlying set of $\TT\RR$ is $(\RR\times \{+,-\}) \cup \{\infty\}$, so given $c\in \TT$, 
	one can identify it with the element 
	$(c,+)\in \RR\times \{+,-\}$ 
	of $\TT\RR$  
	if $c<\infty$ and $\infty$ otherwise.
	
In the language of hyperfields, for a homogeneous polynomial $f$ in $m$ variables and a hyperfield $\mathbb F$, one has the notion of the ``hypersurface of $f$ over $\mathbb F$,'' which is a subset $V_{\mathbb F}(f)$ of $\PP(\mathbb F^m)$.  When $\mathbb F$ is the tropical hyperfield $\TT$, this coincides with $V_{\trop}(f)$ in \Cref{defn:trophyper}.
When $\mathbb F$ is the signed tropical hyperfield $\TT\RR$, a point $w\in \TT^m$, when considered as a point of $\TT\RR^m$, is in $V_{\TT\RR}(f)$ if and only if it is in $V_{\trop}^{\geq 0}(f)$.
Thus, in the language of flag matroids over hyperfields \cite{JL}, the flag Dressian is the partial flag variety $\operatorname{Fl}_{\br;n}(\TT)$ over $\TT$, and the nonnegative flag Dressian is the subset of the partial flag variety $\operatorname{Fl}_{\br;n}(\TT\RR)$ over $\TT\RR$ consisting of points that come from $\TT$.

Now, both the tropical hyperfield and the signed tropical hyperfield are perfect hyperfields because they are doubly distributive \cite[Corollary 3.45]{BB19}.
Our proposition then follows from \cite[Theorem 2.16 \& Corollary 2.24]{JL}, which together state the following:
When $\br$ consists of consecutive integers, for a perfect hyperfield $\mathbb F$, a point $p\in \prod_{i = 1}^k\PP\left(\mathbb F^{\binom{[n]}{r_i}}\right)$ is in the partial flag variety $\operatorname{Fl}_{\br;n}(\mathbb F)$ over $\mathbb  F$ if and only if the support of $p$ is a flag matroid and $p$ satisfies the three-term Pl\"ucker relations over $\mathbb F$.
\end{proof}

For completeness, we include the proof of the following fact.

\begin{lem}
The signed tropical hyperfield $\TT\RR$ is doubly distributive.  That is, for any $x,y,z,w \in \TT\RR$, one has an equality of sets $(x\boxplus y)\cdot (z\boxplus w) = xz \boxplus xw \boxplus yz \boxplus yw$.
\end{lem}

\begin{proof}
If any one of the four $x,y,z,w$ is $\infty$, then the desired equality is the usual distributivity of the signed tropical hyperfield.  Thus, we now assume that all four elements are in $\RR\times \{+,-\}$, and write $x = (x_\RR, x_\SS) \in \RR\times \{+,-\}$ and similarly for $y,z,w$.  If $x_\RR > y_\RR$, then $xz_\RR > yz_\RR$ and $xw_\RR > yw_\RR$,  so the equality follows again from the usual distributivity.  So we now assume that all four elements have the same value in $\RR$, and the equality then follows from the fact that the signed hyperfield $\SS$ is doubly distributive.
\end{proof}

\begin{rem}
Even when $\br$ does not consist of consecutive integers, \cite[Theorem 2.16]{JL} implies that the flag Dressian and the nonnegative flag Dressian are carved out by fewer polynomials than $\mathscr P_{\br;n}$ in the following way:  Denoting by
\[
\mathscr P_{\br;n}^{adj} = \bigcup_{i=1}^k \mathscr P_{r_i,r_i;n} \cup \bigcup_{i=1}^{k-1} \mathscr P_{r_i,r_{i+1};n},
\]
one has
\[
\FlDr_{\br;n} = \bigcap_{f\in\mathscr P_{\br;n}^{adj}} V_{\trop}(f) \quad\text{and}\quad \FlDr_{\br;n}^{\geq 0} = \bigcap_{f\in\mathscr P_{\br;n}^{adj}} V_{\trop}^{\geq 0}(f).
\]
	This generalizes the fact that a sequence of matroids $(M_1, \ldots, M_k)$ is a flag matroid if and only if $(M_i,M_{i+1})$ is a flag matroid for all $i = 1, \ldots, k-1$ \cite[Theorem 1.7.1, Theorem 1.11.1]{BGW03}.
\end{rem}

The following corollary of \Cref{prop:3terms} 
 is often useful in computation.
It states that the nonnegative tropical flag Dressian is in some sense 
``convex'' inside the tropical flag Dressian.

\begin{cor}\label{cor:3terms}
Suppose $\br = (r_1 < \cdots < r_k)$ consists of consecutive integers, and suppose we have points $\bmu_1, \ldots, \bmu_\ell \in \prod_i^k \TT^{\binom{[n]}{r_i}}$ that are in $\FlDr_{\br;n}^{\geq 0}$.  Then, if a nonnegative linear combination $c_1 \bmu_1 + \cdots + c_\ell \bmu_\ell$ is in $\FlDr_{\br;n}$, it is in $\FlDr_{\br;n}^{\geq 0}$.
\end{cor}

\begin{proof}
We make the following general observation:  Suppose $f = c_{\alpha} x^{\alpha} - c_{\beta}x^{\beta} + c_{\gamma} x^{\gamma}$ is a three-term polynomial in $\RR[x_1, \ldots, x_m]$ with $c_\alpha, c_\beta, c_\gamma$ positive.  Then an element $u\in \TT^m$ satisfies the positive-tropical relation of $f$ if and only if $\beta\bullet u = \min\{\alpha\bullet u, \gamma \bullet u\}$.
Hence, if $u_1, \ldots, u_\ell \in \TT^m$ each satisfy this relation, then a nonnegative linear combination of them can satisfy the tropical relation of $f$ only if the term at $\beta$ achieves the minimum, that is, only if the positive-tropical relation is satisfied.
The corollary now follows from this general observation and \Cref{prop:3terms}.
\end{proof}

\subsection{Equivalence of \ref{eqvs:TrFl} and \ref{eqvs:FlDr} 
in \Cref{thm:eqvs}}\label{pf:1}

Let $\br$ be a sequence of consecutive integers $(a, \ldots, b)$ for some $1\leq a \leq b \leq n$.  We will show that $\TrFl_{\br;n}^{\geq 0} = \FlDr_{\br;n}^{\geq 0}$.  The inclusion $\TrFl_{\br;n}^{\geq 0} \subseteq \FlDr_{\br;n}^{\geq 0}$ is immediate from \Cref{defn:tropFl}.  We will deduce $\TrFl_{\br;n}^{\geq 0} \supseteq \FlDr_{\br;n}^{\geq 0}$ by utilizing the two known cases of the equality $\TrFl_{\br;n}^{\geq 0} = \FlDr_{\br;n}^{\geq 0}$ --- when $\br = (r)$ and when $\br = (1, 2, \ldots, n)$.

\medskip

We start by recalling
 that tropicalization behaves well on subtraction-free rational maps.

\begin{defn}\label{def:Tropicalization}
Let 
$f = \sum_{a \in \mathcal A} c_{a} x^{a} \in \RR[x_1, \ldots, x_m]$ be a real polynomial, where $\mathcal A$ is a finite subset of $\ZZ_{\geq 0}^m$ and $0\neq c_{a}\in \RR$.
	We define the \emph{tropicalization}
	$\Trop(f): \RR^m \to \RR$ to be the piecewise-linear map 
	$w\mapsto \min_{a\in \mathcal A}(a\bullet w)$, where as before,
	$a \bullet w = a_1 w_1 + \dots + a_m w_m$.
\end{defn}
	
Note that $\Trop({f_1}{f_2}) = 
\Trop({f_1}) + \Trop({f_2}).$  Moreover,
 if $f_1$ and $f_2$ are two polynomials with positive coefficients, 
	and $a_1,a_2\in\RR_{>0}$, then 
$\Trop(a_1 {f_1}+a_2 {f_2}) = 
\min(\Trop({f_1}),\Trop({f_2})).$
These facts imply the following simple lemma, which 
 appears as \cite[Lemma 11.5]{RW:Duke}. See 
\cite[Proposition 2.5]{SW05} and \cite{PS04}
for closely related statements.

\begin{lem}\label{lem:trop}
	Let $f=(f_1,\dots,f_n):{\C}^m \to \C^n$ be a 
	rational map defined by polynomials $f_1,\dots,f_n$ with positive coefficients
	(or more generally by subtraction-free rational expressions).
	Let  
	$(x_1,\dots,x_m) \in (\C_{\geq 0})^m$,
	such that $f(x_1,\dots,x_m) = (y_1,\dots,y_n)$.
	Then 
 $$(\Trop(f))(\val(x_1),\dots,\val(x_m)) = 
	(\val(y_1),\dots, \val(y_n)).$$
\end{lem}

The next result states that we can extend points in the nonnegative Dressian to points in the nonnegative two-step flag Dressian.

\begin{prop}\label{lem:extend}
Given $\mu_d \in \Dr_{d;n}^{\geq 0}$ with rational coordinates, 
there exists $\mu_{d+1} \in \Dr_{d+1;n}^{\geq 0}$ such that 
$(\mu_d, \mu_{d+1}) \in \FlDr_{d, d+1; n}^{\geq 0}$.  
Similarly, 
there exists $\mu_{d-1} \in \Dr_{d-1;n}^{\geq 0}$ such that 
$(\mu_{d-1}, \mu_{d}) \in \FlDr_{d-1, d; n}^{\geq 0}$.  
\end{prop}

The proof of \Cref{lem:extend} requires the following 
refined results about Rietsch's cell decomposition of the nonnegative
flag variety.
\begin{thm}\label{thm:cell}
The  nonnegative flag variety has a cell decomposition into positive Richardsons
$$\Fl_n(\C_{\geq 0}) = \bigsqcup_{v\leq w} \mathcal{R}_{v,w}(\C_{>0})$$
where each cell 
$\mathcal{R}_{v,w}(\C_{>0})$
can be parameterized 
	using a map $$\phi_{v,w}: (\C_{>0})^{\ell(w)-\ell(v)} \to 
\mathcal{R}_{v,w}(\C_{>0}).$$
Moreover, this parameterization can be expressed 
as an embedding into projective space (e.g. using the flag minors)
using polynomials in the parameters with positive coefficients.
\end{thm}
\begin{proof}
	The first statement comes from \cite[Theorem 11.3]{MR04};  
Marsh and Rietsch were working over $\RR$ and $\RR_{>0}$ but the same proof holds over
Puiseux series. The statement that the parameterization can be expressed
	as an embedding into projective space using positive polynomials
	comes from 
 \cite[Proposition 5.1]{RW08}.
\end{proof}

\begin{cor}\label{cor:project}
Each $m$-dimensional positroid cell $S_{\mathcal{B}}(\C_{>0})$
in the nonnegative Grassmannian 
	$\Gr_{d,n}(\C_{\geq 0})$  is the projection
	$\pi_d(\mathcal{R}_{v,w}(\C_{>0}))$ of some positive Richardson 
of dimension $m=\ell(w)-\ell(v)$ in $\Fl_n(\C_{\geq 0})$, 
so we get a 
subtraction-free rational map  
$$\pi_d \circ \phi_{v,w}: (\C_{>0})^{m} \to \mathcal{R}_{v,w}(\C_{>0}) 
	\to S_{\mathcal{B}}(\C_{>0}).$$
\end{cor}

\begin{proof}
That fact that each positroid cell is the projection of a positive Richardson
was discussed in \Cref{def:backgroundflag}.  
The result now follows from \Cref{thm:cell}.
\end{proof}

\begin{proof}[Proof of \Cref{lem:extend}]
Using \cite[Theorem 9.2]{AHLS}, the fact that 
$\mu_d \in \Dr_{d;n}^{\geq 0}$ with rational coordinates implies that 
$\mu_d= \val(\{\Delta_I(V_d)\})$ 
for some subspace
$V_d\in \Gr_{d,n}(\C_{\geq 0})$, and hence $V_d$ lies in 
some positroid cell 
	$S_{\mathcal{B}}(\C_{>0})$ over Puiseux series.  

By \Cref{cor:project}, $V_d$ is the projection of 
a point 
	$(V_1,\dots,V_n)$ of $\Fl_n(\mathcal{C}_{\geq 0})$,
 which in turn 
is the image of a point $(x_1,\dots,x_{m})\in (\mathcal{C}_{>0})^m$,
and the Pl\"ucker coordinates $\Delta_I(V_j)$ of each $V_j$
are expressed as positive polynomials 
$\Delta_I(x_1,\dots,x_m)$ in the parameters $x_1,\dots,x_m$.

In particular, we have subtraction-free maps 
$$\pi_d \circ \phi_{v,w}: (\C_{>0})^m \to \Fl_n(\C_{\geq 0}) \to \Gr_{d,n}(\C_{\geq 0})$$
taking 
$$(x_1,\dots,x_m) \mapsto \{\Delta_I(x_1,\dots,x_m) \ \vert \ 
I \subset [n]\} \mapsto \left\{\Delta_I(x_1,\dots,x_m) \ \vert \ 
I \in {[n] \choose d}\right\}.$$
The fact that the maps $\phi_{v,w}$ and $\pi_d$ are subtraction-free
 implies by \Cref{lem:trop} that we can tropicalize them, obtaining 
maps 
$$\Trop(\pi_d \circ \phi_{v,w}): \RR^m \to \TrFl_n^{\geq 0} \to \TrGr_{d,n}^{\geq 0}$$
taking 
	$$(\val(x_1),\dots,\val(x_m)) \mapsto \{\val(\Delta_I(x_1,\dots,x_m)) \ \vert \ 
	I \subset [n]\} \mapsto \left
\{\val(\Delta_I(x_1,\dots,x_m)) \ \vert \ 
	I \in {[n] \choose d}\right\}.$$

We now let 
	$\mu_{d+ 1}= \{\val(\Delta_I(V_{d+ 1})) \ \vert \ I\in {[n] \choose d+1}\}$
and 	$\mu_{d- 1}= \{\val(\Delta_I(V_{d- 1})) \ \vert \ I\in {[n] \choose d-1}\}$.  
	By construction we have that 
all the three-term (incidence) Pl\"ucker relations hold for 
$(\mu_d, \mu_{d+1})$, and similarly for $(\mu_{d-1}, \mu_{d})$.
Therefore 
$(\mu_d, \mu_{d+1}) \in \FlDr_{d, d+1; n}^{\geq 0}$ and 
$(\mu_{d-1}, \mu_{d}) \in \FlDr_{d-1, d; n}^{\geq 0}$.  
\end{proof}

The following consequence of 
\Cref{lem:extend} is very useful.
\begin{cor}\label{lem:extend2}
	Let $a' \leq a \leq b \leq b'$ be positive integers,
	and let $\br=(a,a+1,\dots,b)$ and 
		 $\br'=(a',a'+1,\dots,b')$ be sequences of consecutive
		 integers.  Then any point 
		 $(\mu_a,\dots,\mu_b)\in \FlDr_{\br;n}^{\geq 0}$
		 with rational coordinates
		 can be extended to a point
		 $(\mu_{a'},\mu_{a'+1},\dots,\mu_a,\dots, \mu_b, \dots, \mu_{b'})\in \FlDr_{\br';n}^{\geq 0}$.
\end{cor}
\begin{proof}
We start with $\bmu = (\mu_a, \mu_{a+1}, \dots, \mu_b) \in 
\FlDr_{\br;n}^{\geq 0}$.   We take $\mu_b$ and repeatedly use \Cref{lem:extend} to construct
$\mu_{b+1}$, then $\mu_{b+2}, \dots, \mu_{b'}$.
Similarly we take $\mu_a$ and use \Cref{lem:extend} to construct
$\mu_{a-1}, \mu_{a-2},\dots, \mu_{a'}$.
Now by construction 
$(\mu_{a'},\mu_{a'+1},\dots, \mu_a,\dots,\mu_b,\dots,\mu_{b'})$ 
satisfies:
\begin{itemize}
\item $\mu_i \in \Dr_{i;n}^{\geq 0}$ for $i =a', a'+1,\dots, b'$; 
\item all three-term incidence-Pl\"ucker relations hold 
	(because the three-term incidence-Pl\"ucker relations occur only in consecutive ranks).
\end{itemize}
	Therefore $(\mu_{a'},\mu_{a'+1},\dots, \mu_{b'})\in \FlDr_{\br';n}^{\geq 0}$ by \Cref{prop:3terms}.
\end{proof}

\begin{thm}\label{thm:almostAB}
	Let $\br=(a,a+1,\dots,b)$ be a sequence of consecutive integers, and 
let $\bmu \in \FlDr_{\br;n}^{\geq 0}$ with rational coordinates.  Then 
 $\bmu\in \TrFl_{\br;n}^{\geq 0}$.
\end{thm}

\begin{proof}
We start with $\bmu = (\mu_a, \mu_{a+1}, \dots, \mu_b) \in 
\FlDr_{\br;n}^{\geq 0}$, and use \Cref{lem:extend2} to construct
$(\mu_{1},\dots, \mu_{n})\in \FlDr_{n}^{\geq 0}$.
Now \cite[Theorem 5.21\textsuperscript{trop}]{Bor} states that $\FlDr_n^{\geq 0} = \TrFl_n^{\geq 0}$.
Hence, we have $(\mu_1,\dots, \mu_n) \in \TrFl_n^{\geq 0}$, so $(\mu_a,\mu_{a+1},\dots,\mu_b) \in \TrFl_{\br,n}^{\geq 0}$.
\end{proof}

\begin{proof}[Proof of \ref{eqvs:TrFl}$\iff$\ref{eqvs:FlDr} in \Cref{thm:eqvs}]
We only need show that \ref{eqvs:FlDr}$\implies$\ref{eqvs:TrFl}, i.e.\ that $\TrFl_{\br;n}^{\geq 0} \supseteq \FlDr_{\br;n}^{\geq 0}$, since the other direction is trivial.  But this follows from \Cref{thm:almostAB} because the points in $\FlDr_{\br;n}^{\geq 0}$ with rational coordinates are dense in $\FlDr_{\br;n}^{\geq 0}$, and $\TrFl_{\br;n}^{\geq 0}$ is closed. 
\end{proof}

\begin{rem}
Note that our method of proof crucially used the fact that $\br$ is a sequence of consecutive integers:
we used \Cref{lem:extend} to fill in the ranks from $b$ through $n$ and from $a$ down to $1$.
	But if say we were considering $\br = \{a, b\}$ with $b-a>1$ and $\bmu = (\mu_a, \mu_b)$, we could not guarantee using \Cref{lem:extend}
	that we could construct $\mu_{b-1}, \mu_{b-2},\dots, \mu_{a+1}$ in a way that is consistent with $\mu_a$.
\end{rem}

\begin{rem}\label{BK2}
Recall from \Cref{thm:BK} that if $\br$ is a sequence of consecutive integers, 
the two notions of the positive/nonnegative part of the 
flag variety (see \Cref{def:2positive})
 coincide.
	The method used to prove the equivalence of \ref{eqvs:TrFl} and \ref{eqvs:FlDr} in \Cref{thm:eqvs} can be applied in a non-tropical context to prove \Cref{thm:BK} in an alternate way. We start by noting that the result holds when $\br=(a)$, which is to say, for the nonnegative Grassmannian \cite[Corollary 1.2]{TW13} and also when $\br=(1,2,\ldots, n)$, which is to say, for the  nonnegative complete flag variety 
	\cite[Theorem 5.21]{Bor}. To prove the result for $\br=(a,a+1,\ldots, b)$, we start with a flag $V_{\bullet} = (V_a,\dots,V_{b})$ in ranks $\br$ whose Pl\"ucker coordinates are all nonnegative, so that $V_{\bullet}$ is Pl\"ucker nonnegative. As in \Cref{lem:extend}, we can use the $\br=(a)$ case to argue that the flag can be extended to lower ranks in such a way that all the Pl\"ucker coordinates are nonnegative. Dually, we can extend to higher ranks from the $\br=(b)$ case. This yields a complete flag $(V_1,\dots,V_n)$ 
	with all nonnegative Pl\"ucker coordinates. We can then apply the result in the complete flag case to conclude that $(V_1,\dots,V_n)$ lies in $\Fl_{n}^{\geq 0}$. Thus, $V_{\bullet}$ is a projection of the  nonnegative complete flag $(V_1,\dots,V_n)$ and itself lies in 
	$\Fl_{\br;n}^{\geq 0}$, which is to say, $V_{\bullet}$ is Lusztig nonnegative. 
\end{rem}

The \newterm{strictly positive tropicalization} of a partial flag variety $\TrFl_{\br;n}^{> 0}$ is the subset of $\TrFl_{\br;n}^{\geq 0}$ consisting of points whose coordinates are never  $\infty$.  Define similarly the \newterm{strictly positive flag Dressian} $\FlDr_{\br;n}^{>0}$.  The weaker version of \Cref{thm:almostAB} stating that $\TrFl_n^{>0} = \FlDr_n^{>0}$ was established in \cite[Lemma 19]{JLLO} as follows.
One starts by noting that if $\mu \in \Dr_{r+m;n+m}^{\geq 0}$, then the sequence of minors $(\mu_r, \ldots, \mu_{r+m})$ where $\mu_{r+i} = \mu\setminus\{n+1, \ldots, n+i\}/\{n+i+1, \ldots, n+m\}$ is a point in $\FlDr_{r,\ldots,r+m;n}^{\geq 0}$.
Then, the crucial step is a construction in discrete convex analysis \cite[Proposition 2]{MS18} that shows that every element of $\FlDr_n^{>0}$ arises from an element of $\Dr_{n;2n}^{>0}$ in this way.  One then appeals to $\Gr_{r;n}^{>0} = \Dr_{r;n}^{>0}$ established in \cite{SW21}.

 \Cref{eg:cantlift} shows  that the above argument does not work if one replaces
 ``strictly positive'' with ``nonnegative.''  In particular, 
 the crucial step fails: that is, not every element of $\FlDr_n^{\geq 0}$ arises from an element of $\Dr_{n;2n}^{\geq 0}$ in such a way.

\begin{eg}\label{eg:cantlift}
Let $(M_1, M_2, M_3)$ be matroids on $[3]$ whose sets of bases are $(\{1,3\}, \{13\}, \{123\})$.  The matrix 
\[
\begin{bmatrix}
1 & 0 & 1\\ 
0 & 0 & 1\\
0 & -1 & 0
\end{bmatrix}
\]
shows that it is a flag positroid.  However, we 
claim that there is no positroid $M$ of rank $3$ on $[6]$ such that $M_1 = M\setminus 4 / 56$, $M_2 = M\setminus 45 /6$, and $M_3 = M\setminus456$.  
	Since all three cases involve deletion by 4, 
if we replace $M\setminus 4$ by $M'$, and decrease each of $5,6$ by $1$,
then	we are claiming that there is no positroid $M'$ of rank 3 on $[5]$ such that
\begin{equation}\label{eq:minors}
M_1 = M'/45,\ M_2 = M'\setminus 4 / 5,\ \text{and }M_3 = M'\setminus 45.
\end{equation}
From $M_1 = M'/45$ and $M_2 = M'\setminus 4 / 5$, we have that $M'/5$ has bases $\{14,34,13\}$, and similarly, we have $M'\setminus 4$ has bases $\{135,123\}$.  Hence, the set of bases of $M'$ contains $\{123, 135,145,345\}$, and does not contain $\{125,235,245\}$.  
By considering the Pl\"ucker relation 
	$$p_{134} p_{235} = p_{123} p_{345}+p_{135} p_{234},$$
	we see that no positroid satisfies these properties.
\end{eg}

\section{Positively oriented flag matroids}\label{sec:POM}

In this section we explain the relationship between the nonnegative flag
Dressian and positively oriented flag matroids, 
and we apply our previous results to flag matroids.
In particular, we prove 
 \Cref{cor:real}, which says that every positively oriented flag
 matroid of consecutive ranks is realizable.
We also  prove  \Cref{lem:extend3}, which says 
that a positively oriented flag matroid
of consecutive ranks $a,\dots,b$ can be extended
to ranks $a',\dots,b'$ (for $a'\leq a \leq a \leq b$).

\subsection{Oriented matroids and flag matroids}\label{sec:posorient}

We give here a brief review of oriented matroids in terms of Pl\"ucker relations.
Let $\SS = \{-1,0,1\}$ be the 
\emph{hyperfield of signs}.  For a polynomial $f = \sum_{a\in \mathcal A} c_a x^a \in \RR[x_1, \ldots, x_m]$, we say that an element $\chi\in \SS^m$ is \newterm{in the null set} of $f$ if the set $\{\operatorname{sign}(c_a) \chi^a\}_{a\in \mathcal A}$ is either $\{0\}$ or contains $\{-1,1\}$.

\begin{defn}\label{def:oriented}
	An \newterm{oriented matroid} of rank $r$ on $[n]$ is a point $\chi \in \SS^{\binom{[n]}{r}}$, called a \emph{chirotope},
	such that $\chi$ is in the null set of $f$ for every $f\in \mathscr P_{r,r;n}$.  Similarly, an \newterm{oriented flag matroid} of ranks $\br$ is a 
	point $\boldsymbol\chi = (\chi_1, \ldots, \chi_k) \in \prod_{i = 1}^k \SS^{\binom{[n]}{r_i}}$ 
	such that $\boldsymbol\chi$ is in the null set of $f$ for every $f\in \mathscr P_{\br;n}$.
\end{defn}

While these definitions may seem different from those in the standard reference \cite{BLVSWZ99} on oriented matroids,  \Cref{def:oriented} is equivalent to \cite[Definition 3.5.3]{BLVSWZ99} by \cite[Example 3.33]{BB19}.  The definition of oriented flag matroid here is equivalent to the definition of a sequence of oriented matroid quotients \cite[Definition 7.7.2]{BLVSWZ99} by \cite[Example above Theorem D]{JL}.

\begin{defn}
A \newterm{positively oriented matroid} is an oriented matroid $\chi$ such that $\chi$ only takes values 0 or 1.  
Similarly, we define a \newterm{positively oriented flag matroid} to be an oriented flag matroid $\boldsymbol\chi$ such that $\boldsymbol\chi$ only takes values 0 or 1.
\end{defn}

A positroid $M$ defines a positively oriented matroid $\chi = \chi_M$ 
where $\chi$ takes value 1 on its bases and 0 otherwise.  In 1987, 
da Silva \cite{daS} conjectured that every positively oriented matroid arises in this way; this conjecture was subsequently proved in  \cite{ARW17}
and then \cite{SW21}.
\begin{thm}\label{thm:realizable}\cite{ARW17}
Every positively oriented matroid $\chi$ is realizable, i.e. 
$\chi$ has the form $\chi_M$ for some positroid $M$.
\end{thm}
By \Cref{thm:realizable}, each positively oriented flag matroid
is a sequence of positroids which is also an oriented flag matroid.

In this section 
we will prove \Cref{cor:real}, which generalizes
\Cref{thm:realizable}, and says that every positively oriented flag matroid
$(\chi_1,\dots,\chi_k)$ of consecutive ranks $r_1<\dots<r_k$ 
can be realized by a flag positroid.
But before we prove it, let us give an example that shows that imposing 
the oriented flag matroid condition is stronger than imposing that we have a realizable flag matroid whose consistent matroids are positroids.

\begin{eg}\label{eg:notreal}
We give an example of a realizable flag matroid that has positroids as its constituent matroids but is not a flag positroid.  
This example also appeared in \cite[Example 5]{JLLO} and 
	\cite[Example 6]{BK}.
	Let $(M,M')$ be matroids of ranks 1 and 2 on $[3]$ whose sets of bases are $\{1,3\}$ and $\{12,13,23\}$, respectively.  Both are positroids.  
	We can realize $(M,M')$ as a flag matroid using the matrix
\[
\begin{bmatrix}
a & 0 & b\\
c & d & e
\end{bmatrix},
\]
where the nonvanishing minors $a,b,ad, -bd, ae-bc$ are nonzero.
	In order to realize $(M,M')$ as a flag positroid,
	we need to choose real numbers $a, b, c, d, e$ such that
	all these minors are strictly positive.  However, $a>0$ and $ad>0$ implies $d>0$, while $b>0$ and $-bd>0$ implies $d<0$.

	This example is consistent with \Cref{cor:real} because $(M,M')$, when considered as a sequence of positively oriented matroids, is not an oriented flag matroid.
\end{eg}

\subsection{From the nonnegative flag Dressian 
to positively oriented flag matroids}\label{pf:cor}

We start with the following  simple observation. While the proof is very simple,
we label it a ``theorem'' to emphasize its importance.

\begin{thm}\label{lem:signembed}
The set of positively oriented flag matroids of ranks $\br$ can be identified with the set of points of the nonnegative flag Dressian $\FlDr_{\br;n}^{\geq 0}$ whose coordinates are all either 0 or $\infty$.
\end{thm}

\begin{proof}
Given a point $\chi = (\chi_1, \ldots, 
	\chi_m) \in \{0,1\}^m\subset \SS^m$,\footnote{Note that 
	$(\chi_1,\dots,\chi_m)$ is not a sequence of chirotopes in this proof,
	instead each $\chi_i \in \SS.$} we define
	$t(\chi) = (t_1, \ldots, t_m) \in \TT^m$ by setting
	$t_i = 0$ if $\chi_i = 1$ and $t_i = \infty$ if $\chi_i = 0$.  Then, we observe that $\chi$ is in the null set of a polynomial $f\in \RR[x_1, \ldots, x_m]$ if and only if the image of $t(\chi)$ in $\PP(\TT^m)$ is a point in $V_{\trop}^{\geq 0}(f)$.  Therefore, 
each  positively oriented flag matroid $\boldsymbol\chi$ can be 
	identified with the element $t(\boldsymbol\chi)$ in the nonnegative flag Dressian $\FlDr_{\br;n}^{\geq 0}$.
\end{proof}

We now prove that every positively oriented flag matroid
	$\boldsymbol\chi
=(\chi_1,\dots,\chi_k)$ of consecutive ranks $r_1<\dots<r_k$ 
is realizable.

\begin{proof}[Proof of \Cref{cor:real}]
By the lemma, we may identify a positively oriented flag matroid 
	$\boldsymbol\chi$ 
	as an element $t(\boldsymbol\chi)$
	of the nonnegative flag Dressian.
	Because the ranks $\br$ are consecutive integers, the equivalence \ref{eqvs:TrFl}$\iff$\ref{eqvs:FlDr} of \Cref{thm:eqvs} implies that $t(\boldsymbol\chi)$ is thus a point in $\TrFl_{\br;n}^{\geq 0}$.  
	Because $t(\boldsymbol\chi)$ has rational coordinates 
	(all non-$\infty$ coordinates are 0), 
\Cref{prop:closure} implies that 
	 $t(\boldsymbol\chi)=\val(p)$ for some 
	 $p \in\prod_{i = 1}^k \PP\left(\C_{\geq 0}^{\binom{[n]}{r_i}}\right)$.
	 Setting the parameter $t$ in each Puisseux series of $p$ to 0 now gives 
	 the realization of $\boldsymbol\chi$ as a flag positroid.
\end{proof}

As in \Cref{ques}, we do not know whether the corollary holds when $\br$ does not consist of consecutive integers.  The following example shows that one cannot reduce to the consecutive ranks case.

\begin{eg}\label{eg:cantcomplete}
We give an example of a flag positroid $(M,M')$ on $[4]$ of ranks $(1,3)$ such that there is no flag positroid $(M,M_2,M')$ with rank of $M_2$ equal to 2.  Let the sets of bases of $M$ and $M'$ be $\{1,2,3,4\}$ and $\{123,234\}$, respectively.  The matrix
\[
\begin{bmatrix}
1 & 1 & 1 & 1\\
0 & 1 & 0 & 0\\
0 & 0 & 1 & 0
\end{bmatrix}
\]
for example shows that $(M,M')$ is a flag positroid.  However, this flag positroid cannot be extended to a flag positroid with consecutive ranks.  To see this, note that any realization of $(M,M')$ as a flag positroid, after row-reducing by the first row, is of the form
\[
\begin{bmatrix}
1 & a & b & c \\
0 & x & y & 0\\
0 & z & w & 0
\end{bmatrix}
\]
where $a,b,c>0$ and $xw-yz>0$.  The minors of the matrix formed by the first two rows include $x, y, -cx, -cy$, which cannot be all nonnegative since $c>0$ and not both of $x$ and $y$ are zero.
\end{eg}

\begin{rem}\label{rem:altproof}
Let us sketch an alternate proof of \Cref{cor:real} that relies only on the weaker version of \ref{eqvs:TrFl}$\iff$\ref{eqvs:FlDr} in \Cref{thm:eqvs} that the strictly positive parts agree, i.e. that $\TrFl_{\br;n}^{>0} = \FlDr_{\br;n}^{>0}$.
For a matroid $M$ of rank $d$, define $\rho_M \in \RR^{\binom{[n]}{d}}$ by $\rho_M(S) = d - \operatorname{rk}_M(S)$ for $S\in \binom{[n]}{d}$, where $\operatorname{rk}_M$ is the rank function of $M$.
If $M$ is a positively oriented matroid, then $\rho_M$ is a point in the positive Dressian $\Dr^{>0}_{d,n}$ \cite[Proof of Theorem 5.1]{SW21}.  One can use this to show that if $\M = (M_1, \ldots, M_k)$ is a positively oriented flag matroid of consecutive ranks $\br$, then the sequence $\boldsymbol\rho = (\rho_{M_1},\ldots, \rho_{M_k})$ is a point in $\FlDr_{\br;n}^{>0}$.  
	Since $\TrFl_{\br;n}^{>0} = \FlDr_{\br;n}^{>0}$ and $\boldsymbol\rho$ has rational coordinates, \Cref{prop:closure} implies that there is a point $p \in \Fl_{\br;n}(\C_{\geq 0})$ with $\operatorname{val}(p) = \boldsymbol\rho$.  Consider the coordinate $p(S) \in \C$ of $p$ at a subset $S\in \binom{[n]}{r_i}$.  By construction, the initial term of $p(S)$ is $ct^q$ for some positive real $c$ and a nonnegative integer $q$, where $q$ is zero exactly when $S$ is a basis of $M_i$.
Thus, setting the parameter $t$ to $0$ in the 
 Puisseux series of $p$ 
 gives a realization of $\M$ as a flag positroid.
\end{rem}

We now use \Cref{lem:signembed} to give a matroidal analogue
of \Cref{lem:extend2}.
\begin{cor}\label{lem:extend3}
	Let $a' \leq a \leq b \leq b'$ be positive integers,
	and let $(M_a, M_{a+1},\dots,M_b)$ be 
	a positively oriented flag matroid on $[n]$ of consecutive
	ranks $a,a+1,\dots,b$, that is,
	a sequence of positroids $M_a,\dots,M_b$ 
	which is also an oriented flag matroid.  Then we can extend 
	it to a positively oriented flag matroid 
		 $(M_{a'},M_{a'+1},\dots,M_a,\dots, M_b, \dots, 
		 M_{b'})$ of consecutive ranks
		 $a',a'+1,\dots,b'$.
\end{cor}
\begin{proof}
As in \Cref{lem:signembed}, we view 
the positively oriented flag matroid $(M_a,\dots,M_b)$
as a point of the nonnegative flag Dressian 
$(\mu_a,\dots,\mu_b) \in \FlDr_{\br;n}^{\geq 0}$ 
whose coordinates are all either 0 or $\infty$.  The desired statement
	\emph{almost} follows from 
	\Cref{lem:extend}: we just need to check that 
	we can extend $(\mu_a,\dots,\mu_b)$ in a way which 
	preserves the fact that coordinates are all either $0$ or 
	$\infty$.  This is true, and we prove it by 
	following the proof of 
	\Cref{lem:extend} and replacing all instances
	of the positive Puiseux series $\mathcal{C}_{>0}$ by 
	the positive Puiseux series with 
	\emph{constant coefficients},
	that is, by $\RR_{>0}$.
	Alternatively, we can use our result that 
 $(M_a,\dots,M_b)$ is realizable by a flag positroid, 
	and then argue as in \Cref{BK2}.
\end{proof}

\section{Subdivisions of flag matroid polytopes}\label{sec:subdivision}

\subsection{Flag Dressian and flag matroidal subdivisions}

Consider a point $\bmu=(\mu_1,\dots,\mu_k) \in \prod_{i=1}^k \PP\left(\TT^{\binom{[n]}{r_i}}\right)$ 
such that its support $\underline\bmu$ is a flag matroid.  By construction, the vertices of the flag matroid polytope $P(\underline\bmu)$ have the form $\be_{B_1} + \cdots + \be_{B_k}$ where $B_i$ is a basis of the matroid $\underline\mu_i$ for each $i = 1, \ldots, k$.

\begin{defn}\label{defn:subdiv}
We define $\mathcal D_\bmu$ to be the coherent subdivision of $P(\underline\bmu)$ induced by assigning each vertex $\be_{B_1} + \cdots + \be_{B_k}$ of $P(\underline\bmu)$ the weight $\mu_1(B_1) + \cdots + \mu_k(B_k)$.  That is, the faces of $\mathcal D_\bmu$ correspond to the faces of the lower convex hull of the set of points
\[
\{(\be_{B_1} + \cdots + \be_{B_k}, \mu_1(B_1) + \cdots + \mu_k(B_k)) \in \RR^n \times \RR : \be_{B_1} + \cdots + \be_{B_k} \text{ a vertex of $P(\bmu)$}\}.
\]
\end{defn}

The points of the flag Dressians are exactly the ones for which the subdivision $\mathcal D_\bmu$ consists of flag matroid polytopes.

\begin{thm}\cite[Theorem A.(a)\&(c)]{BEZ21} \label{th:BEZ}
A point $\bmu\in \prod_{i = 1}^k \PP\left(\TT^{\binom{[n]}{r_i}}\right)$ is in the flag Dressian $\FlDr_{\br;n}$ if and only if the all faces of the subdivision $\mathcal D_\bmu$ are flag matroid polytopes.
\end{thm}

When $\br$ consists of consecutive integers $(a,a+1,\ldots, b)$, the nonnegative analogue of this theorem is the equivalence of \ref{eqvs:FlDr} and \ref{eqvs:subdiv} in \Cref{thm:eqvs}, which states that a point $\bmu\in \prod_{i = a}^b \PP\left(\TT^{\binom{[n]}{i}}\right)$ is in the nonnegative flag Dressian $\FlDr_{\br;n}^{\geq 0}$ if and only if all faces of the subdivision $\mathcal D_\bmu$ are flag positroid polytopes.  A different nonnegative analogue of \Cref{th:BEZ} that holds for $\br$ not 
necessarily consecutive, but loses the flag positroid property, can be found in \Cref{rem:subdivmod}.

\subsection{The proof of 
\ref{eqvs:FlDr}$\implies$\ref{eqvs:subdiv}$\implies$\ref{eqvs:2faces}$\implies$\ref{eqvs:FlDr} in \Cref{thm:eqvs}}\label{pf:2}
We start by recording two observations.
The first is a well-known consequence of the greedy algorithm for matroids; see for instance \cite[Proposition 4.3]{AK06}.
For a matroid $M$ on $[n]$ and a vector $\mathbf v \in \RR^n$, let $\operatorname{face}(P(M), {\mathbf v})$ be the face of the matroid polytope $P(M)$ that maximizes the standard pairing with $\mathbf v$.

\begin{prop}\label{prop:greedy}
	Let $M$ be a matroid on $[n]$ and 
let $\mathscr S = (\emptyset\subsetneq S_1 \subsetneq \cdots \subsetneq S_\ell \subsetneq [n])$ be a chain of nonempty proper subsets of $[n]$.  
For a vector $\mathbf v_{\mathscr S}$ in the relative interior of the cone $\RR_{\geq 0}\{\be_{S_1}, \ldots, \be_{S_\ell}\}$, we have
\[
\operatorname{face}(P(M), {\mathbf v_{\mathscr S}}) = P(M^{\mathscr S}),
\]
where $M^{\mathscr S} = M|S_1 \oplus M|S_2/S_1 
	\oplus M|S_3/S_2 \oplus \cdots \oplus M/S_\ell$ is the direct sum of minors of $M$.

	For $\M = (M_1, \ldots, M_k)$ a flag matroid, since $P(\M)$ is the Minkowski sum $P(M_1) + \cdots + P(M_k)$, we likewise have that $\operatorname{face}(P(\M), \mathbf v_{\mathscr S}) = P(\M^{\mathscr  S}) = P(M_1^{\mathscr S}) + \cdots + P(M_k^{\mathscr S})$, where $\M^{\mathscr S}=(M_1^{\mathscr S},\ldots, M_k^{\mathscr S})$. 
	In particular, the face of a flag matroid
	polytope is a flag matroid polytope.
\end{prop}

The second observation concerns the following operations that we will show  preserve the nonnegative flag Dressian.
Recall that for $w\in \TT^{\binom{[n]}{r}}$,  its support $\underline w$ is $\{S\in \binom{[n]}{r} : w_S\neq \infty\}$.
\begin{itemize}
\item We consider a point $w\in \TT^{\binom{[n]}{r}}$ as a set of weights 
	on the vertices
		$\{\be_S : S \in \underline w\}$ of  $P(\underline w) \subset \RR^n$.
	Given an affine-linear function $\varphi: \RR^n \to \RR$ and an element $w\in \TT^{\binom{[n]}{r}}$, we define
\[
\varphi w \in \TT^{\binom{[n]}{r}} \quad\text{by}\quad (\varphi w)(S) = \varphi(\be_S) + w(S) \text{ for $S\in \binom{[n]}{r}$}.
\]
\item For a point $w\in \TT^{\binom{[n]}{r}}$, denote by $w^{\mathrm{in}} \in \TT^{\binom{[n]}{r}}$ its \emph{initial part}, i.e.
\[
w^\mathrm{in}(S)  = \begin{cases}
0 & \text{if $w(S) = \min\{w(S') : S' \in \binom{[n]}{r}$}\}\\
\infty & \text{otherwise}.
\end{cases}
\]
\end{itemize}

\begin{prop}\label{prop:afflin}
Let $\br = (r_1,\dots,r_k)$ be a sequence of increasing integers in $[n]$.  
Suppose $\boldsymbol w = (w_1, \ldots, w_k) \in \FlDr_{\br;n}^{\geq 0}$.
Then, the following hold.
\begin{enumerate}
\item The support $\underline{\boldsymbol w}$ is a positively oriented flag matroid.  In particular, it is a flag positroid when $\br = (r_1, \ldots, r_k)$ consists of consecutive integers.
\item We have $\varphi \boldsymbol w = (\varphi w_1, \ldots, \varphi w_k) \in \FlDr_{\br;n}^{\geq 0}$ for any affine-linear functional $\varphi$ on $\RR^n$.
\item We have $\boldsymbol w^{\mathrm{in}}= (w_1^{\mathrm{in}}, \ldots,  w_k^{\mathrm{in}})\in \FlDr_{\br;n}^{\geq 0}$.
\end{enumerate}
\end{prop}

\begin{proof}
We may consider $\underline{\boldsymbol w}$ as an element $\prod_{i = 1}^k \PP\left(\TT^{\binom{[n]}{r_i}}\right)$ by assigning the value 0 to a subset $S$ if it is in the support of $\boldsymbol w$ and $\infty$ otherwise.
Then, we have $\underline{\boldsymbol w} \in \FlDr_{\br;n}^{\geq 0}$ because the terms in each of the tropical Pl\"ucker relations that achieve the minimum when evaluated at $\boldsymbol w$ continue to do so when evaluated at $\underline{\boldsymbol w}$.
The statement (1) follows from \Cref{lem:signembed} and \Cref{cor:real}

The support is unchanged by $\varphi$, so $\underline{\varphi\boldsymbol w}$ is a flag matroid.  
The statement (2) now follows because for each of the positive-tropical Pl\"ucker relations, the operation $\varphi$ preserves the terms at which the minimum is achieved.

	The support $\underline{\boldsymbol w^{\mathrm{in}}}$ is a 
	flag matroid   
by \Cref{th:BEZ}
and  because $P(\underline{\boldsymbol w^{\mathrm{in}}})$ is a face in the subdivision $\mathcal D_{\boldsymbol w}$ of $P(\underline{\boldsymbol w})$.
The statement (3) now follows because for each of the positive-tropical Pl\"ucker relations, the operation $^{\mathrm{in}}$ either preserves the terms at which the minimum is achieved or changes all the terms involved to $\infty$.
\end{proof}

\begin{rem}
While it's not needed here, we note that \Cref{prop:afflin} is the ``positive'' analogue of the following statement, which is proved similarly:
If $\boldsymbol w \in \FlDr_{\br;n}$, then (1) $\underline{\boldsymbol w}$ is a flag matroid, (2) $\varphi\boldsymbol w\in \FlDr_{\br;n}$, and (3) $\boldsymbol w^{\rm{in}} \in \FlDr_{\br;n}$.  See also \cite[Corollary 4.3.2]{BEZ21} for related statements.
\end{rem}

\begin{proof}[Proof of \ref{eqvs:FlDr}$\implies$\ref{eqvs:subdiv}]
Every face in the coherent subdivision is the initial one after an affine-linear transformation.  Hence, the implication follows from \Cref{prop:afflin}.
\end{proof}

\begin{rem}\label{rem:subdivmod}
One may modify the statement \ref{eqvs:subdiv} to the following:
\begin{itemize}
\item[(c')] Every face in the coherent subdivision $\mathcal D_\bmu$ of $P(\underline\bmu)$ is the flag matroid polytope of a positively oriented flag matroid.
\end{itemize}
Similar argument as above shows that \ref{eqvs:FlDr}$\implies$(c') even when $\br$ doesn't consist of consecutive integers.  One can also verify the converse (c')$\implies$\ref{eqvs:FlDr} in this more general case as follows:

Suppose for contradiction (c') but not \ref{eqvs:FlDr} for some $\bmu$.  Then \Cref{th:BEZ} implies that $\bmu$ is in the flag Dressian, and thus the failure of \ref{eqvs:FlDr} implies that there is a Pl\"ucker relation where the minimum occurs at least twice but at the terms whose coefficients have the same sign.  \Cref{prop:afflin} implies that, replacing $\bmu$ by $\varphi\bmu$ for some $\varphi$ if necessary, we may conclude that the same is true for that Pl\"ucker relation evaluated at $\bmu^{\rm{in}}$.  But then $\bmu^{\rm{in}}$, which arise as a face in the subdivision, is not a positively oriented flag matroid by \Cref{lem:signembed}, contradicting (c').

There is no equivalence of (c') and \ref{eqvs:3terms} since three-term incidence relations exist only for consecutive ranks.
\end{rem}

The implication \ref{eqvs:subdiv}$\implies$\ref{eqvs:2faces} is immediate.

\begin{proof}[Proof of \ref{eqvs:2faces}$\implies$\ref{eqvs:FlDr}]
First, assumption \ref{eqvs:2faces} implies that every edge of the subdivision $\mathcal D_{\bmu}$ of $P(\underline\bmu)$ is a flag matroid polytope, i.e.\ it is parallel to $\be_i - \be_j$ for some $i\neq j \in [n]$ and its two vertices are equidistant from the origin.  Hence the edges of $P(\underline\bmu)$ have the same property, so $\underline\bmu$ is a flag matroid.
By Proposition~\ref{prop:3terms}, to show \ref{eqvs:FlDr} it now suffices to show that every positive-tropical three-term Pl\"ucker relation is satisfied.

We start with the case $a = b$, where $\bmu$ is just $(\mu)$.  We need check the validity of the three-term positive-tropical Grassmann-Pl\"ucker relations, say for an arbitrary choice of $S\in \binom{[n]}{a-2}$ and $\{i<j<k<\ell\} \subseteq [n]\setminus S$.
If $S$ is not independent in the matroid $\underline\mu$, then every term in the three-term relation involving $S$ and $ijk\ell$ is $\infty$, so we may assume $S$ is independent.
Let $\mathscr S$ be a maximal chain $S_1\subsetneq \cdots \subsetneq S_m$ of subsets of $[n]$ with the property that $S_{a-2}= S$ and $S_{a-1} = S\cup \{ijk\ell\}$.  Then, Proposition~\ref{prop:greedy} implies that for a vector $\mathbf v_{\mathscr S}$ in the relative interior of the cone 
	$\RR_{\geq 0}\{\be_{S_1}, \ldots, \be_{S_m}\}$, we have
\[
\operatorname{face}(P(\underline\mu), \mathbf v_{\mathscr S}) = P(\underline\mu^{\mathscr S}) \simeq 
	P(\underline\mu|S\cup ijk\ell / S).
\]
For the second identification, we have used that 
	\begin{enumerate}
		\item \label{en:1}
	the matroid polytope of a direct sum of matroids is the product of the matroid polytopes;
\item 	\label{en:2} with the exception of $(S_{a-2}, S_{a-1}) = (S,S\cup ijk\ell)$, all other minors of the matroid $\underline\mu$ corresponding to $(S_c, S_{c+1})$ in the chain have their polytopes being a point because $|S_{c+1}\setminus S_c| = 1$.
	\end{enumerate}
Since $S$ is assumed to be independent, the rank of the matroid minor $\underline\mu|S\cup ijk\ell /S$ is at most $2$.  If it is less than 2, then every term in the three-term relation involving $S$ and $ijk\ell$ is $\infty$, so let us now treat the case when the rank is exactly 2. 
For a basis $\widehat B$ of $\underline \mu|S\cup ijk\ell /S$, let $B$ be the basis of $\underline \mu$ such that the vertex $\be_B$ of $P(\underline \mu)$ corresponds to the vertex $\be_{\widehat B}$ of $P(\underline\mu|S\cup ijk\ell /S)$ under the identification above.
Identifying $[4] = \{1<2<3 <4\}$ with $\{i<j<k<\ell\}$,
we may thus consider ``restricting'' $\mu$ to the face $P(\underline\mu|S\cup ijk\ell/S)$ to obtain an element $\widehat\mu = \mu|S\cup ijk\ell/S \in \Dr_{2;4}$ defined by
\[
\widehat\mu(\widehat B) = \begin{cases}
\mu(B)& \text{ if $\widehat B$ a basis of $\underline\mu|S\cup ijk\ell /S$}\\
\infty& \text{otherwise}
\end{cases}
\qquad\text{for $ \textstyle \widehat B\in \binom{[4]}{2}$}.
\]
It is straightforward to check that for $\Dr_{2;4}$, the three-term positive-tropical Grassmann-Pl\"ucker relations are satisfied if and only if all 2-dimensional faces in the corresponding subdivision are positroid polytopes.
Since the faces of the subdivision $\mathcal D_{\widehat \mu}$ of $P(\underline\mu|S\cup ijk\ell /S)$ are a subset of the faces of the subdivision $\mathcal D_\mu$, we have that $\mu$ satisfies the three-term tropical-positive Grassmann-Pl\"ucker relation involving $ijk\ell$ and $S$.

\smallskip
Let us now treat the case $a<b$.
That the three-term Grassmann-Pl\"ucker relations are satisfied for every $\mu_i$ where $i = a, \dotsc, b$ follows from our previous case of $a = b$ once we show the following claim:
\begin{quote}
For a flag matroid $\underline\bmu$ with consecutive rank sequence $(a,\dotsc, b)$, if every face of $P(\underline\bmu)$ of dimension at most 2 is a flag positroid polytope, then the same holds for every constituent matroid, i.e.\ for every $c = a, \dotsc, b$, every face of $P(\underline\mu_c)$ of dimension at most 2 is a positroid polytope .
\end{quote}
To prove the claim, suppose for some $a\leq c\leq b$ that a 2-dimensional face $Q$ of $P(\underline\mu_c)$ is not a positroid polytope. 
	Our goal is to use $Q$ to
	find a $2$-dimensional face of $P(\underline{\bmu})$ 
	that is not a flag positroid polytope.
	By \cite[Theorem 3.9]{LPW}, 
	a 2-dimensional matroid polytope
	which is not a positroid polytope has vertices of the form
	$\be_{Sij},\be_{Sk\ell},\be_{Si\ell},\be_{Sjk}$, where 
	 $S\subset [n]$ with $|S| = c-2$ and $\{i<j<k<\ell\} \subset [n]\setminus S$; thus 
	 $Q = \operatorname{conv}
	(\be_{Sij},\be_{Sk\ell},\be_{Si\ell},\be_{Sjk})$ for such $\{S,i,j,k,l\}$\footnote{One may also deduce this independently of \cite{LPW} by using the argument given in the first third of this proof of \ref{eqvs:2faces}$\implies$\ref{eqvs:FlDr} concerning the $a = b$ case.}.  
	Note that this 2-face $Q$ is the Minkowski sum of $\be_S$ with the product $\operatorname{conv}(\be_i,\be_k) \times \operatorname{conv}(\be_j,\be_\ell)$.

    Let $\mathscr S$ be a maximal chain $S_1\subsetneq \dotsb \subsetneq S_m$ of subsets of $[n]$ with the property that $S_{c-1} = S$, $S_c = S\cup ik$, and $S_{c+1} = S\cup ijk\ell$. Then, Proposition~\ref{prop:greedy} implies that for a vector $v_{\mathscr S}$ in the relative interior of the cone 
	$\RR_{\geq 0}\{\be_{S_1}, \ldots, \be_{S_m}\}$, we have	
    
\[
\operatorname{face}(P(\underline\bmu), \mathbf v_{\mathscr S})  = P(\underline\bmu^{\mathscr S})  \simeq P(\underline\bmu|S_{c+1}/S_c) \times P(\underline\bmu|S_{c}/S_{c-1}).
\]
For the second identification, we have used that 
	\begin{enumerate}
		\item 
	the matroid polytope of a direct sum of matroids is the product of the matroid polytopes;
\item 	
	with the exception of $(S_{c-1},S_c)=(S,S\cup ik)$ and $(S_c, S_{c+1}) = (S\cup ik, S\cup ijkl)$, all other minors of the constituent matroids of $\underline\bmu$ corresponding to $(S_d, S_{d+1})$ in the chain have their polytopes being a point because $|S_{d+1}\setminus S_d| = 1$.
	\end{enumerate}
Note that the polytope $P(\underline\bmu|S_{c+1}/S_c) \times P(\underline\bmu|S_{c}/S_{c-1})$ is at most 2-dimensional since $\underline\bmu|S_{c+1}/S_c$ and $\underline\bmu|S_{c}/S_{c-1}$ are flag matroids on ground sets $\{j,\ell\}$ and $\{i,k\}$, respectively. The polytope has $Q$ as a Minkowski summand, and thus in particular is not a flag positroid polytope.

\smallskip
Lastly, we check the validity of the three-term positive-tropical incidence-Pl\"ucker relations, say for an arbitrary choice of $S\subset [n]$ with $a-1\leq |S|\leq b-2$ and $\{i<j<k\} \subseteq [n]\setminus S$.
We may assume that $S$ has rank $|S|$ in the matroid $\mu_{|S|+1}$, since otherwise every term in the three-term positive-tropical incidence relation is $\infty$, so that the relation is vacuously satisfied.
Let $\mathscr S$ be a maximal chain $S_1\subsetneq \cdots \subsetneq S_m$ of subsets of $[n]$ with the property that $S_c = S$ and $S_{c+1} = S\cup ijk$ for $c = |S|$.  Then, 
	Proposition~\ref{prop:greedy} implies that
for a vector $\mathbf v_{\mathscr S}$ in the relative interior of the cone 
	$\RR_{\geq 0}\{\be_{S_1}, \ldots, \be_{S_m}\}$, we have
\[
\operatorname{face}(P(\underline\bmu), \mathbf v_{\mathscr S}) = P(\underline\bmu^{\mathscr S}) \simeq 
	P(\underline\bmu|S\cup ijk / S).
\]
For the second identification, we have used that 
	\begin{enumerate}
		\item 
	the matroid polytope of a direct sum of matroids is the product of the matroid polytopes;
\item 	
	with the exception of $(S_c, S_{c+1}) = (S,S\cup ijk)$, all other minors of the constituent matroids of $\underline\bmu$ corresponding to $(S_d, S_{d+1})$ in the chain have their polytopes being a point because $|S_{d+1}\setminus S_d| = 1$.
	\end{enumerate}
Note that the 
	polytope $P(\underline\bmu|S\cup ijk / S)$
	is at most 2-dimensional since it is a flag matroid polytope on 3 elements.
	Similarly to the $a =b$ case, we may ``restrict'' $\underline\bmu$ to the face $P(\underline\bmu|S\cup ijk/S)$ to obtain an element  
	$\widehat\bmu = \bmu|S\cup ijk/S \in \FlDr_{\widehat{\boldsymbol r}; 3}$.
	We may assume that $\widehat{\boldsymbol r} = (1,2)$ since otherwise every term in the three-term incidence relation of the pair $(S,ijk)$ is $\infty$.
For $\operatorname{FlDr}_3$, it is straightforward to verify that the unique three-term positive-tropical incidence relation involving $S$ and $ijk$ is satisfied if and only if the subdivision $\mathcal D_{\widehat\bmu}$ consists only of flag positroid polytopes.  Since the faces of the subdivision $\mathcal D_{\widehat\bmu}$ are a subset of the faces of the subdivision $\mathcal D_{\bmu}$, we have that $\bmu$ satisfies the three-term incidence relation
	involving $S$ and $\{i,j,k\}$.
\end{proof}

\section{Three-term incidence relations}\label{sec:3}

\subsection{The proof of 
\ref{eqvs:3terms}$\iff$\ref{eqvs:FlDr} in \Cref{thm:eqvs}}\label{pf:3}

In the case that $a=b$ in \Cref{thm:eqvs}, the equivalence 
\ref{eqvs:3terms}$\iff$\ref{eqvs:FlDr} is the content of 
\Cref{prop:3terms}.  

To prove the implication when $a<b$, we will show the following key theorem.

\begin{thm}\label{thm:almost3term}
Suppose $\boldsymbol\mu = (\mu_1, \mu_2) \in \PP\left(\TT^{\binom{[n]}{r}}\right) \times \PP\left(\TT^{\binom{[n]}{r+1}}\right)$ satisfies every three-term positive-tropical incidence relation, and suppose that the support $\underline{\boldsymbol\mu}$ is a flag matroid.
Then, we have $\boldsymbol\mu \in \FlDr^{\geq 0}_{r,r+1;n}$ if either of the following (incomparable) conditions hold:
\begin{enumerate}[label = (\roman*)]
\item The support $\underline\bmu$ consists of uniform matroids.
\item Either $\mu_1\in \Dr^{\geq 0}_{r;n}$ or $\mu_2\in \Dr^{\geq 0}_{r+1;n}$.
\end{enumerate}
\end{thm}

\begin{proof}[Proof of \ref{eqvs:FlDr}$\iff$\ref{eqvs:3terms}]
By Proposition~\ref{prop:3terms}, the implication \ref{eqvs:FlDr}$\implies$\ref{eqvs:3terms} is immediate.
For the converse, since $\br$ consists of consecutive integers, if $\underline\bmu$ is a flag matroid and $\bmu$ satisfies every three-term positive-tropical incidence relation, then $\bmu$ also satisfies every three-term positive-tropical Grassmann-Pl\"ucker relation if either of the conditions (i) or (ii) of Theorem~\ref{thm:almost3term} is satisfied.
The hypothesis of \ref{eqvs:3terms} satisfies this, so $\bmu$ is an element of $\FlDr_{\br;n}^{\geq 0}$ by \Cref{prop:3terms}.
\end{proof}

The proof of \Cref{thm:almost3term} relies on the following technical lemma.

\begin{lem}\label{lem:EB}
Suppose $w\in \TT^{\binom{[5]}{2}}$ satisfies all three-term positive-tropical Grassmann-Pl\"ucker relations involving the element 5.  
	Suppose moreover that $w_{i5}<\infty$ for some $i=1,2,3,4$.
Then	$w \in \Dr_{2;5}^{\geq 0}$, i.e.\ $w$ also satisfies the three-term positive-tropical Grassmann-Pl\"ucker relation not involving 5.
\end{lem}

\begin{proof}
	The idea of the proof of \Cref{lem:EB} is that in the usual Grassmannian $\Gr_{2,5}$,
	if we can invert certain Pl\"ucker coordinates, then 
	we can write the three-term 
	Grassmann-Pl\"ucker relation not involving $5$
	as a linear combination of three of the other
	three-term Grassmann-Pl\"ucker relations.  In particular,
	we have the following identity, which is easy to verify.
	\begin{lem}\label{lem:ID}
    If $p_{25}\neq 0$ (respectively, $p_{35} \neq 0$) then
		$p_{13} p_{24}-p_{12}p_{34} - p_{14}p_{23}$
		can be written in the following ways.
	\begin{align*}
		&p_{13} p_{24}-p_{12}p_{34} - p_{14}p_{23}  \\
		&=(p_{13}p_{25} - p_{12}p_{35}-p_{15}p_{23})\frac{p_{24}}{p_{25}}
		-(p_{14}p_{25}-p_{12}p_{45}-p_{15}p_{24})\frac{p_{23}}{p_{25}}
		+(p_{24}p_{35}-p_{23}p_{45}-p_{25}p_{34})\frac{p_{12}}{p_{25}}\\
	&=(p_{13}p_{25} - p_{12}p_{35}-p_{15}p_{23})\frac{p_{34}}{p_{35}}
		-(p_{14}p_{35}-p_{13}p_{45}-p_{15}p_{34})\frac{p_{23}}{p_{35}}
		+(p_{24}p_{35}-p_{23}p_{45}-p_{25}p_{34})\frac{p_{13}}{p_{35}}.
	\end{align*}
	\end{lem}

  We next note that we can
	interpret the first (respectively, second) expression in \Cref{lem:ID} 
	tropically as long as $w_{25} <\infty$ (respectively,
	$w_{35}<\infty$).

\emph{Case 1: $w_{25} < \infty.$} Then we can make sense of the terms
	on the right hand side of the first expression of \Cref{lem:ID}
	tropically.  Since the three-term positive tropical Pl\"ucker
	relations involving $5$ hold, and $w_{25}<\infty$,
	we have 
	\begin{align*}
		w_{13}+w_{25}+w_{24}-w_{25}&=
		\min(w_{12}+w_{35}+w_{24}-w_{25},
		w_{15}+w_{23}+w_{24}-w_{25})\\
		w_{14}+w_{25}+w_{23}-w_{25}&=
		\min(w_{12}+w_{45}+w_{23}-w_{25},
		w_{15}+w_{24}+w_{23}-w_{25})\\
		w_{24}+w_{35}+w_{12}-w_{25} &=
		\min(w_{23}+w_{45}+w_{12}-w_{25},
		w_{25}+w_{34}+w_{12}-w_{25}).
	\end{align*}
We now simplify these expressions and underline
terms that agree, obtaining:
	\begin{align}
		w_{13}+w_{24}&= \label{eq:1}
		\min(\underline{w_{12}+w_{35}+w_{24}-w_{25}},
		\underline{\underline{w_{15}+w_{23}+w_{24}-w_{25}}})\\
		w_{14}+w_{23}&= \label{eq:2}
		\min(\uwave{w_{12}+w_{45}+w_{23}-w_{25}},
		\underline{\underline{w_{15}+w_{24}+w_{23}-w_{25}}}) \\
		\underline{w_{24}+w_{35}+w_{12}-w_{25}} &= \label{eq:3}
		\min(\uwave{w_{23}+w_{45}+w_{12}-w_{25}},
		w_{34}+w_{12}).
	\end{align}
There are now eight cases to consider, based on whether 
the minimum is achieved by the first or second term in each of 
\eqref{eq:1}, \eqref{eq:2}, \eqref{eq:3}.
All cases are straightforward. If the minimum is achieved by 
the first term in \eqref{eq:1} and the second term in \eqref{eq:3},
then we find that $w_{13}+w_{24}=w_{12}+w_{34} \leq w_{14}+w_{23}$.
In the other six cases, we find that 
 $w_{13}+w_{24}=
 w_{14}+w_{23} \leq 
 w_{12}+w_{34}$.
Therefore the positive tropical Pl\"ucker relation involving $1,2,3,4$
is satisfied.

\emph{Case 2: $w_{35} < \infty.$}
The argument for Case 2 is the same as for Case 1, except we 
	use the tropicalization of the second identity in \Cref{lem:ID}.

\emph{Case 3: $w_{25}=w_{35}=\infty.$}
	In this case, since $5$ is not a loop, either $w_{15}<\infty$ or 
	$w_{45}<\infty.$  Suppose that 
	$w_{15}<\infty$. 
Then	the positive tropical Pl\"ucker relations 
\begin{itemize}
	\item
		$w_{13}+w_{25} = 
		\min(w_{12}+w_{35}, w_{15}+w_{23})$ 
	\item 
		$w_{14}+w_{25} = 
		\min(w_{12}+w_{45}, w_{15}+w_{24})$
	\item 
		$w_{14}+w_{35} = 
		\min(w_{13}+w_{45}, w_{15}+w_{34})$
\end{itemize}
	imply that $w_{23}=w_{24}=w_{34}=\infty$, and hence
	the positive tropical Pl\"ucker relation involving 
	$1,2,3,4$ is satisfied.  
	The case where $w_{45}<\infty$ is similar. 
\end{proof}

For $w\in \TT^{\binom{[n]}{r}}$, define its dual $w^\perp \in \TT^{\binom{[n]}{n-r}}$ by $w^\perp(I) = w([n]\setminus I)$.  It is straightforward to verify that $w$ is an element of $\Dr_{r;n}$ (resp.\ $\Dr_{r;n}^{\geq 0}$) if and only if $w^\perp$ is an element of $\Dr_{n-r;n}$ (resp.\ $\Dr_{n-r;n}^{\geq 0}$).  This matroid duality gives the following dual formulation of \Cref{lem:EB}.

\begin{cor}\label{cor:EB}
Suppose $w\in \TT^{\binom{[5]}{3}}$ satisfies all three-term positive-tropical Grassmann-Pl\"ucker relations that contain a variable indexed by $S\in \binom{[5]}{3}$ with $5\notin S$.  If $\underline w$ is a matroid such that 5 is not a coloop, then $w\in \Dr_{3;5}^{\geq 0}$, i.e.\ $w$ also satisfies the three-term positive-tropical Grassmann-Pl\"ucker relation whose every variable contains 5 in its indexing subset.
\end{cor}

We are now ready to prove \Cref{thm:almost3term}.
We expect that the proof of \Cref{thm:almost3term} here adapts well to give an analogous statement for arbitrary perfect hyperfields.

\begin{proof}[Proof of \Cref{thm:almost3term}]
Given such $\boldsymbol\mu = (\mu_1, \mu_2) \in \PP\left(\TT^{\binom{[n]}{r}}\right) \times \PP\left(\TT^{\binom{[n]}{r+1}}\right)$, define $\widetilde{\boldsymbol\mu}\in \PP\left(\TT^{\binom{[n+1]}{r+1}}\right)$ by
\[
\widetilde{\boldsymbol\mu}(S) =
\begin{cases}
 \mu_1(S\setminus (n+1)) &\text{if $(n+1)\in S$}\\
 \mu_2(S) &\text{otherwise}.
\end{cases}
\]
Because $\underline{\boldsymbol \mu}$ is a flag matroid, we have that $\underline{\widetilde{\boldsymbol\mu}}$ is a matroid, with the element $(n+1)$ that is neither a loop nor a coloop.
We observe that $\widetilde{\boldsymbol\mu} \in \Dr_{r+1;n+1}^{\geq 0}$ if and only if $\boldsymbol\mu \in \FlDr^{\geq 0}_{r,r+1;n}$ because the validity of the three-term positive-tropical Grassmann-Pl\"ucker relations for $\widetilde{\boldsymbol\mu}$ is equivalent to the validity of both the three-term positive-tropical incidence relations and the three-term positive-tropical Grassmann-Pl\"ucker relations for $\boldsymbol\mu$.

We need to check that $\widetilde\bmu$ satisfies every three-term positive-tropical Grassmann-Pl\"ucker relation of type $(r+1;n+1)$.  Consider the three-term relation associated to the subset $S\subseteq [n+1]$ of cardinality $r-1$ and $\{i<j<k<\ell\}\subseteq [n+1]$ disjoint from $S$.  We have three cases:
\begin{itemize}
\item $\ell = n+1$.  In this case, erasing the index $n+1$ in the expression for the corresponding three-term Grassmann-Pl\"ucker relation yields a three-term incidence relation of type $(r,r+1;n)$, which is satisfied by our assumption on $\bmu$.
\item $(n+1) \in S$.  In this case, if $(n+1)$ is not a coloop in the minor $\widetilde\bmu|S\cup ijk\ell / (S\setminus(n+1))$, then applying \Cref{cor:EB} to $\widetilde\bmu|S\cup ijk\ell / (S\setminus(n+1))$ implies that the three-term Grassmann-Pl\"ucker relation is satisfied.
\item $(n+1) \notin S\cup ijk\ell$.   In this case, if $(n+1)$ is not a loop in the minor $\widetilde\bmu|S\cup ijk\ell(n+1) / S$, then applying \Cref{lem:EB} to $\widetilde\bmu|S\cup ijk\ell(n+1) / S$ implies that the three-term Grassmann-Pl\"ucker relation is satisfied.
\end{itemize}

	Under condition (i) of Theorem~\ref{thm:almost3term}, i.e.\ when the support $\underline\bmu$ consists of uniform matroids, the element $(n+1)$ is not a coloop in the minor $\widetilde\bmu|S\cup ijk\ell / (S\setminus(n+1))$, and is not a loop in the minor $\widetilde\bmu|S\cup ijk\ell(n+1) / S$.  Hence, both Corollary~6.4 and Lemma~6.2 apply respectively, and we conclude that in every case the three-term positive-tropical Grassmann-Pl\"ucker relation is satisfied.

	Now suppose condition (ii) of Theorem~\ref{thm:almost3term} holds.
We verify that in the cases where Corollary~6.4 or Lemma~6.2 do not apply, the relevant positive-tropical Grassmann-Pl\"ucker relation is satisfied.  Let us consider the third bullet point,
	and suppose that $(n+1)$ is a loop in the minor $\widetilde\bmu|S\cup ijk\ell(n+1) / S$, i.e.\ where Lemma~6.2 does not apply; the argument for the second bullet point is similar by matroid duality.
	In this case, since $(n+1)$ is not a loop in the matroid $\underline{\widetilde\bmu}$, $(n+1)$ belongs to the closure (also called \emph{span}) 
	in $\underline{\widetilde\bmu}$ of $S$. Since $S$ is also independent, there is an element $s\in S$ such that $(S\setminus s)\cup (n+1)$ is independent and has the same closure as $S$ in $\underline{\widetilde\bmu}$.
Let $S' = S\setminus s$.
For any $a,b\in \{i,j,k,\ell\}$, by our choice of $s\in S$, we have that $Sab$ is a basis of $\underline{\widetilde\bmu}$ if and only if $S'ab(n+1)$ is a basis of $\underline{\widetilde\bmu}$.
Moreover, for any $a,b,c \in \{i,j,k,\ell\}$ such that the values involved below are finite, we claim
\[
\widetilde\bmu(Sab) - \widetilde\bmu(Sac)  = \widetilde\bmu(S'ab(n+1)) - \widetilde\bmu(S'ac(n+1)).
\]
Note that using the definition of $\widetilde\bmu$, the above claim can be 
equivalently written as
$$
	\mu_2(Sab)-\mu_2(Sac) = \mu_1(S'ab)-\mu_1(S'ac).$$
From the claim, we conclude as follows.  Let $\overline\mu_1$ be the projection of $\mu_1$ to the coordinates labelled by $S'xy$ where $x\neq y \in \{i,j,k,\ell\}$, and let $\overline\mu_2$ be the projection of $\mu_2$ to the coordinates labelled by $Sxy$ where $x\neq y \in \{i,j,k,\ell\}$.  Then, as elements of $\PP(\TT^{\binom{\{i,j,k,\ell\}}{2}})$, the two tropical vectors $\overline\mu_1$ and $\overline\mu_2$ are equal.  Hence, the claim implies that if one of $\mu_1$ or $\mu_2$ satisfies the three-term Grassmann-Pl\"ucker relations on these coordinates, then so does the other.

The claim follows from the validity of three-term tropical incidence relations, which is implied by the validity of three-term positive-tropical incidence relations.  Namely, we have that the minimum is achieved at least twice in
\[
\{\mu_1(S'ab)+\mu_2(S'asc), \mu_1(S'as)+\mu_2(S'abc), \mu_1(S'ac)+\mu_2(S'asb)\},
\]
from which the claim follows because $Sa(n+1)$ is not a basis of $\underline{\widetilde\bmu}$, forcing $\mu_1(S'as) = \infty$.
\end{proof}

\section{Projections of positive Richardsons to positroids}\label{sec:projection}

One recurrent theme in our paper has been the utility of 
projecting a complete flag positroid
(equivalently, a positive Richardson)
to a positroid  (or a positroid cell).
This has come up in Rietsch's cell decomposition of a 
nonnegative (partial) flag variety,
in our proofs 
in \Cref{pf:1}, and in 
the expression of 
a Bruhat interval polytope as a Minkowski sum of positroid polytopes
in \Cref{rem:BIPsum}.
Positive Richardsons can be indexed by pairs $(u,v)$ of permutations with 
$u\leq v$.   Meanwhile, by work of Postnikov \cite{Pos}, 
positroid cells of $\Gr_{d,n}^{\geq 0}$ can be indexed by 
\emph{Grassmann necklaces}.
In this section we will give several concrete combinatorial recipes
for constructing the positroids obtained by projecting a (complete) flag positroid.
We will also discuss the problem of determining when a collection of positroids
can be identified with a (complete) flag positroid.

\subsection{Indexing sets for cells of $\Gr_{d,n}^{\geq 0}$}

As discussed in
\Cref{def:backgroundflag}, there are two equivalent ways of thinking about 
the positroid cell decomposition of $\Gr_{d,n}^{\geq 0}$:
$$\Gr_{d,n}^{\geq 0} = \bigsqcup S_{\mathcal{B}}^{>0} = \bigsqcup_{u,v} \pi(\mathcal{R}_{u,v}^{>0}).$$
In the union on the right, $\pi$ is the projection from $\Fl_n$ to $\Gr_{d,n}$,
and $u,v$ range over all 
permutations $u\leq v$  in $S_n$, such that $v$ is a minimal-length coset representative
of $W/W_d$, and $W_d = \langle s_1,\dots,s_{d-1},\hat{s}_d,s_{d+1},\dots,s_{n-1}\rangle.$
We write $W^d$ for the set of minimal-length coset representatives of $W/W_d$.
Recall that a \emph{descent} of a permutation $z$ is a position
$j$ such that $z(j)>z(j+1)$.  We have that $W^d$ is the subset of permutations in $S_n$
which have at most one descent, and if it exists, that descent must be in position $d$.

Even if $v\notin W^d$, the projection of $\mathcal{R}_{u,v}^{>0}$ to $\Gr_{d,n}^{\geq 0}$
is still a positroid, which we will characterize below.  We start by 
defining \emph{Grassmann necklaces} \cite{Pos}.

\begin{defn}
Let $\I=(I_1,\ldots, I_n)$ be a sequence of subsets of ${[n] \choose d}$. 
We say $\I$ is a \textbf{Grassmann necklace} of \emph{type $(d,n)$} if the following holds:
\begin{itemize}
    \item If $i\in I_i$, then $I_{i+1}=(I_i\setminus i)\cup j$ for some $j\in[n]$.
    \item If $i\notin I_i$, then $I_{i+1}=I_i$. 
\end{itemize}
\end{defn}

In order to define the bijection between these Grassmann necklaces and positroids, we need to define the \textit{$i$-Gale order} on $\binom{[n]}{d}$. 

\begin{defn}\label{def:shiftedorder}
We write $<_i$ for the following \emph{shifted linear order} on $[n]$.
\begin{equation*}
    i<_ii+1<_i\ldots <_i n<_i1<_i\ldots <_i i-1.
\end{equation*}
	We also define the \emph{$i$-Gale order} on $d$-element subsets by 
	setting $$\{a_1 <_i \dots <_i a_d\} \leq_i \{b_1 <_i \dots <_i b_d\}$$
	if and only if $a_{\ell} \leq_i b_{\ell}$ for all $1\leq \ell \leq d$.
\end{defn}

Given a positroid $M$, 
we define a sequence $\mathcal{I}_M = (I_1,\dots,I_n)$ of subsets of $[n]$  
by letting ${I}_i$ be the minimal basis of $M$ in the $i$-Gale order. 
	The following result is from \cite[Theorem 17.1]{Pos}.

\begin{prop}
	\label{prop:postonecklace}
For any positroid $M$, $\I_M$ is a Grassmann necklace. 
The map $M \mapsto \I_M$ gives a bijection between positroids of rank $d$ on $[n]$ and 
	Grassmann necklaces of type $(d,n)$.
\end{prop}

\subsection{Projecting positive Richardsons to positroids}\label{sec:projecting}

In this section we will give several descriptions 
of the constituent positroids appearing in 
 a complete flag positroid (that is, 
a flag matroid represented by a positive Richardson). We start by reviewing 
a cryptomorphic definition of flag matroid, based on 
 \cite[Sections 1.7-1.11]{BGW03}. 

A \emph{flag} 
$F=F_1 \subset F_2 \subset \dots \subset F_k$ on $[n]$ is an increasing
sequence of finite subsets of $[n]$. 
A \emph{flag matroid} is a collection $\mathcal{F}$ of flags satisfying the \emph{Maximality Property}.
Recall that $e_S$ denotes the $01$ indicator vector in $\RR^n$ associated to a subset $S\subset [n]$.
For a flag 
$F=F_1 \subset F_2 \subset \dots \subset F_k$ we let 
$e_F = e_{F_1} + \dots + e_{F_k}$.  In this language, the 
flag matroid polytope of $\mathcal{F}$ is 
$P_{\mathcal{F}} = \Conv\{e_F \ \vert \ F\in \mathcal{F}\}$,
whose vertices are precisely the points $e_F$ for $F\in \mathcal{F}$.

In the complete flag case, each point $e_F$ is a 
permutation vector $(z(1),\dots,z(n))$ for some $z\in S_n$.
Note that we can read off $z:=z(F)$ from $F$ by setting 
$z(i)=j$, where $j$ is the unique element of $F_i \setminus F_{i-1}$.

Given $u\leq v$ in Bruhat order,
we define the \emph{Bruhat interval flag matroid} $\mathcal{F}_{u,v}$ 
to be the complete flag matroid
whose flags  are precisely
$$\{z([1]) \subset z([2]) \subset \dots \subset z([n])\} 
\text{ for }u\leq z \leq v,$$
where $[i]$ denotes $\{1,2,\dots,i\}$ and $z([i])$ 
denotes $\{z(1),\dots,z(i)\}$.
Then by the above discussion,
 the (twisted) Bruhat interval polytope 
$$\tilde{P}_{u,v} = \Conv\{(n+1-z^{-1}(1),n+1-z^{-1}(2),\dots,n+1-z^{-1}(n)) \ \vert \ u \leq z \leq v\}$$
is the
flag matroid polytope of the Bruhat interval flag matroid $\mathcal{F}_{u,v}$.

This observation leads naturally to the following definition.
\begin{defn}\label{def:envelope}
Consider a complete flag matroid $\mathcal{F}$ on $[n]$,
which we identify with a collection $\mathcal{S}$ of 
permutations on $[n]$. By the Maximality Property 
\cite[Section 1.7.2]{BGW03} and its relation to the
tableau criterion for Bruhat order \cite[Theorem 5.17.3]{BGW03},
	$\mathcal{S}$ contains a unique permutation $u$ (respectively,
	$v$) which is 
	minimal (respectively, maximal)
	in Bruhat order among all elements of $\mathcal{S}$.
We say that $\mathcal{F}_{u,v}$ is the 
\emph{Bruhat interval envelope} of $\mathcal{F}$.
\end{defn}
It follows from \Cref{def:envelope} that  the
Bruhat interval envelope of a complete flag matroid $\mathcal{F}$
contains $\mathcal{F}$;
however, in general this inclusion
is strict.  It is an equality precisely when 
$\mathcal{F}$ is a Bruhat interval flag matroid.

Recall that if $F=(F_1,\dots, F_n)$ and $G=(G_1,\dots,G_n)$ are flags,
we say that 
$F$ is \emph{less than or equal to $G$ in the $\leq_j$ Gale order}
(and write 
$F \leq_j G$)
if and only if $F_i \leq_j G_i$ for all $1 \leq i \leq n$.
(We also talk about the ``usual'' Gale order with respect to the 
 total order $1 < 2 < \dots < n$.)
The \emph{Maximality Property} for flag matroids implies that for any flag matroid 
$\mathcal{F}$, there is always a unique element which is maximal
(and a unique element which is minimal) with respect to  $\leq_j$.

We now give a Grassmann necklace characterization of the positroid constituents
of a complete flag positroid, which follows from the previous discussion 
	plus \Cref{prop:postonecklace}.
\begin{prop}\label{prop:neck}
Consider a complete flag positroid 
$\M = (M_1, \ldots, M_n)$
on $[n]$, that is, 
the flag positroid associated to any point of $\mathcal{R}_{u,v}^{>0}$,
for some $u\leq v$. For each $1 \leq j \leq n$, let 
	$z^{(j)}$ be the Gale-minimal permutation with respect to $\leq_j$ in the 
	interval $[u,v]$.  Then the Grassmann necklace of the positroid $M_j$
	is 
	$(z^{(1)}([j]),z^{(2)}([j]),\dots, z^{(n)}([j]))$.
\end{prop}

\begin{eg}
Consider the flag positroid associated to a point of 
	$\mathcal{R}_{u,v}^{>0}$, where $u=(1,2,4,3)$ and $v=(4,2,1,3)$
	(which we abbreviate as 1243 and 4213).
	The interval $[u,v]$ consists of 
	$$[u,v]=\{1243, 1423, 2143, 2413, 4123, 4213\}.$$

	We now use \Cref{prop:neck}, and find that 
	 the Gale-minimal permutations of $[u,v]$ with respect to 
	$\leq_1, \leq_2, \leq_3, \leq_4$ are 
	$1243, 2413, 4123, 4123$.  Therefore the Grassmann necklaces for the constituent
	positroids
	$M_1, M_2, M_3$ and $M_4$ are
	$(1,2,4,4)$, $(12, 24, 14, 14)$, $(124, 124, 124, 124)$,
	and $(1234, 1234, 1234, 1234)$.  

	Alternatively, 
	we can read off the flags in the flag positroid from the permutations in $[u,v]$,
	obtaining the flags
	$$\{1 \subset 12 \subset 124, 1 \subset 14 \subset 124 , 
	2 \subset 12 \subset 124, 2 \subset 24 \subset 124 , 
	4 \subset 14 \subset 124,
	4 \subset 24 \subset 124\}.$$
	(Note that for brevity, we have omitted the subset 1234 from the end of each flag above.)
	We can now read off the bases of $M_1, M_2, M_3, M_4$ from the flags,
	obtaining 
	$\{1,2,4\}$, $\{12, 14, 24\}$, $\{124\}$, and $\{1234\}$.  We can then
	directly calculate the Grassmann necklaces from these sets of bases, 
	getting the same answer as above.

If we compute the Minkowski sum of the positroids $M_1, M_2, M_3, M_4$ above, 
	we obtain the twisted Bruhat interval polytope
	$\tilde{P}_{1243,4213}=P_{2314, 4312}$, whose 
	vertices are $$\{(4,3,1,2),(4,2,1,3),(3,4,1,2), (3,2,1,4), (2,4,1,3), (2,3,1,4)\},$$ as noted in \Cref{rem:BIPsum}.  
\end{eg}

The following result gives an alternative description of the constituent positroids
of a complete flag matroid, this time in terms of bases.
\begin{lem}[{\cite[Lemma 3.11]{KW15} and \cite[Theorem 1.4]{BW}}] 
	\label{lem:flagbases}
Consider a complete flag positroid, that is,
a flag matroid represented by a point of a positive Richardson
$\mathcal{R}_{u,v}^{>0}$, where $u,v\in S_n$ and 
	$u\leq v$ in Bruhat order. Choose $1\leq d \leq n$.
Let $\pi$ denote the projection from $\Fl_n$ to $\Gr_{d,n}$.
Then the bases of the rank $d$ 
positroid represented by $\pi(\mathcal{R}_{u,v}^{>0})$
	are $\{ z([d])  \ \vert \ u \leq z \leq v \}.$
\end{lem}

Finally, we remark that \cite[Remark 5.24]{BK} gives
yet another description of the constituent positroids of a complete 
flag positroid, this time in terms of pairs of permutations.

\subsection{Characterizing when two adjacent-rank
positroids form an oriented 
matroid quotient}

We have discussed how to compute the projection of a
complete flag positroid to a positroid.   Moreover, it is well-known
that every positroid is the projection of a complete flag positroid.
In this section we will give a criterion for determining 
 when two positroids 
$M_i$ and $M_{i+1}$ on $[n]$
of ranks $i$ and $i+1$
 can be obtained as the projection of a complete
flag positroid
(see \Cref{thm:necklacequotient}).

We recall the definition of oriented matroid quotient
in the setting at hand.
\begin{defn} 
We say that two positroids $M_i$ and $M_{i+1}$ on $[n]$ 
of ranks $i$ and $i+1$ 
\emph{form an oriented matroid quotient} if 
$(M_i,M_{i+1})$ is an oriented flag matroid.
\end{defn}

The following statement is a direct consequence of \Cref{lem:extend3}.
\begin{prop}
Let  $M_i$ and $M_{i+1}$ be positroids on $[n]$ 
of ranks $i$ and $i+1$.  Then there is a complete
flag positroid with $M_i$ and $M_{i+1}$ as constituents
if and only if $(M_i, M_{i+1})$  form an oriented  matroid quotient.
\end{prop}

\begin{prop}\label{prop:compatibilitycondition}
Suppose that $(M_1,\dots,M_n)$ is a sequence of positroids
of ranks $1,2,\dots,n$ on $[n]$, such that each 
pair $M_i$ and $M_{i+1}$ forms an oriented matroid quotient. Then 
$(M_1,\dots,M_n)$ is a complete flag positroid.
Moreover, it is realized by a point
of the positive Richardson $\mathcal{R}_{u,v}^{>0}$,
where we can explicitly construct $u$ and $v$ as follows:
	\begin{itemize}
		\item 
Let $B_1^{\min},\dots,B_n^{\min}$ 
(respectively, $B_1^{\max},\dots,B_n^{\max}$) be the 
bases of $M_1,\dots,M_n$ which are minimal (maximal) with respect 
			to the usual Gale ordering.  Then 
			$u,v\in S_n$ are defined by
			\begin{equation*}
				u(i)=
			B_i^{\min}\setminus B_{i-1}^{\min} \ 
				\text{ and }\	v(i)=
			B_i^{\max}\setminus B_{i-1}^{\max}.
			\end{equation*}
	\end{itemize}
\end{prop}
\begin{proof}
As in 
\Cref{lem:signembed}, we identify each positroid $M_i$
	with the image $t(\chi_i)$ of its chirotope $\chi_i$;
we have that 
 $t(\chi_i)$ lies in $\Dr_{i;n}^{\geq 0}$.  The fact that each
pair $M_i, M_{i+1}$ forms an oriented matroid quotient
	means that $(t(\chi_1),\dots,t(\chi_n))$
	satisfies all three-term incidence-Pl\"ucker relations,
	and hence 
	 $(t(\chi_1),\dots,t(\chi_n)) \in 
	 \FlDr_n^{\geq 0}$.
	 Since $\FlDr_n^{\geq 0} = \TrFl_n^{\geq 0}$, 
	 we have proved that $(M_1,\dots,M_n)$ is a complete flag positroid.

To prove the characterization of $u$ and $v$, we use 
	\Cref{lem:flagbases}.  In particular, it follows from 
	\Cref{lem:flagbases} and	the 
	Tableaux Criterion for Bruhat order that 
	the Gale-minimal and Gale-maximal
	bases of the  rank $d$ positroid 
	$\pi(\mathcal{R}_{u,v}^{>0})$ are $u([d])$ and $v([d])$.
The result now follows.
\end{proof}

As we've seen in \Cref{eg:notreal} it is a  subtle question 
to 
determine whether a pair of positroids $M_1$ and $M_2$ of ranks $r$ and $r+1$
form an oriented matroid quotient.
One way is to 
construct an $n$ by $r+1$ matrix such that the minor in rows $1,\ldots,r$ and columns $I$ is non-zero if and only if $I$ is a basis of $M_1$ while the maximal minor in rows $1,\ldots,r+1$ and columns $J$ is non-zero if and only if $J$ is a basis of $M_2$. Another way 
is to check the three-term relations over the signed tropical hyperfield, as in \Cref{prop:3terms}. 
We do not have an efficient way to do either of these things.  Instead, in 
\Cref{thm:necklacequotient},
 we will give an algorithmic, combinatorial way to verify 
 whether $M_1$ and $M_2$ form an oriented matroid quotient.

{\bf Construction \# 1.}
Given two positroids $M_1$ and $M_2$ on the ground set $[n]$ of ranks $r$ and $r+1$, respectively, which form a positively oriented matroid quotient, we construct
a positroid $M:=M(M_1,M_2)$ of rank $r+1$ on the ground set $[n+1]$ where $n+1$ is neither a loop nor a coloop. The bases of $M$ are precisely 
\begin{equation*}
    \mathcal{B}(M)=\mathcal{B}(M_2)\cup\{B\cup \{n+1\}\mid B\in \mathcal{B}(M_1)\}.
\end{equation*}

{\bf Construction \#2.}
Conversely, given a rank $r$ positroid $M$ on  ground set $[n+1]$, where $(n+1)$ is neither a loop nor  coloop, we construct two positroids $M_1:=M_1(M)$ and $M_2:=M_2(M)$ which form a positively
oriented matroid quotient, as follows.
Let $\tilde{A}$ be a matrix realizing $M$; 
therefore its Pl\"ucker coordinates are nonnegative.
We apply row operations to rewrite $\tilde{A}$ in the form 
\begin{equation*}
			A=\begin{bmatrix}
				\hspace{1cm} &\vline& 0\\
				 \hspace{1cm} A'\ \  \ \ \hspace{2cm}&\vline& 0\\
				\hspace{1cm} &\vline& 0\\
				 \hline
			*\;\;\;\;*\;\;	\;\;* \;\;\;\;\;\;\cdots\;\; \;\;\;\; *  &\vline &1
			\end{bmatrix}.
\end{equation*}
Let $M_1$ denote the matroid on $[n]$ realized by $A'$ and let $M_2$ denote the matroid on $[n]$
realized by
$A'$ together with the row of $*$'s below it.
Then $M_1$ and $M_2$ are both positroids (since the Pl\"ucker coordinates of $A'$ and $A$ are all
nonnegative), and they form a positively oriented quotient.
Moreover, it is clear that $M_1=M\setminus (n+1)$ and $M_2=M / (n+1)$.

The idea of our algorithm is to translate Constructions $\# 1$ and $\# 2$ into operations on 
Grassmann necklaces, so that Construction $\# 1$ is well-defined even if $M_1$ and $M_2$
fail to form a positively oriented quotient.
Clearly if we start with positroids $M_1$ and $M_2$ forming a positively oriented matroid quotient,
then Construction $\# 1$ followed by $\# 2$ is the identity map. 
Conversely, if  
 Construction $\# 1$ followed by $\# 2$ \emph{is} the identity map, then since Construction
 $\# 2$ always outputs a positively oriented matroid quotient, we must have started with positroids
 forming a positively oriented matroid quotient.

We let $\min_i\{S_1,\cdots, S_k\}$ denote the minimum of the sets $S_1,\cdots, S_k$ in the $\leq_i$ order.

\begin{prop}\label{prop:flagtopos} Let $M_1$ and $M_2$ be positroids of consecutive ranks which form a positively oriented quotient. Let $\mathcal{I}_{M_j}=(I_1^{(j)},\ldots, I_n^{(j)})$ be the Grassmann necklace of $M_j$ for $j=1,2$. 
	Define
\begin{equation*}
    J_{i}= \begin{cases}
	    I_1^{(2)}, & \text{ for }i=1\\
	    \min_i\{I_i^{(1)}\cup \{n+1\}, I_i^{(2)}\}, & \text{ for }2\leq i\leq n\\
	    I_1^{(1)}\cup \{n+1\}, &\text{ for }i=n+1.
\end{cases} .   
\end{equation*}
Then $\mathcal{J}=(J_1,\dots,J_{n+1})$
	is the Grassmann necklace of the positroid $M=M(M_1,M_2)$ on $[n+1]$ 
whose bases are precisely 
\begin{equation*}
    \mathcal{B}(M)=\mathcal{B}(M_2)\cup\{B\cup \{n+1\}\mid B\in \mathcal{B}(M_1)\}.
\end{equation*}
\end{prop}

\begin{proof}

It suffices to show that each basis of $M$ is $i$-Gale greater than $J^{(i)}$ for all $i\in [n+1]$. One also need to check that the $J^{(i)}$ are in fact bases of $M$ but this is clear by definition.

 Note that the $\leq_i$ minimal flag  of a flag matroid consists of the $\leq_i$ minimal bases of each of its constituent matroids \cite[Corollary 7.2.1]{BGW03}. Thus, $I_t^{(1)}\subset I_t^{(2)}$ for each $t\in [n]$.

 First, let $S\subset [n]$ be a basis of $M_2$. For $ i\in [n]$, we have $S\geq_i I_i^{(2)}\geq_i J_{i}$. Since neither $S$ nor $I_i^{(2)}$ contain $n+1$, $S\geq_{n+1}I_1^{(2)}$. By our earlier observation, $I_1^{(2)}=I_1^{(1)}\cup \{a\}$ for some $a\in [n]$. Thus, $I_1^{(2)}\geq_{n+1}I_1^{(1)}\cup \{n+1\}$. We conclude that $S\geq_iJ_i$ for all $i\in[n+1]$.

Next, consider $S\cup \{n+1\}$ for $S$ a basis of $M_1$. For $2\leq i\leq n$, we have $S\cup\{n+1\}\geq_i I_i^{(1)}\cup \{n+1\}\geq_i J_{i}$. Since neither $S$ nor $I_i^{(1)}$ contain $n+1$, we have $S\geq_1 I_1^{(1)}$ and $S\cup \{n+1\}\geq_{n+1}I_1^{(1)}\cup \{n+1\}=J_{n+1}$. Since $I_1^{(2)}=I_1^{(1)}\cup \{a\}$, we have $I_1^{(1)}\cup \{n+1\}\geq_{1}I_1^{(2)}=J_1$. We conclude that $S\cup\{n+1\}\geq_i J_i$ for all $i\in[n+1]$. 
\end{proof}

If $M_1$ and $M_2$ form a positively oriented quotient, we should obtain them 
from the positroid $M=M(M_1,M_2)$, constructed as in \Cref{prop:flagtopos}, by deleting and contracting $n+1$.  The following result explains how these operations affect Grassmann necklaces.

\begin{prop}\label{prop:contractdeletepos}\cite[Proposition 7 and Lemma 9]{Oh}
Let $M$ be a positroid on $[n+1]$ such that $n+1$ is neither a loop nor a coloop,  with Grassmann necklace $(J_{i})_{i=1}^{n+1}$. Then the Grassmann necklaces $\left(K_1^{(1)},\cdots ,K_n^{(1)}\right)$ and $\left(K_1^{(2)},\cdots ,K_n^{(2)}\right)$ of $M_1=M/(n+1)$ and $M_2=M\setminus (n+1)$, are  as follows:

\begin{align*}    
K_i^{(1)}=&
\begin{cases}
J_{i}\setminus \{n+1\}, &n+1\in J_{i}\\
J_{i}\setminus \{\max_i (J_{i}\setminus J_{n+1})\}, & n+1\notin J_{i}
\end{cases}
\\
K_i^{(2)}=&\begin{cases}(J_{i}\setminus \{n+1\})\cup\{\min_i (J_{n+1}\setminus J_{i})\}, & n+1\in J_{i}\\
J_{i}, &n+1\notin J_{i}
\end{cases}.
\end{align*}
\end{prop}

Taken together, the last two results yield a recipe for verifying whether two positroids, 
given in terms of their Grassmann necklaces, form a positively oriented quotient.
First apply the construction of \Cref{prop:flagtopos}. If that yields a Grassmann necklace,  
apply \Cref{prop:contractdeletepos} and see if that yields the original Grassmann necklaces. 
If so, the two Grassmann necklaces form a positively oriented quotient.

Our next goal is to streamline this recipe.
Let $\mathcal{I}^{(1)}=\left(I^{(1)}_1,\ldots, I^{(1)}_n\right)$ and $\mathcal{I}^{(2)}=\left(I^{(2)}_1,\ldots, I^{(2)}_n\right)$ be Grassmann necklaces of positroids of ranks $r$ and $r+1$, respectively. 
Note that a necessary condition for the positroids corresponding to $\mathcal{I}^{(1)}$ and $\mathcal{I}^{(2)}$ forming
a positively oriented quotient is  that $I^{(1)}_i\subset I^{(2)}_i$ for all $i\in[n]$. 
Now, we define a subset $S$ as follows: For each $i$, if $I^{(1)}_i\cup \{n+1\}<_i I^{(2)}_i$, let $i\in S$. Since $I^{(2)}_i=I^{(1)}_i\cup a$ for some $a\in [n]$, this is as simple as checking whether $a<_i n+1$. If the positroids corresponding to $\mathcal{I}^{(1)}$ and $\mathcal{I}^{(2)}$ form a positively oriented quotient, applying \Cref{prop:flagtopos} and then \Cref{prop:contractdeletepos} should leave them unchanged. It is straightforward to see that $i\in S$ if and only if $n+1\in J_i$ in \Cref{prop:contractdeletepos}. In particular, since $\mathcal{J}$ is a Grassmann necklace, $S$ must either be an interval of the form $[d,n]$, or empty. 

Next we claim that, once we verify that $S$ is an interval of the form $[d,n]$ or is empty, then it follows automatically that $\mathcal{J}$, as constructed in \Cref{prop:flagtopos}, is a Grassmann necklace. 

\begin{lem}\label{lem:GrassmannNecklace}
Let $\mathcal{I}^{(1)}=\left(I^{(1)}_1,\ldots, I^{(1)}_n\right)$ and $\mathcal{I}^{(2)}=\left(I^{(2)}_1,\ldots, I^{(2)}_n\right)$ be Grassmann necklaces of types $(r,n)$ and $(r+1,n)$, respectively. Construct $\mathcal{J}=(J_1,\ldots,J_{n+1})$ as in \Cref{prop:flagtopos}. Let $S=\left\{i\in [n]|I^{(1)}_i\cup (n+1)<_i I^{(2)}_i\right\}$. If $S=[d,n]$ for some $d\leq n$ or $S=\emptyset$, then $\mathcal{J}$ is a Grassmann necklace. 
\end{lem}

\begin{proof}
It is clear from the definition that $\mathcal{J}$ satisfies the Grassmann necklace condition for each pair of consecutive sets $J_i$ and $J_{i+1}$ except for when $i=k-1$, $i=n$ and $i=n+1$ (where we label sets cyclically so that $J_{n+2}=J_1$). 

If $S\neq \emptyset$, then $J_n=I^{(1)}_n\cup \{n+1\}$. This makes it clear that the Grassmann necklace condition holds for $J_n$ and $J_{n+1}$. Also, sing the fact that $I^{(1)}_i\subset I^{(2)}_i$ for all $i$, it is not hard to verify the Grassmann necklace condition for $J_{n+1}$ and $J_1$.

This leaves us to check the condition for $J_{k-1}$ and $J_k$. In this case, $J_{k-1}=I^{(2)}_{k-1}$ and $J_k=I^{(1)}_k\cup\{n+1\}$. Our goal is to show that $J_{k}=(J_{k-1}\setminus \{k-1\})\cup \{a\}$ for some $a\in [n+1]$. It is immediately obvious that we necessarily have $a=n+1$. Thus, we are left to show that $I^{(1)}_k\cup\{n+1\}=(I^{(2)}_{k-1}\setminus \{k-1\})\cup\{n+1\}$, or that $I^{(1)}_k=I^{(2)}_{k-1}\setminus \{k-1\}$. 

Let $a_i$ be defined by $I^{(1)}_i=(I^{(1)}_{i-1}\setminus \{i-1\})\cup \{a_i\}$, let $b_i$ be defined by $I^{(2)}_i=(I^{(2)}_{i-1}\setminus \{i-1\})\cup \{b_i\}$ and let $c_i$ be defined by $I^{(2)}_i=I^{(1)}_i\cup \{c_i\}$. We observe that $I^{(1)}_k=(I^{(1)}_{k-1}\setminus \{k-1\})\cup \{a_{k}\}=(I^{(2)}_{k-1}\setminus \{c_{k-1},k-1\})\cup \{a_{k}\}$. Also, $I^{(1)}_k=I^{(2)}_{k}\setminus \{c_{k}\}=(I^{(2)}_{k-1}\setminus \{c_{k},k-1\})\cup \{b_{k}\}$. Comparing these two equalities, we conclude that either $a_{k}=c_{k-1}$ and $b_{k}=c_k$, or $c_{k-1}=c_k$ and $a_{k}=b_{k}$. The first case is what we want to prove, so let us show by contradiction that the second case cannot occur.

Assume $c_k=c_{k-1}$ and $a_k=b_k$. By assumption, $I^{(1)}_{k-1}\cup \{c_{k-1}\}=I^{(2)}_{k-1}<_{k-1} I^{(1)}_{k-1}\cup \{n+1\}$ and $I^{(1)}_k\cup \{n+1\}<_kI^{(2)}_k=I^{(1)}_k\cup \{c_k\}$.  Thus, $c_{k-1}<_{k-1} n+1$ and $c_k >_k n+1$. Since $c_k=c_{k-1}$, this means they are both equal to $k-1$. However, if $c_k=k-1$, then $M_2$ has $k-1$ as a coloop. it follows that $b_{k}=k-1$, which means $a_k=k-1$ as well. Thus, in this case, $I_k^{(1)}=I_{k-1}^{(1)}=I_{k-1}^{(2)}\setminus \{k-1\}$, as desired. 

Finally, if $A=\emptyset$, we can check that the Grassmann necklace condition holds for $J_{n+1}$ and $J_1$ as before. The we are just left to verify this condition for $J_n$ and $J_{n+1}$. We can apply the same logic but with $J_{k-1}$ replaced by $J_{n}=I_n^{(2)}$ and $J_{k}$ replaced by $J_{n+1}=I^{(1)}\cup \{n+1\}$. Specifically, we find $I_{1}^{(1)}=(I_{n}^{(2)}\setminus \{c_n, n\})\cup \{a_1\}=I_{n}^{(2)}\setminus \{c_1, n\})\cup \{b_1\}$. We then must show that it is impossible for $c_1=c_n$ and $a_1=b_1$. However, $I_{n}^{(1)}\cup \{c_n\}=I_n^{(2)}<_n I_{n}^{(1)}\cup \{n+1\}$. Moreover, it is always true that $I_{1}^{(1)}\cup \{n+1\}<_{n+1}I_{1}^{(1)}\cup \{c_1\}=I_{2}^{(1)}$. Using $c_1=c_n$, we then find $c_n<_n(n+1)$ and $c_n>_{n+1}(n+1)$ which means that $c_1=c_n=n$ and we can conclude as in the previous paragraph. 
\end{proof}

Combining \Cref{prop:flagtopos}, \Cref{prop:contractdeletepos} and \Cref{lem:GrassmannNecklace}, we obtain the following: 

\begin{thm}\label{thm:necklacequotient}
	Fix positroids $M_1$ and $M_2$ on $[n]$ of ranks $r$ and $r+1$, respectively. Let $\mathcal{I}=\mathcal{I}_{M_1}=\left(I_1,\ldots, I_n\right)$ and $\mathcal{J}=\mathcal{I}_{M_2}=\left(J_1,\ldots, J_n\right)$ be their Grassmann necklaces. We now set $S=\left\{i \in [n] \mid I_i\cup \{n+1\}\leq_i J_i\right\}$,
	where $\leq_i$ denotes the $\leq_i$ Gale order on $[n+1]$.  Define $a_i=\max_i\left(J_i\setminus I_1\right)$ and $b_i=\min_i\left(I_1\setminus I_i\right)$. Then $M_1$ and $M_2$ form a positively oriented quotient if and only if the following conditions hold: 

\begin{enumerate}
    \item For $i\in [n]$, $I_i\subset J_i$.
    \item $S$ is an interval of the form $[d,n]$ or $S=\emptyset$.
    \item For $i\notin S$, $I_i=J_i
    \setminus \{a_i\}$.
    \item For $i\in S$, $J_i=I_i\cup \{b_i\}$.
\end{enumerate}
\end{thm}

\begin{proof}
First, suppose that we have a positively oriented quotient. As explained earlier, the first two conditions always hold for positively oriented quotients. We know that 
	applying the constructions of \Cref{prop:flagtopos} and \Cref{prop:contractdeletepos} in sequence should preserve our positively oriented quotient. Observing what conditions this imposes on the constituent Grassmann necklaces yields conditions $3$ and $4$. 

	Conversely, if the conditions in the theorem statement hold, then by \Cref{lem:GrassmannNecklace}, applying the construction of \Cref{prop:flagtopos} to $\mathcal{I}$ and $\mathcal{J}$ yields another Grassmann necklace $\mathcal{K}$ on $[n+1]$ such that $n+1$ is neither a loop nor a coloop of the positroid corresponding to $\mathcal{K}$. Then, conditions $3$ and $4$ guarantee that applying the construction of \Cref{prop:contractdeletepos} to $\mathcal{K}$ will recover $\mathcal{I}$ and $\mathcal{J}$. The result of applying \Cref{prop:contractdeletepos} to the Grassmann necklace of a positroid $M$ with $n+1$ neither a loop nor a coloop is the pair of Grassmann necklaces corresponding to $M/(n+1)$ and $M \setminus (n+1)$, which form a positively oriented quotient.  
\end{proof}

\begin{eg}
Let $\mathcal{I}=(123,235,356,456,561,613)$ and $\mathcal{J}=(1235,2356,3456,4562,5612,6123)$. Then $A=\{4,5,6\}$ is an interval with upper endpoint $n=6$. Note that $a_1=5$, $a_2=6$ and $a_3=6$, while $b_4=1$, $b_5=2$ and $b_6=2$. The positroids with these Grassmann necklaces do not form a positively oriented quotient since it is false that $I_3=J_3\setminus \{a_3\}$. 

However, if we start with Grassmann necklaces 
	$\mathcal{I}=(123,235,345,456,561,613)$ and $\mathcal{J}=(1235,2356,3456,4562,5612,6123)$, then the values of the $a_i$ and $b_i$ are unchanged. It is straightforward to verify that the conditions of \Cref{thm:necklacequotient} hold and so the positroids corresponding to $\mathcal{I}$ and $\mathcal{J}$ do in fact form a positively oriented quotient. 
\end{eg}

We now have a tool that allows us to recognize flag positroids in consecutive ranks 
without finding a realization or certifying the incidence relations over the signed hyperfield.

\begin{cor}\label{cor:recgonizeflags}
Suppose $(M_a,M_{a+1},\ldots, M_b)$ is a sequence of positroids of ranks $a, a+1,\dots,b$.
 Then $(M_a,M_{a+1},\ldots, M_b)$ is a flag positroid if and only if for $a\leq i<b$, the pair of positroids $(M_i, M_{i+1})$ satisfy the conditions of \Cref{thm:necklacequotient}.
\end{cor}

\begin{proof}
By \Cref{prop:compatibilitycondition}, it suffices to check that each such pair forms a positively oriented quotient, which is precisely the content of \Cref{thm:necklacequotient}.
\end{proof}

\section{Fan structures for and coherent subdivisions
from $\TrGr_{d;n}^{>0}$ and $\TrFl_n^{>0}$}
\label{sec:BIP}

In this section we make some brief remarks about the various
fan structures for 
 $\TrFl_{\br;n}^{>0}$ and coherent subdivisions from points of  $\TrFl_{\br;n}^{>0}$.
 Codes written for computations here are available at \url{https://github.com/chrisweur/PosTropFlagVar}.
We take a detailed look at the Grassmannian and complete flag variety,
in particular the case of $\TrFl_4^{>0}$.

\subsection{Fan structures}
There are multiple  possibly different natural fan structures  for
$\TrFl_{\br;n}^{>0}$:
\begin{enumerate}[label = (\roman*)]
\item The Pl\"ucker fan (induced by the three-term tropical Pl\"ucker relations).
\item The secondary fan (induced according to the coherent subdivision as in \Cref{cor:permute}).
\item The Gr\"obner fan (induced according to the initial ideal of the ideal 
	$\langle \mathscr P_{\br;n}\rangle$). 
\item The simultaneous refinement of the fans dual to the Newton polytopes
of the Pl\"ucker coordinates,
when the 
Pl\"ucker coordinates are expressed in terms of a 
``positive parameterization'' of $\Fl_{\br;n}^{>0}$,
		such as an \emph{$\mathcal{X}$-cluster chart.}
\item (If the cluster algebra associated to 
	$\Fl_{\br;n}$ has finitely many cluster variables) 
	the same fan as above but with (the larger set of) cluster variables replacing 
Pl\"ucker coordinates.
\end{enumerate}

Note that by definition, fan (v) is always a refinement of (iv).

In the case of the positive tropical Grassmannian, 
the fan structures in (iv) and (v) were studied in 
\cite[Definition 4.2 \& Section 8]{SW05}, where the authors observed that for 
$\Gr_{2,n}$, fan (iv) (which coincides with (v)) is isomorphic to the 
\emph{cluster complex}\footnote{See \cite{ca2} for background 
on the cluster complex.} of type $A_{n-3}$;
 for $\Gr_{3,6}$ and $\Gr_{3,7}$, fan (iv) is isomorphic to a 
coarsening of the corresponding {cluster complex},
 while fan (v) is isomorphic to the cluster complex (of types $D_4$ and $E_6$, respectively).  
\cite[Conjecture 8.1]{SW05} says that fan (v) (associated to the 
positive tropicalization of a full rank cluster variety of finite type)
should be isomorphic to the corresponding cluster complex.
This conjecture was essentially resolved in \cite{JLS21, AHL21}
by working with $F$-polynomials.

\cite[Theorem 14]{Olarte} states that the Pl\"ucker fan and the secondary fan structures for Dressians coincide, and hence implies that (i) and (ii) coincide because the positive Dressian and the positive tropical Grassmannian are the same \cite{SW21}.
For $\TrGr_{2,n}$, the results of \cite[\S4]{SS04} imply that (i), (ii), and (iii) agree, and combining this with \cite[\S5]{SW05} implies that all five fan structures agree for $\TrGr_{2,n}^{>0}$.
For $\TrGr_{3,6}^{>0}$, we computed that (iii) and (v) strictly refine (i), but the two fan structures are not comparable.

\smallskip
We can consider the same fan structures in the case of the 
positive tropical complete
flag variety. 
When $n=3$, the fan $\TrFl_n^{>0}$ modulo its lineality space is a one-dimensional fan, and all fan structures coincide.
For $\TrFl_n$ (before taking the positive part), one can find computations of the fan (iii) for $n = 4$ and $n=5$ in \cite[\S3]{Bossinger}, the fan (i) and its relation to (iii) for $n=4$ in \cite[Example 5.2.3]{BEZ21}, and the fan (ii) and its relation to (iii) for $n=4$ in \cite[\S5]{JLLO}.  Returning to the positive tropicalization, 
\cite[Section 5.1]{Bos22} computed the fan structure (iii) for $\TrFl_4^{>0}$, and found it was dual to the three-dimensional associahedron; in particular,
there are $14$ maximal cones and the $f$-vector is $(14, 21, 9, 1)$.   
Using the positive parameterization of
\cite{Bor} (a graphical version of 
the parameterizations of \cite{MR04}) for $\TrFl_n^{>0}$, we computed the polyhedral complex underlying (iv) for $n=4$ in Macaulay2 by computing the normal fan of the Minkowski sum of the Newton polytopes of the 
Pl\"ucker coordinates expressed in the chosen parametrization; we obtained
 the $f$-vector $(13, 20, 9, 1)$.  We also computed (v) 
 after incorporating the additional non-Pl\"ucker cluster variable 
$p_2 p_{134}-p_1 p_{234}$.
Combining these, we find that for $n=4$, (i)=(iv) and (ii)=(iii).  We also find that both (ii) and (v) strictly refine (i)=(iv) and are both isomorphic to the normal fan of the three-dimensional associahedron, but are not comparable fan structures.

The fact that the fan structure (v) of $\TrFl_4^{>0}$ is dual to the 
three-dimensional associahedron is consistent with 
\cite[Conjecture 8.1]{SW05} and the fact that 
$\Fl_4$ has a cluster algebra structure of 
finite type $A_3$ \cite[Table 1]{GLS}, 
whose cluster complex is dual to the associahedron.

We now give a graphical way to think about the fan structure
on $\TrFl_4^{>0}$, building on the ideas of 
 \cite{SW05} and 
\cite[Example 5.2.3]{BEZ21}.

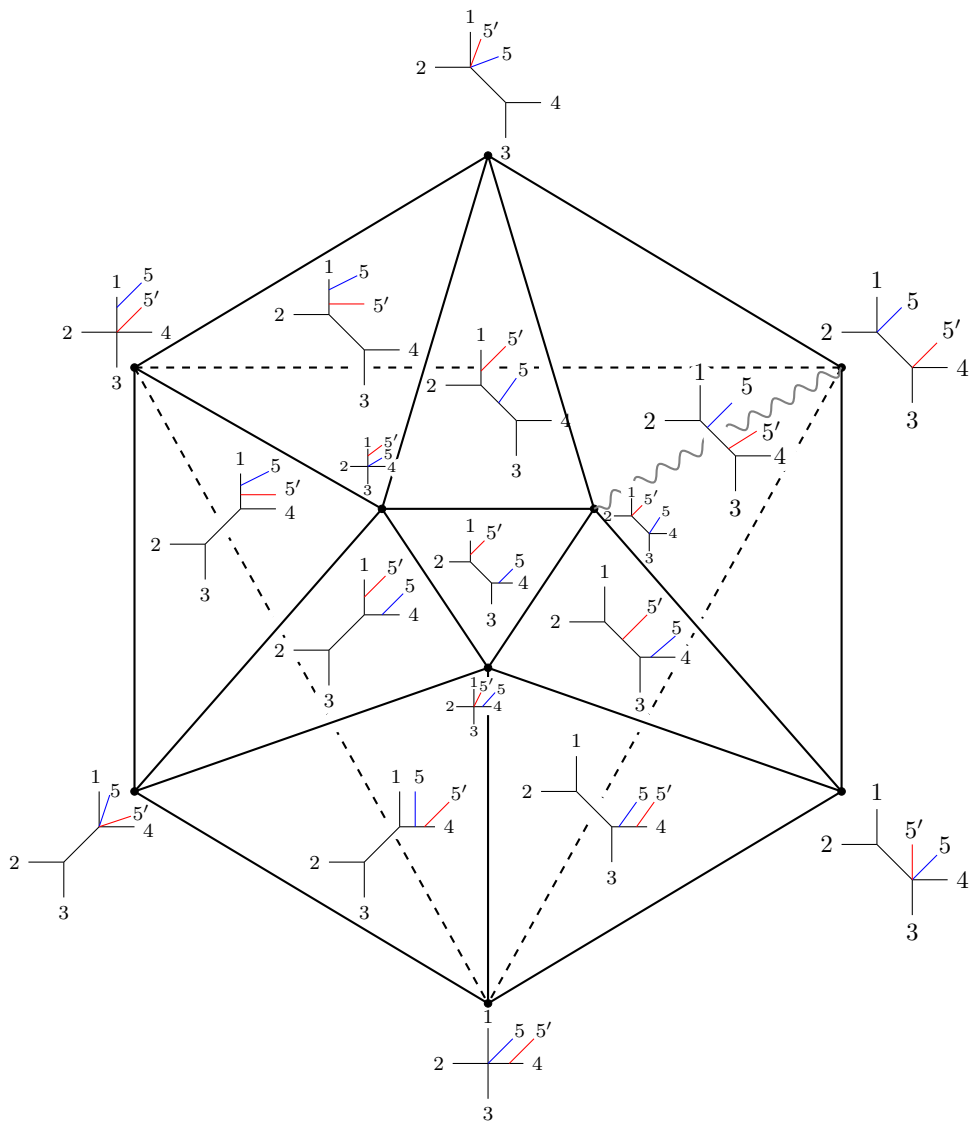
\begin{figure}[h]
\begin{tikzpicture}[scale=0.94, every node/.style={scale=0.94}]
\filldraw[black] (1.5,1) circle (1.5pt) node[anchor=west]{};
\filldraw[black] (-1.5,1) circle (1.5pt) node[anchor=west]{};
\filldraw[black] (0,-1.25) circle (1.5pt) node[anchor=west]{};
\filldraw[black] (5,3) circle (1.5pt) node[anchor=west]{};
\filldraw[black] (-5,3) circle (1.5pt) node[anchor=west]{};
\filldraw[black] (0,6) circle (1.5pt) node[anchor=west]{};
\filldraw[black] (5,-3) circle (1.5pt) node[anchor=west]{};
\filldraw[black] (-5,-3) circle (1.5pt) node[anchor=west]{};
\filldraw[black] (0,-6) circle (1.5pt) node[anchor=west]{};

\draw[black, thick, dashed] (-5,3) -- (0,-6);
\draw[black, thick, dashed] (5,3) -- (0,-6);
\draw[black, thick, dashed] (-5,3) -- (5,3);

\fill[white] (0.9,3) circle(0.15);
\fill[white] (-0.9,3) circle(0.15);
\fill[white] (2.96,-0.65) circle(0.12);
\fill[white] (-2.96,-0.65) circle(0.15);
\fill[white] (2.18,-2.02) circle(0.15);
\fill[white] (-2.18,-2.02) circle(0.15);

\draw[black, thick] (1.5,1) -- (-1.5,1);
\draw[black, thick] (1.5,1) -- (0,-1.25);
\draw[black, thick] (0,-1.25) -- (-1.5,1);
\draw[black, thick] (0,6) -- (5,3);
\draw[black, thick] (5,3) -- (5,-3);
\draw[black, thick] (5,-3) -- (0,-6);
\draw[black, thick] (0,-6) -- (-5,-3);
\draw[black, thick] (-5,-3) -- (-5,3);
\draw[black, thick] (-5,3) -- (0,6);
\draw[black, thick] (0,6) -- (-1.5,1);
\draw[black, thick] (1.5,1) -- (0,6);
\draw[black, thick] (-5,-3) -- (-1.5,1);
\draw[black, thick] (-5,-3) -- (0,-1.25);
\draw[black, thick] (5,-3) -- (0,-1.25);
\draw[black, thick] (1.5,1) -- (5,-3);
\draw[black, thick] (-5,3) -- (-1.5,1);
\draw[black, thick] (0,-6) -- (0,-1.25);

\draw[gray, thick, snake it] (5,3) -- (1.5,1);

\draw[black, thin] (6,3)--(5.5,3.5);
\draw[black, thin] (5.5,3.5)--(5,3.5);
 \draw[black, thin] (5.5,3.5)--(5.5,4);
 \draw[black, thin] (6,3)--(6.5,3);
 \draw[black, thin] (6,3)--(6,2.5);
 \draw[blue, thin] (5.5,3.5)--(5.85, 3.85);
 \draw[red, thin] (6,3)--(6.35,3.35);
 \node[anchor=south] at (5.5,4) {$1$};
 \node[anchor=east] at (5,3.5) {$2$};
 \node[anchor=north] at (6,2.5) {$3$};
 \node[anchor=west] at (6.5,3) {$4$};
 \node[anchor=south west] at (5.8,3.7) {$5$};
 \node[anchor=south west] at (6.3,3.3) {$5'$};
 
 \draw[black, thin] (6,-4.25)--(5.5,-3.75);
\draw[black, thin] (6,-4.25)--(6,-4.75);
\draw[black, thin] (6,-4.25)--(6.5,-4.25);
 \draw[black, thin] (5.5,-3.75)--(5.5,-3.25);
 \draw[black, thin] (5.5,-3.75)--(5,-3.75);
  \draw[blue, thin] (6,-4.25)--(6.35,-3.9);
 \draw[red, thin] (6,-4.25)--(6,-3.75);
 \node[anchor=south] at (5.5,-3.25) {$1$};
 \node[anchor=east] at (5,-3.75) {$2$};
 \node[anchor=north] at (6,-4.75) {$3$};
 \node[anchor=west] at (6.5,-4.25) {$4$};
 \node[anchor=south west] at (6.25,-4) {$5$};
 \node[anchor=south] at (6.03,-3.78) {$5'$};

 \draw[black, thin] (-0.25,0.25)--(0.05,-0.05);
\draw[black, thin] (-0.25,0.25)--(-0.25,0.55);
\draw[black, thin] (-0.25,0.25)--(-0.55,0.25);
 \draw[black, thin] (0.05,-0.05)--(0.05,-0.35);
 \draw[black, thin] (0.05,-0.05)--(0.35,-0.05);
  \draw[blue, thin] (0.15,-0.05)--(0.35,0.15);
 \draw[red, thin] (-0.25,0.35)--(-0.05,0.55);
 \node[anchor=south] at (-0.25,0.55) {\footnotesize$1$};
 \node[anchor=east] at (-0.5,0.25) {\footnotesize$2$};
 \node[anchor=north] at (0.05,-0.35) {\footnotesize$3$};
 \node[anchor=west] at (0.3,-0.05) {\footnotesize$4$};
 \node[anchor=south west] at (0.25,0.05) {\footnotesize$5$};
 \node[anchor=south west] at (-0.15,0.45) {\footnotesize{$5'$}};
 
\draw[black, thin] (0.25,6.75)--(-.25,7.25);
\draw[black, thin] (-0.25,7.25)--(-0.25,7.75);
\draw[black, thin] (-0.25,7.25)--(-0.75,7.25);
 \draw[black, thin] (0.25,6.75)--(0.25,6.25);
 \draw[black, thin] (0.25,6.75)--(0.75,6.75);
  \draw[blue, thin] (-0.25,7.25)--(0.15,7.4);
 \draw[red, thin] (-0.25,7.25)--(-0.1,7.65);
 \node[anchor=south] at (-0.25,7.75) {\footnotesize$1$};
 \node[anchor=east] at (-0.75,7.25) {\footnotesize$2$};
 \node[anchor=north] at (0.25,6.25) {\footnotesize$3$};
 \node[anchor=west] at (0.76,6.75) {\footnotesize$4$};
 \node[anchor=south west] at (0.07,7.23) {\footnotesize$5$};
 \node[anchor=south west] at (-0.2,7.55) {\footnotesize{$5'$}};

  
  \fill[white] (-0.05,3) circle(0.15);
  \fill[white] (0.5,3) circle(0.15);
 \draw[black, thin] (-0.1,2.75)--(0.4,2.25);
\draw[black, thin] (-0.1,2.75)--(-0.1,3.25);
\draw[black, thin] (-0.1,2.75)--(-0.6,2.75);
 \draw[black, thin] (0.4,2.25)--(0.4,1.75);
 \draw[black, thin] (0.4,2.25)--(0.9,2.25);
  \draw[blue, thin] (0.15,2.5)--(0.4,2.85);
 \draw[red, thin] (-0.1,2.95)--(0.25,3.3);
 \node[anchor=south] at (-0.1,3.25) {\footnotesize$1$};
 \node[anchor=east] at (-0.6,2.75) {\footnotesize$2$};
 \node[anchor=north] at (0.4,1.75) {\footnotesize$3$};
 \node[anchor=west] at (0.9,2.25) {\footnotesize$4$};
 \node[anchor=south west] at (0.3,2.75) {\footnotesize$5$};
 \node[anchor=south west] at (0.15,3.2) {\footnotesize{$5'$}};

 \fill[white] (2.03,1.25) circle(0.15);
\draw[black, thin] (2.03,0.9)--(2.28,0.65);
\draw[black, thin] (2.03,0.9)--(2.03, 1.15);
\draw[black, thin] (2.03,0.9)--(1.78,0.9);
 \draw[black, thin] (2.28,0.65)--(2.28,0.4);
 \draw[black, thin] (2.28,0.65)--(2.53,0.65);
  \draw[blue, thin] (2.28,0.65)--(2.43,0.88);
 \draw[red, thin] (2.03,0.9)--(2.18,1.05);
 \node[anchor=south] at (2.03,1.05) {\tiny$1$};
 \node[anchor=east] at (1.88,0.9) {\tiny$2$};
 \node[anchor=north] at (2.28,0.5) {\tiny$3$};
 \node[anchor=west] at (2.43,0.65) {\tiny$4$};
 \node[anchor=south west] at (2.33,0.78) {\tiny$5$};
 \node[anchor=south west] at (2.08,.95) {\tiny{$5'$}};

\fill[white] (3.25,2) circle(0.2);
\fill[white] (4.15,1.75) circle(0.22);
\fill[white] (3,2.85) circle(0.2);
 \draw[black, thin] (3,2.25)--(3.5,1.75);
\draw[black, thin] (3,2.25)--(3,2.75);
 \draw[black, thin] (3,2.25)--(2.5,2.25);
 \draw[black, thin] (3.5,1.75)--(3.5,1.25);
 \draw[black, thin] (3.5,1.75)--(4,1.75);
 \draw[blue, thin] (3.1,2.15)--(3.45, 2.5);
 \draw[red, thin] (3.4,1.85)--(3.8,2.1);
 \node[anchor=south] at (3,2.7) {$1$};
 \node[anchor=east] at (2.5,2.25) {$2$};
 \node[anchor=north] at (3.5,1.25) {$3$};
 \node[anchor=west] at (3.92,1.75) {$4$};
 \node[anchor=south west] at (3.45,2.5) {$5$};
 \node[anchor=south west] at (3.75,1.85) {$5'$};

 \fill[white] (1.5,-3.25) circle(0.2);
\draw[black, thin] (1.25,-3)--(1.75,-3.5);
\draw[black, thin] (1.25,-3)--(1.25,-2.5);
\draw[black, thin] (1.25,-3)--(0.75,-3);
 \draw[black, thin] (1.75,-3.5)--(1.75,-4);
 \draw[black, thin] (1.75,-3.5)--(2.25,-3.5);
  \draw[blue, thin] (1.85,-3.5)--(2.1,-3.15);
 \draw[red, thin] (2.1,-3.5)--(2.35,-3.15);
 \node[anchor=south] at (1.25,-2.5) {\footnotesize$1$};
 \node[anchor=east] at (0.75,-3) {\footnotesize$2$};
 \node[anchor=north] at (1.75,-4) {\footnotesize$3$};
 \node[anchor=west] at (2.25,-3.5) {\footnotesize$4$};
 \node[anchor=south west] at (2,-3.25) {\footnotesize$5$};
 \node[anchor=south west] at (2.25,-3.25) {\footnotesize{$5'$}};

 \fill[white] (2.74,-1.12) circle(0.2);
\draw[black, thin] (1.65,-0.6)--(2.15,-1.1);
\draw[black, thin] (1.65,-0.6)--(1.65,-0.1);
\draw[black, thin] (1.65,-0.6)--(1.15,-0.6);
 \draw[black, thin] (2.15,-1.1)--(2.15,-1.6);
 \draw[black, thin] (2.15,-1.1)--(2.65,-1.1);
  \draw[blue, thin] (2.3,-1.1)--(2.65,-0.8);
 \draw[red, thin] (1.9,-0.85)--(2.25,-0.5);
 \node[anchor=south] at (1.65,-.1) {\footnotesize$1$};
 \node[anchor=east] at (1.15,-0.6) {\footnotesize$2$};
 \node[anchor=north] at (2.15,-1.55) {\footnotesize$3$};
 \node[anchor=west] at (2.6,-1.1) {\footnotesize$4$};
 \node[anchor=south west] at (2.47,-0.9) {\footnotesize$5$};
 \node[anchor=south west] at (2.12,-0.6) {\footnotesize{$5'$}};

 \fill[white] (-1.8,3) circle(0.18);
\draw[black, thin] (-2.25,3.75)--(-1.75,3.25);
\draw[black, thin] (-2.25,3.75)--(-2.25,4.25);
\draw[black, thin] (-2.25,3.75)--(-2.75,3.75);
 \draw[black, thin] (-1.75,3.25)--(-1.75,2.75);
 \draw[black, thin] (-1.75,3.25)--(-1.25,3.25);
  \draw[blue, thin] (-2.25,4.1)--(-1.85,4.3);
 \draw[red, thin] (-2.25,3.9)--(-1.75,3.9);
 \node[anchor=south] at (-2.25,4.22) {\footnotesize$1$};
 \node[anchor=east] at (-2.75,3.75) {\footnotesize$2$};
 \node[anchor=north] at (-1.75,2.75) {\footnotesize$3$};
 \node[anchor=west] at (-1.25, 3.25) {\footnotesize$4$};
 \node[anchor=south west] at (-1.95,4.14) {\footnotesize$5$};
 \node[anchor=west] at (-1.75,3.9) {\footnotesize{$5'$}};
 
 \fill[white] (-3.75,0.75) circle(0.18);
\draw[black, thin] (-3.5,1)--(-4,0.5);
\draw[black, thin] (-3.5,1)--(-3.5,1.5);
\draw[black, thin] (-4,0.5)--(-4.5,0.5);
 \draw[black, thin] (-4,.5)--(-4,0);
 \draw[black, thin] (-3.5,1)--(-3,1);
  \draw[blue, thin] (-3.5,1.33)--(-3.1,1.53);
 \draw[red, thin] (-3.5,1.2)--(-3,1.2);
 \node[anchor=south] at (-3.5,1.5) {\footnotesize$1$};
 \node[anchor=east] at (-4.5,0.5) {\footnotesize$2$};
 \node[anchor=north] at (-4,0) {\footnotesize$3$};
 \node[anchor=west] at (-3,0.95) {\footnotesize$4$};
 \node[anchor=south west] at (-3.2,1.4) {\footnotesize$5$};
 \node[anchor=west] at (-3,1.3) {\footnotesize{$5'$}};

 \fill[white] (-2.75,-1) circle(0.18);
  \fill[white] (-2.3,-1.7) circle(0.18);
\draw[black, thin] (-1.75,-0.5)--(-2.25,-1);
\draw[black, thin] (-1.75,-0.5)--(-1.75,0);
\draw[black, thin] (-2.25,-1)--(-2.75,-1);
 \draw[black, thin] (-2.25,-1)--(-2.25,-1.5);
 \draw[black, thin] (-1.75,-0.5)--(-1.25,-0.5);
  \draw[blue, thin] (-1.5,-0.5)--(-1.2,-0.2);
 \draw[red, thin] (-1.75,-0.25)--(-1.45,0.05);
 \node[anchor=south] at (-1.75,0) {\footnotesize$1$};
 \node[anchor=east] at (-2.75,-1) {\footnotesize$2$};
 \node[anchor=north] at (-2.25,-1.5) {\footnotesize$3$};
 \node[anchor=west] at (-1.25,-0.5) {\footnotesize$4$};
 \node[anchor=south west] at (-1.3,-0.3) {\footnotesize$5$};
 \node[anchor= south west] at (-1.55,-0.05) {\footnotesize{$5'$}};
 
 \fill[white] (-1.25,-3.5) circle(0.18);
 
\draw[black, thin] (-1.25,-3.5)--(-1.75,-4);
\draw[black, thin] (-1.25,-3.5)--(-1.25,-3);
\draw[black, thin] (-1.75,-4)--(-2.25,-4);
 \draw[black, thin] (-1.75,-4)--(-1.75,-4.5);
 \draw[black, thin] (-1.25,-3.5)--(-0.75,-3.5);
  \draw[blue, thin] (-1.03,-3.5)--(-1.03,-3);
 \draw[red, thin] (-0.9,-3.5)--(-0.55,-3.15);
 \node[anchor=south] at (-1.3,-3) {\footnotesize$1$};
 \node[anchor=east] at (-2.25,-4) {\footnotesize$2$};
 \node[anchor=north] at (-1.75,-4.5) {\footnotesize$3$};
 \node[anchor=west] at (-0.75,-3.5) {\footnotesize$4$};
 \node[anchor=south] at (-0.97,-3) {\footnotesize$5$};
 \node[anchor= south west] at (-0.65,-3.25) {\footnotesize{$5'$}};
 

\draw[black, thin] (-5.5,-3.5)--(-6,-4);
\draw[black, thin] (-5.5,-3.5)--(-5.5,-3);
\draw[black, thin] (-6,-4)--(-6.5,-4);
 \draw[black, thin] (-6,-4)--(-6,-4.5);
 \draw[black, thin] (-5.5,-3.5)--(-5,-3.5);
  \draw[blue, thin] (-5.5,-3.5)--(-5.35,-3.05);
 \draw[red, thin] (-5.5,-3.5)--(-5.05,-3.35);
 \node[anchor=south] at (-5.55,-3) {\footnotesize$1$};
 \node[anchor=east] at (-6.5,-4) {\footnotesize$2$};
 \node[anchor=north] at (-6,-4.5) {\footnotesize$3$};
 \node[anchor=west] at (-5,-3.55) {\footnotesize$4$};
 \node[anchor=south west] at (-5.47,-3.2) {\footnotesize$5$};
 \node[anchor= south west] at (-5.15,-3.53) {\footnotesize{$5'$}};
 

\draw[black, thin] (0,-6.85)--(0,-6.35);
\draw[black, thin] (0,-6.85)--(-0.5,-6.85);
 \draw[black, thin] (0,-6.85)--(0,-7.35);
 \draw[black, thin] (0,-6.85)--(0.5,-6.85);
  \draw[blue, thin] (0,-6.85)--(0.35,-6.5);
 \draw[red, thin] (0.3,-6.85)--(0.65,-6.5);
 \node[anchor=south] at (0,-6.4) {\footnotesize$1$};
 \node[anchor=east] at (-0.5,-6.85) {\footnotesize$2$};
 \node[anchor=north] at (0,-7.35) {\footnotesize$3$};
 \node[anchor=west] at (0.5,-6.85) {\footnotesize$4$};
 \node[anchor=south west] at (0.25,-6.6) {\footnotesize$5$};
 \node[anchor= south west] at (0.55,-6.6) {\footnotesize{$5'$}};

\draw[black, thin] (-5.25,3.5)--(-5.25,4);
\draw[black, thin] (-5.25,3.5)--(-5.75,3.5);
 \draw[black, thin] (-5.25,3.5)--(-5.25,3);
 \draw[black, thin] (-5.25,3.5)--(-4.75,3.5);
  \draw[blue, thin] (-5.25,3.85)--(-4.9,4.2);
 \draw[red, thin] (-5.25,3.5)--(-4.9,3.85);
 \node[anchor=south] at (-5.25,4) {\footnotesize$1$};
 \node[anchor=east] at (-5.75,3.5) {\footnotesize$2$};
 \node[anchor=north] at (-5.25,3) {\footnotesize$3$};
 \node[anchor=west] at (-4.75,3.5) {\footnotesize$4$};
 \node[anchor=south west] at (-5,4.1) {\footnotesize$5$};
 \node[anchor= south west] at (-5,3.75) {\footnotesize{$5'$}};

 
\fill[white] (-1.25,1.56) circle(0.15); 
\draw[black, thin] (-1.7,1.6)--(-1.7,1.85);
\draw[black, thin] (-1.7,1.6)--(-1.95,1.6);
 \draw[black, thin] (-1.7,1.6)--(-1.7,1.35);
 \draw[black, thin] (-1.7,1.6)--(-1.45,1.6);
  \draw[blue, thin] (-1.7,1.6)--(-1.5,1.72);
 \draw[red, thin] (-1.7,1.75)--(-1.5,1.9);
 \node[anchor=south] at (-1.7,1.75) {\tiny$1$};
 \node[anchor=east] at (-1.85,1.6) {\tiny$2$};
 \node[anchor=north] at (-1.7,1.45) {\tiny$3$};
 \node[anchor=west] at (-1.55,1.59) {\tiny$4$};
 \node[anchor=south west] at (-1.6,1.58) {\tiny$5$};
 \node[anchor= south west] at (-1.6,1.75) {\tiny{$5'$}};
 
 
  \fill[white] (0,-1.65) circle(0.27);
\draw[black, thin] (-0.2,-1.8)--(-0.2,-1.55);
\draw[black, thin] (-0.2,-1.8)--(-0.45,-1.8);
 \draw[black, thin] (-0.2,-1.8)--(-0.2,-2.05);
 \draw[black, thin] (-0.2,-1.8)--(0.05,-1.8);
  \draw[blue, thin] (-0.08,-1.8)--(0.1,-1.6);
 \draw[red, thin] (-0.2,-1.8)--(-0.1,-1.6);
 \node[anchor=south] at (-0.2,-1.65) {\tiny$1$};
 \node[anchor=east] at (-0.35,-1.8) {\tiny$2$};
 \node[anchor=north] at (-0.2,-1.95) {\tiny$3$};
 \node[anchor=west] at (-0.05,-1.8) {\tiny$4$};
 \node[anchor=south west] at (0, -1.75) {\tiny$5$};
 \node[anchor= south west] at (-0.25,-1.7) {\tiny{$5'$}};
 \end{tikzpicture}
	\caption{The fan structure (ii)=(iii) of $\TrFl_4^{>0}$.}
        \label{fig:TrFl4}
\end{figure}

\begin{eg}
A \emph{planar tree} on $[n]$ is an unrooted tree drawn
in the plane with 
$n$ leaves labeled by $1,2,\dots,n$ (in counterclockwise 
order).  By \cite{SW05}, 
$\TrGr_{2;n}^{>0}$ parameterizes metric planar trees, and its cones correspond
to the various combinatorial types of planar trees.
In particular, 
 if we assign real-valued lengths to the edges of a planar tree, 
then the negative of the distance 
  between leaf $i$ and $j$ encodes
the positive tropical Pl\"ucker coordinate $w_{ij}$ of a point
	in the corresponding  cone.
	In particular, 
	it is easy to see that the negative distances $w_{ij}$ 
	associated to such a planar tree satisfy the positive tropical Pl\"ucker 
	relations.

Now as in 
\cite[Example 5.2.3]{BEZ21}, 
we note that for a valuated matroid $\mu$ whose 
underlying matroid is the uniform matroid $U_{2,4}$,
the tropical linear spaces $\trop(\mu)$
and $\trop(\mu^*)$
associated to $\mu$ and its dual 
$\mu^*$ are translates of each other. 
This allows us to identify points 
 $\bmu = (\mu_1,\mu_2,\mu_3)$ of 
$\TrFl_4^{>0}$ with planar trees on 
	the vertices $\{1,2,3,4,5,5'\}$
	such that the vertices
	$\{1,2,3,4,5\}$ and separately the vertices
	$\{1,2,3,4,5'\}$ appear in counterclockwise
	order.
	To see this, note that (using the same idea as 
	Construction \#1 from \Cref{sec:projection})
	we can identify $(\mu_1,\mu_2)$, with Pl\"ucker coordinates
	$(w_1,\dots,w_4; w_{12},\dots,w_{34})$,
with an element $(w_{ab})$ of $\TrGr_{2,5}^{>0}$:
	we simply set $w_{a5}:=w_a$ for $1\leq a \leq 4$.
   Similarly,
we identify $(\mu_2,\mu_3)$, where $\mu_3$ has Pl\"ucker coordinates
$(w_{123},\dots,w_{234})$,
with an 
element of $\TrGr_{2,5}^{>0}$:
	we simply set $w_{d5'}:=w_{abc}$, where $\{a,b,c\}:=[4]\setminus \{d\}$.

This gives us the Pl\"ucker fan structure (i)=(iv) 
with thirteen maximal cones, as shown in 
	\Cref{fig:TrFl4}.  
	To get the Gr\"obner fan structure (iii) we subdivide one of the 
cones into two, along the squiggly line shown in 
	\Cref{fig:TrFl4}.  
This squiggly line occurs when
$\distance(x_1,blue) = \distance(x_2,red)$,
where $x_1$ and $x_2$ are the two black trivalent nodes in the tree on $[4]$.
To obtain the fan structure (v), instead of the squiggly line, the square face is subdivided along the other diagonal.
\end{eg}

Using the computation of $\TrFl_5$ in \cite{Bossinger}, available at \url{https://github.com/Saralamboglia/Toric-Degenerations/blob/master/Flag5.rtf} and \Cref{cor:3terms}, we further computed that $\TrFl_5^+$ with (iii) has 938 maximal cones (906 of which are simplicial) and that (iv) has 406 maximal cones.
According to \cite[Conjecture 8.1]{SW05}, the (v) fan structure for $\TrFl_5^+$ has 672 maximal cones.

\subsection{Coherent subdivisions}

We next discuss coherent subdivisions coming from the positive
tropical Grassmannian and positive tropical complete flag variety.
When $\Fl_{\br;n}$ is the Grassmannian $\Gr_{d,n}$ 
and the support $\underline{\bmu}$
is the uniform matroid, 
 \Cref{thm:eqvs} gives rise to the following corollary
 (which  was first proved
  in \cite{LPW} and \cite{AHLS}).
 \begin{cor}\label{cor:hyper}
Let $\bmu  = (\mu_d) \in  \PP\left( \TT^{\binom{[n]}{d}}\right)$, 
 and suppose it has no $\infty$ coordinates.  Then the following statements are equivalent.
 \begin{itemize}
\item $\bmu\in \TrGr_{d,n}^{> 0}$, that is, $\bmu$ lies in the strictly positive
	tropical Grassmannian.
\item
		Every face in the coherent subdivision $\mathcal 
		 D_{\boldsymbol\mu}$ of the hypersimplex $\Delta_{d,n}$
 induced by $\bmu$ is a positroid polytope.
 \end{itemize}
 \end{cor}

The coherent subdivisions above (called \emph{positroidal subdivisions})
were further studied in \cite{SW21}, 
where the finest positroidal subdivisions were characterized in terms of 
 series-parallel matroids.  Furthermore, all finest
 positroidal subdivisions of $\Delta_{d,n}$ achieve equality in Speyer's 
 \emph{$f$-vector theorem}; in particular, they all 
 consist of ${n-2 \choose d-1}$
 facets
	\cite[Corollary 6.7]{SW21}.

When  $\Fl_{\br;n}$ is the complete flag variety $\Fl_n$,
and the support $\underline{\bmu}$
is the uniform flag matroid, 
 \Cref{thm:eqvs} gives rise to the following corollary,
 which appeared 
 in \cite[Theorem 20]{JLLO}.

\begin{cor}\label{cor:permute}
Let $\bmu  = (\mu_1,\ldots, \mu_n) \in \prod_{i = a}^{b} \PP\left( \TT^{\binom{[n]}{i}}\right)$, 
 and suppose it has no $\infty$ coordinates.  Then the following statements are equivalent.
 \begin{itemize}
\item $\bmu\in \TrFl_{n}^{> 0}$, that is, $\bmu$ lies in the strictly positive
	tropical flag variety.
\item
		Every face in the coherent subdivision $\mathcal 
		D_{\boldsymbol\mu}$ of the permutohedron $\Perm_n$
 induced by $\bmu$ is a Bruhat interval polytope.
 \end{itemize}
 \end{cor}

In light of the results of \cite{SW21}, 
it is natural to ask if one can characterize the 
finest coherent subdivisions of the permutohedron $\Perm_n$ into Bruhat interval polytopes.  
Furthermore, do they all have the same $f$-vector?

Explicit computations for $\TrFl_4$ show that the answer to the 
second question is no.  
We find that $\TrFl_4$ with the fan structure (iii) 
(which agrees with (ii) by \cite[\S5]{JLLO}) has $78$ maximal cones.
We choose a 
point in the relative interior of each of the $78$ cones to 
use as a height function
(thinking of points in $\TrFl_4$ as weights on the vertices of $\Perm_4$ as in \Cref{eqvs:subdiv} of \Cref{thm:eqvs}),
 then 
use Sage to compute the corresponding coherent subdivision of 
$\Perm_4$.
As expected, precisely $14$ of the $78$ cones induce subdivisions of $\Perm_4$ 
into Bruhat interval polytopes,  see
  \Cref{tab:my_label}.

\begin{table}\label{tab:subdivisions}
    \centering
    \begin{tabular}{|c|c|c|}
         \hline
         
         \centered{Height function
	    $(P_1,P_2,P_3, P_4$; $P_{12},P_{13},$\\  $P_{14},P_{23}, P_{24},P_{34}$; $P_{123}, P_{124},P_{134}, P_{234})$}& \centered{Bruhat interval polytopes\\ in subdivision}   & \centered{$f$-vector}\\
        \hline
        \centered{$(15,-1,-7,-7;4,-2,-2,$\\$-2,-2,4;-7,-7,-1,15)$} &\centered{$P_{3214,4321}, P_{3124,4231},P_{2314,3421},$\\$P_{2134,3241},P_{1324,2431},P_{1234,2341}$} &\multirow{23}{*}{$(24,46,29,6)$}\\
        \cline{1-2}
        \centered{$(15,3,-9,-9;4,-8,-8,$\\$-4,-4,20;-1,-1,-1,3)$}& \centered{$P_{2413,4321}, P_{3124,4231},P_{2314,4231},$\\$P_{2134,3241},P_{1324,2431},P_{1234,2341}$}&\\
        \cline{1-2} 
        \centered{$(15,-7,-1,-7;-2,4,-2,$\\$-2,4,-2;-7,-1,-7,15)$} &\centered{$P_{3142,4321},P_{3124,4312},P_{2143,3421},$\\$P_{2134,3412},P_{1243,2431},P_{1234,2413}$} & \\
        \cline{1-2}
        \centered{$(-1,-1,-1,3;4,-8,-4,$\\$-8,-4,20;15, 3,-9,-9)$} &\centered{$P_{2413,4321}, P_{1423,4231},P_{1342,4231},$\\ $P_{1324,4213},P_{1243,4132},P_{1234,4123}$} &\\
        \cline{1-2}
        \centered{$(-7,-7,-1,15;4,-2,-2,$\\$-2,-2,4;15,-1,,-7,-7)$} &\centered{$P_{1432,4321},P_{1423,4312}, P_{1342,4231},$\\ $P_{1324,4213},P_{1243,4132},P_{1234,4123}$} &\\
        \cline{1-2}
        \centered{$(-1,-7,-7,15;-2,-2,4,$\\$4,-2,-2;15,-7,-7,-1)$} & \centered{$P_{3142,4321},P_{2143,4312},P_{2134,4213},$\\$P_{1342,3421},P_{1243,3412},P_{1234,2413}$}&\\
        \cline{1-2}
        \centered{$(-9,-9,3,15;20,-4,-8,$\\$-4,-8,4;3,-1,-1,-1)$ }& \centered{$P_{1432,4321}, P_{1423,4312},P_{1342,4231},$\\$P_{1324,4213},P_{1324,4132},P_{1234,3142}$}&\\
        \cline{1-2}
        \centered{$(11,-7,-7,3;-6,-6,4,$\\$4,2,2;11,-7,-7,3)$} &\centered{$P_{3142,4321},P_{2143,4312},P_{2134,4213},$\\ $P_{2143,3421},P_{1243,2431},P_{1234,2413}$} &\\
        \cline{1-2}
        \centered{$(3,3,-3,-3;20,-10,-10,$\\$-10,-10,20;-3,-3,3,3) $}& \centered{$P_{2413,4321},P_{3124,4231},P_{2314,4231},$\\ $P_{1324,2431},P_{1324,3241},P_{1234,3142}$}&\\
        \cline{1-2}
        \centered{$(3,-1,-1,-1;20,-4,-4,$\\$-8,-8,4;-9,-9,3,15)$} &\centered{$P_{3214,4321},P_{3124,4231},P_{2314,3421},$\\ $P_{1324,3241},P_{1324,2431},P_{1234,3142}$} &\\
        \cline{1-2}
        \centered{$(-3,-3,3,3;20,-10,-10,$\\$-10,-10,20;3,3,-3,-3)$} &\centered{$P_{2413,4321}, P_{1423,4231}, P_{1342,4231},$\\ $P_{1324,4132},P_{1324,4213}, P_{1234,3142}$} &\\
        \cline{1-2}
        \centered{$(3,-7,-7,11;2,2,4,$\\$4,-6,-6;3,-7,-7,11)$} &\centered{$P_{3142,4321}, P_{3124,4312}, P_{1342,3421},$\\ $P_{2134,3412}, P_{1243,3412},P_{1234,2413}$} & \\
        \hline
        \centered{$(11,-1,-7,-3;-2,-8,-4,$\\$-4,0,18;11,-1,-7,-3)$} &\centered{$P_{2413,4321}, P_{2143,4231}, P_{2134,4213},$\\ $P_{1243,2431}, P_{1234,2413}$} &\multirow{3}{*}{$(24,45,27,5)$}\\
        \cline{1-2}
        \centered{$(-3,-7,-1,11;18,0,-4,$\\$-4,-8,-2;-3,-7,-1,11)$} &\centered{$P_{3142,4321}, P_{3124,4312}, P_{1342,3421}$\\ $P_{1324,3412}, P_{1234,3142}$} &\\
        \hline
        \end{tabular}
    \caption{Table documenting the $14$ finest coherent subdivisions of $\Perm_4$ into Bruhat interval polytopes. There are two possible $f$-vectors, each of which can be realized in multiple ways.}
    \label{tab:my_label}
\end{table}

Of the $14$ coherent subdivisions coming from maximal 
cones of $\TrFl_4^{>0}$, $12$ of them 
contain $6$ facets, while the other $2$ contain $5$ facets. 
  \Cref{tab:my_label}
  lists the facets 
 and $f$-vectors of each of these $14$ subdivisions.
 Note that each Bruhat interval polytope $P_{v,w}$ which appears
 as a facet satisfies $\ell(w)-\ell(v)=3$. 
 Thus, any Bruhat interval polytope $P_{v',w}$ properly contained
 inside $P_{v,w}$ would have the property that 
 $\ell(w')-\ell(v')\leq 2$, and hence $\dim(P_{v',w'})\leq 2$. 
 Since $\Perm_4$ is $3$-dimensional, 
 all $14$ of these subdivisions are finest subdivisions.

We note that the $12$ finest subdivisions whose $f$-vector
is $(24, 46, 29, 6)$ are subdivisions of the permutohedron 
into cubes.  Subdivisions of the permutohedron into Bruhat 
interval polytopes which are cubes
have been previously studied in \cite[Sections 5 and 6]{Harada}
 \cite{Lee}, and in \cite[Section 6]{NT}.  
In particular, there is a subdivision of $\Perm_n$ into $(n-1)!$
Bruhat interval polytopes
$$\{P_{u,v} \ \vert \ u=(u_1\dots,u_n) \text { with }u_n=n, \text{ and }
                      v=(v_1,\dots,v_n) \text { with }v_i = u_i+1 \text{ modulo }n\}.$$
The first subdivision in 
    \Cref{tab:my_label} has this form.

We can further study the $f$-vectors of subdivisions of $\TrFl_4^{>0}$ which are coarsest (without being trivial), rather than finest. In this case, we observe three different $f$-vectors, each of which occurs in multiple subdivisions. The detailed results of our explicit computations on coarsest subdivisions can be found in \Cref{tab:coarsest}.

\begin{table}
    \centering
    \begin{tabular}{|c|c|c|}
         \hline
         
         \centered{Height function $(P_1,P_2,P_3, P_4; P_{12},P_{13},$ \\ $P_{14},P_{23}, P_{24},P_{34}; P_{123}, P_{124},P_{134}, P_{234})$}& \centered{Bruhat interval polytopes\\ in subdivision}   & \centered{$f$-vector}\\
        \hline
        \centered{$(-1,-1,-1,0;-1,-1,0,-1,0,0;0,0,0,0)$} &\centered{$P_{1243,4321}, P_{1234,4213}$} &\multirow{4}{*}{$(24,39,18,2)$}\\
        \cline{1-2}
        \centered{$(-1,-1,-1,0;0,0,0,0,0,0;0,0,0,0)$}& \centered{$P_{1342,4321}, P_{1234,4312}$}&\\
        \cline{1-2} 
        \centered{$(1,0,0,0;0,0,0,0,0,0;0,0,0,0)$} &\centered{$P_{2134,4321},P_{1234,2431}$} & \\
        \cline{1-2}
        \centered{$(1,0,0,0;0,0,0,1,1,1;0,0,0,0)$} &\centered{$P_{3124,4321}, P_{1234,3421}$} &\\ \hline

        \centered{$(0,0,0,0;-1,-1,-1,-1,-1,0;0,0,0,0)$} &\centered{$P_{2413,4321},P_{1234,4231}$} &\multirow{2}{*}{$(24,40,19,2)$}\\
        \cline{1-2}
        \centered{$(0,0,0,0;1,0,0,0,0,0;0,0,0,0)$} & \centered{$P_{1324,4321},P_{1234,3142}$}&\\
        \hline
        
        \centered{$(-1,-1,0,0;-1,-1,-1,-1,-1,0;0,0,0,0)$ }& \centered{$P_{1423,4321}, P_{1342,4231},$\\$P_{1324,4213},P_{1234,4132}$}&\multirow{5}{*}{$(24,42,23,4)$}\\
        \cline{1-2}
        \centered{$(0,-1,-1,0;0,0,1,0,0,0;0,0,0,0)$} &\centered{$P_{3142,4321},P_{1243,3421},$\\ $P_{2134,4312},P_{1234,2413}$} &\\
        \cline{1-2}
        \centered{$(1,1,0,0;1,0,0,0,0,0;0,0,0,0) $}& \centered{$P_{2314,4321},P_{1324,2431},$\\ $P_{3124,4231},P_{1234,3241}$}&\\
        \hline
        \end{tabular}
    \caption{Table documenting the $9$ coarsest coherent subdivisions of $\Perm_4$ into Bruhat interval polytopes. There are three possible $f$-vectors, each of which can be realized in multiple ways.}
    \label{tab:coarsest}
\end{table}

\small
\bibliography{BEW_bib_pruned}
\bibliographystyle{alpha}

\end{document}